\documentclass[11pt, 
]{amsart}

\newcommand{\hide}[1]{}

\usepackage{amssymb}
\usepackage{graphicx}

\usepackage{geometry}
\usepackage[colorlinks=true,urlcolor=blue]{hyperref}
\usepackage[mathscr]{euscript}

\def\textcolor#1{}

\newcommand{\N}{\mathbb{N}}
\newcommand{\R}{\mathbb{R}}
\newcommand{\C}{\mathbb{C}}

\newcommand{\disk}{\mathbb{D}}
\newcommand{\Cc}{\widehat{{\C}}}

\newcommand{\inter}{\mathring}
\newcommand{\Sspace}{\mathscr{S}}

\DeclareMathOperator{\Dom}{\mathcal R} 
\DeclareMathOperator{\DomE}{\mathcal L}		
\DeclareMathOperator{\DomL}{\hat{\mathcal L}}	

\DeclareMathOperator{\PC}{PC}

\newcommand{\U}{\mathcal U}
\newcommand{\V}{\mathcal V}
\newcommand{\X}{\mathcal X}
\newcommand{\W}{\mathcal W}
\newcommand{\LL}{\ell}

\DeclareMathOperator{\modulus}{mod}

\renewcommand{\tilde}{\widetilde}
\renewcommand{\rho}{\varrho}
\renewcommand{\phi}{\varphi}
\newcommand{\eps}{\varepsilon}
\renewcommand{\theta}{\vartheta}

\DeclareMathOperator{\id}{id}
\DeclareMathOperator{\fib}{fib}
\DeclareMathOperator{\cfib}{cf}

\DeclareMathOperator{\Crit}{Crit}
\DeclareMathOperator{\orb}{orb}
\DeclareMathOperator{\diam}{diam}

\DeclareMathOperator{\area}{meas}

\newcommand{\CritNF}{\Crit_\textit{nf}}

\newcommand{\Kwi}{K_{\text{well-inside}}}
\newcommand{\Jwi}{J_{\text{well-inside}}}

\newcommand{\Koc}{K_{\text{off-crit}}}
\newcommand{\Joc}{J_{\text{off-crit}}}

\newcommand\ovl[1]{\overline{#1}}
\newcommand{\sm}{\setminus}

\renewcommand{\ge}{\geqslant}
\renewcommand{\le}{\leqslant}

\theoremstyle{theorem}
\newtheorem{theorem}{Theorem}[section]
\newtheorem{lemma}[theorem]{Lemma}

\newtheorem{proposition}[theorem]{Proposition}
\newtheorem{corollary}[theorem]{Corollary}

\newtheorem{MainTheorem}{Theorem}

\newtheorem*{RatRigidityPrinciple}{Rational Rigidity Principle}

\theoremstyle{definition}
\newtheorem{definition}[theorem]{Definition}

\theoremstyle{remark}
\newtheorem*{remark}{\textsc{Remark}}

\newcounter{reminder} \newcounter{reminderCold}

\usepackage{tikz-cd}

\newtheoremstyle{claim}
  {}
  {}
  {\itshape}
  {0pt}
  {\scshape}
  {.}
  { }
  {\thmname{#1}\thmnumber{ #2}\thmnote{ (#3)}}

\theoremstyle{claim}

\numberwithin{equation}{section}

\newcounter{stepctr}

\newenvironment{step}[1][]
{
	\smallskip
	\refstepcounter{stepctr}
	\ifx\hfuzz#1\hfuzz
  		\textit{Step \thestepctr. }\ignorespaces
 	\else
  		\textit{Step \thestepctr: #1. }\ignorespaces
	\fi
}{}  

\newcommand{\Newpage}{}

\title{Rigidity of Newton dynamics}
\subjclass[2000]{30D05, 37F10, 37F20}
\author[Kostiantyn Drach]{Kostiantyn Drach}
\author[Dierk Schleicher]{Dierk Schleicher}
\address{Aix--Marseille Universit\'e, Institut de Math\'ematiques de Marseille, 163 Avenue de Luminy, 13009 Marseille, France}
\email{drach.kostiantyn@univ-amu.fr}
\email{dierk.schleicher@univ-amu.fr}

\thanks{\textbf{Acknowledgements.}
We are grateful to a number of colleagues for helpful and inspiring discussions during the time when we worked on this projects, in particular Dima Dudko, Misha Hlushchanka, John Hubbard, Misha Lyubich, Oleg Kozlovski, and Sebastian van Strien. 
Finally, we would like to thank our dynamics research group for numerous helpful and enjoyable discussions: Konstantin Bogdanov, Roman Chernov, Russell Lodge, Steffen Maa{\ss}, David Pfrang, Bernhard Reinke, Sergey Shemyakov, and Maik Sowinski.
We gratefully acknowledge support by the Advanced Grant ``HOLOGRAM'' of the European Research Council, as well as hospitality of Cornell University in the spring of 2018 while much of this work was prepared.
}

\begin{document}

\begin{abstract} 
We study rigidity of rational maps that come from Newton's root finding method for polynomials of arbitrary degrees. We establish dynamical rigidity of these maps: each point in the Julia set of a Newton map is either rigid (i.e.\ its orbit can be distinguished in combinatorial terms from all other orbits), or the orbit of this point eventually lands in the filled-in Julia set of a polynomial-like restriction of the original map. As a corollary, we show that the Julia sets of Newton maps in many non-trivial cases are locally connected; in particular, every cubic Newton map without Siegel points has locally connected Julia set.

In the parameter space of Newton maps of arbitrary degree we obtain the following rigidity result: any two combinatorially equivalent Newton maps are quasiconformally conjugate in a neighborhood of their Julia sets provided that they either non-renormalizable, or they are both renormalizable ``in the same way''. 

Our main tool is the concept of complex box mappings due to Kozlovski, Shen, van Strien; we also extend a dynamical rigidity result for such mappings so as to include irrationally indifferent or renormalizable situations.   
\end{abstract}

\maketitle

\section{Introduction and main results}

\subsection{Local connectivity, topological models, and rigidity}
\label{SSec:RRPPhil}

We investigate the fine structure in the dynamical systems formed by iteration of Newton maps of polynomials: the goal is to show that any two points within any given dynamical system can be distinguished in combinatorial terms (``dynamical rigidity''), and similarly that any two Newton dynamical systems can be distinguished combinatorially as well (``parameter rigidity''). Analogous rigidity results are known to be false for polynomial dynamics, and our main result is that they hold for the Newton dynamics everywhere except when embedded polynomial dynamics interferes (both in the dynamical plane and in parameter space). These results are strongest possible: embedded non-rigidity of polynomial dynamics makes rigidity in the Newton dynamics impossible.

This research connects to and builds upon a deep body of research on polynomial dynamics, initiated by Douady and Hubbard in their seminal Orsay Notes \cite{Orsay} and extended in celebrated work by Yoccoz \cite{HY}, Lyubich and coauthors (see e.g.\ \cite{KLUnicr, Lyu,Pacman2}), Kozlovski--van Strien \cite{KvS09}, and numerous others. The goal in much of this work is often phrased as showing that polynomial Julia sets are locally connected (many of these are, but not all; see for instance Milnor~\cite{MiLocConn}). The importance of local connectivity of Julia sets comes from several closely connected aspects: if a Julia set is locally connected, then
\begin{itemize}
\item
it has a simple and satisfactory topological model, for instance in terms of Thurston laminations \cite{ThurstonLaminations,ThurstonLaminationsDS} or Douady's pinched disks \cite{DouadyPinchedDisks};
\item
any two points in the Julia set can be distinguished in terms of symbolic dynamics, for instance in the complement of pairs of dynamic rays that land at common periodic or preperiodic points.
\end{itemize}
For instance, Yoccoz' theorem on quadratic polynomials can be phrased as saying that all quadratic polynomials that are non-renormalizable and for which both fixed points are repelling have locally connected Julia sets, or equivalently that any two points in the Julia set can be distinguished in terms of their itineraries with respect to the unique fixed point that disconnects the Julia set (usually called the $\alpha$ fixed point).

Meanwhile, it is known that the Julia set of a polynomial of any degree is locally connected, and all its points can be distinguished in terms of symbolic dynamics, when the dynamics is not infinitely renormalizable and no periodic points are irrationally indifferent \cite{KvS09}. In many cases with irrationally indifferent periodic points, or in the infinitely renormalizable setting, the Julia sets are locally connected anyway (compare e.g.\ \cite{deZottiRoesch}); however, there are explicitly known examples when local connectivity fails, especially in the presence of Cremer points \cite[\S18]{MiIntro} and in certain infinitely renormalizable cases \cite{MiLocConn}.

The research on local connectivity of polynomial Julia sets is among the deepest in all of dynamical systems. It has often been thought that the dynamics of rational maps must be even more complicated because polynomials have a basin of infinity that provides a simple and good coordinate system for the study of the dynamics, in particular through dynamic rays and their landing properties. In this paper, we propose a rather opposite point of view, at least for the dynamics of rational maps that are Newton maps of polynomials, that we phrase as the following principle:

\begin{RatRigidityPrinciple}[dynamical version]
In the dynamics of any polynomial Newton map, the orbit of every point $z$ in the Julia set can be distinguished by symbolic dynamics from every other point $z'$, unless the Newton dynamics is renormalizable and admits an embedded polynomial Julia set that fails to be rigid, and that contains the two points $z$ and $z'$.
\end{RatRigidityPrinciple}

Here we say that a polynomial Julia set is embedded in the Newton dynamics when the latter is renormalizable and a domain of renormalization has a Julia set (called a \emph{little Julia set}) that is quasiconformally conjugate to the given polynomial Julia set.

There is a parallel discussion in parameter space that has also started with the work by Douady and Hubbard \cite{Orsay} on the Mandelbrot set: if it is locally connected, then 
\begin{itemize}
\item
it has a simple and satisfactory topological model, for instance in terms of Thurston laminations \cite{ThurstonLaminations,ThurstonLaminationsDS} and Douady's pinched disks \cite{DouadyPinchedDisks};
\item
any two parameters in its boundary (the bifurcation locus) have Julia sets that can be distinguished in combinatorial terms;
\item
hyperbolic dynamics is open and dense in the space of quadratic polynomials. 
\end{itemize}

For spaces of polynomial maps beyond quadratic polynomials, local connectivity of the connectedness locus is false \cite{LavaursThesis}, but is not the right concept (see the discussion below); instead the goal is to establish \emph{rigidity} for instance in the form that any two polynomials for which the Julia sets are combinatorially indistinguishable are already quasiconformally conjugate. This rigidity conjecture is false in general \cite{Henr}, but it holds for instance when the polynomial dynamics is not renormalizable. Again, the study of parameter spaces of rational maps seems harder than for polynomials, but still we propose an analogous rigidity principle also in parameter space:

\begin{RatRigidityPrinciple}[parameter space version]
Any two polynomial Newton maps that are combinatorially equivalent are quasiconformally conjugate provided these Newton maps are either non-renormalizable, or they are both renormalizable ``in the same way'': the little Julia sets are hybrid equivalent and embedded into the Newton dynamics in combinatorially the same way.
\end{RatRigidityPrinciple}

Both versions of our rational rigidity principle, in the dynamical plane and in parameter space, can be interpreted as saying that ``the Newton dynamics behaves well unless embedded polynomial dynamics interferes'', so contrary to frequent belief the dynamics of rational maps does not exhibit any additional complications beyond those known from polynomials, once a good combinatorial structure is established. This is true at least in the case of polynomial Newton maps, which are the first family of rational maps for which a good combinatorial structure has been established \cite{LMS1, LMS2}.

\subsection{Statement of results on rigidity}
After this overview, we now provide a more precise statement of results. 
The \emph{Newton map} of a polynomial $p \colon \Cc \to \Cc$ is defined to be the rational map $N_p(z) := z - p(z)/p'(z)$; we call such a map a \emph{polynomial Newton map}. 

Our goal is to distinguish all orbits of $N_p$ in terms of symbolic dynamics. For polynomials, this issue is closely related to the topology of the Julia set, in the sense that in many cases the distinction of all orbits is possible when the Julia set is locally connected. In analogy to \cite{Fibers,Fibers2}, we define the \emph{fiber} of a point $z\in\Cc$ as the set of points whose orbits are combinatorially indistinguishable from that of $z$; this is a compact connected set (see Definition~\ref{Def:Fiber} for the general definition and also Section~\ref{Sec:Puzzles}, where this notion is discussed specifically for Newton maps). We say that \emph{the fiber of $z$ is trivial} if it consists of $z$ alone. Providing sufficient conditions for triviality of fibers is one of the chief goals of this paper, and it will imply local connectivity.

The purpose of Newton's method is to find the roots of $p$. Each root is an attracting fixed point of $N_p$ and the points with orbits converging to the roots form the \emph{basins of roots}. The dynamics in the basins is hence well-understood, and it is more interesting to look at their complement. This complement consists of the points that are either in the Julia set $J(N_p)$ of $N_p$, or are contained in Fatou components that eventually have period $2$ or higher. Every Fatou component of period~$1$ is the basin of a root because every fixed point is either attracting or the repelling fixed point at $\infty$, and there cannot be Herman rings either \cite{Shishikura}. 

Our first main theorem (Theorem~\ref{Thm:RRP}) says that for every polynomial Newton map every point that is not attracted to a root can fail to have trivial fiber only if it belongs to (or is mapped to) an embedded quasiconformal copy of the filled Julia set of an actual polynomial mapping. 

\begin{MainTheorem}[Dynamical Rigidity for Newton maps]
\label{Thm:RRP}
Let $N_p$ be a polynomial Newton map of degree $d \geqslant 2$. Then for every point $z \in \Cc$ at least one of the following possibilities holds true:
\begin{enumerate}
\item[(B)]
\label{It:B}
$z$ belongs to the \textbf{B}asin of attraction of a root of $p$;
\item[(T)]
\label{It:T}
$z$ has \textbf{T}rivial fiber;
\item[(R)]
\label{It:R}
$z$ belongs, or is mapped by some finite iterate, to the filled Julia set of \textbf{R}enormalizable dynamics (a polynomial-like restriction of $N_p$ with connected Julia set). 
\end{enumerate}
\end{MainTheorem}

Theorem~\ref{Thm:RRP} implies the following two corollaries.

\begin{corollary}[Basins have locally connected boundaries]
\label{Cor:BasinsLocConn}
For every Newton map of degree $d\ge 2$, every connected component of the basin of every root has locally connected boundary.
\end{corollary}

Corollary~\ref{Cor:BasinsLocConn} was also shown recently by \cite{WYZ} as their main theorem.

For $k \ge 2$, let ${\Sspace}_k$ be the set of  polynomials $f \colon \C \to \C$  of degree at most $k$ and with connected Julia set such that for each $f \in {\Sspace}_k$
\begin{itemize}
\item
most of the Fatou components of $f$ are small: for every $\eps > 0$ there exists only finitely many Fatou components with spherical diameter exceeding $\eps$;
\item
if $f$ has Siegel periodic points, then the boundaries of the corresponding Siegel disks are {Jordan curves}.
\end{itemize}
Moreover, let $\Sspace$ be the union of all $\Sspace_k$ for all $k$. (The letter ${\Sspace}$ in the notation stands for \textbf{S}mall Fatou and  circular \textbf{S}iegel boundaries.) 

The Julia set of a polynomial in ${\Sspace}_k$ need not be locally connected (for instance, it may have Cremer points). If the degree of a Newton map is $d \ge 3$, then a polynomial-like restriction of the map can have up to $d-2$ critical points, and if they are in different periodic components then this polynomial-like map can have degree at most $2^{d-2}$. 

\begin{corollary}[Local connectivity of Newton Julia sets, general case]
\label{Cor:LC}
Every Newton map $N_p$ of degree $d \ge 3$ has locally connected Julia set provided every polynomial-like restriction of $N_p$ straightens to a polynomial in ${\Sspace}$. 
\end{corollary}

This corollary establishes local connectivity of Julia sets of Newton maps in many non-trivial cases. For example, ${\Sspace}$ contains all polynomials without bounded Fatou components (which includes many examples of polynomials with non-locally connected Julia sets), as well as all geometrically finite polynomials. The latter polynomials have locally connected Julia sets \cite{TLYin}, and thus most of their Fatou components are small by a well-known criterion for local connectivity of sets \cite{Why}. In particular, this implies that ${\Sspace}_2$ contains all quadratic polynomials $f$ without Siegel disks: if the only critical point of $f$ is not in the attracting or parabolic basin (i.e.\ $f$ is not sub-hyperbolic), then the critical point must be in the Julia set, and hence $f$ has no bounded Fatou components at all. Corollary~\ref{Cor:LC} combined with the last observation yields

\begin{corollary}[Local connectivity of cubic Newton maps]
The Julia set of every cubic Newton map without Siegel disks is locally connected. \qed
\end{corollary}

This corollary generalizes a result of Roesch~\cite{Roe} and provides a positive answer to Conjecture 8.7 in that paper modulo the Siegel case. 

Corollary~\ref{Cor:LC} demonstrates the phenomenon of ``enhanced connectivity'' within Newton Julia sets, as was first observed by Roesch: even if non-locally connected polynomial Julia sets are embedded in the Newton dynamical plane, the Julia set of the Newton map can still be locally connected. (However, this does not imply that the Newton Julia set enjoys the advantages usually associated with local connectivity of polynomial Julia sets, for instance a satisfactory topological model.)

\begin{remark}
It might well be that ${\Sspace}$ contains all polynomials without Siegel disks. Using the axiomatic approach towards rigidity from \cite{Dr} and triviality of fibers proven in~\cite{Kiwi} it is possible to establish this claim partially: most of the Fatou components of polynomials without irrationally neutral periodic points are small, and hence such polynomials belong to ${\Sspace}$.
\end{remark}

\begin{remark}
Many polynomial Julia sets are known to be locally connected, and even have all their fibers trivial. The corresponding results can be imported to the Newton dynamics, for example, as follows. Suppose $z$ belongs to some renormalizable polynomial-like restriction of $N_p$ for which the polynomial Julia set has trivial fiber at the point corresponding to $z$, then the Julia set of $N_p$ has trivial fiber at $z$, and in particular is locally connected at $z$. The idea of proof for this statement is that two points $w, w'$ in a polynomial Julia set are in different fibers if and only if there is a pair of (pre)periodic dynamic rays landing at the same point in the polynomial Julia set that separates $w$ from $w'$. Within the dynamics of $N_p$ these rays can be replaced by ``bubble rays'' consisting of sequences of components of root basins that converge to the same landing point with the same separation properties. (Note, however, that it is not clear that if a polynomial Julia set is locally connected at some point, then its fiber must be trivial, and this point can be separated from every other point in the Julia set; compare \cite{Fibers,Fibers2}.) 
\end{remark}

It may come as a surprise that for local connectivity of Newton maps, the only potential issue might be Siegel disks, not Cremer points or infinite renormalizability. If there existed a polynomial with non-locally connected Siegel disk boundary, then it would be possible to construct a Newton map with non-locally connected Julia set by embedding the Julia set of that polynomial, in agreement with our rigidity principle. However, it is conjectured that Siegel disks of polynomials always have locally connected boundaries (see, for example, \cite{Cheritat, PetZakeri, Zhang, SY}). We view this as a strong hint that perhaps all polynomial Newton maps have locally connected Julia sets.

\begin{remark}
The degree of a Newton map $N_p$ is always equal to the number of distinct roots of $p$ (ignoring multiplicities). It is well known that if $N_p$ has degree $d=2$, then the Julia set $J(N_p)$ is a quasi-circle through $\infty$ (in particular, it is a straight line if $p$ has degree $2$ as well), and the two complementary domains are the basins of the two roots. This case is trivial, and the case $d=1$ is even more trivial, so they are excluded from our discussions and we assume $d\ge 3$. 
\end{remark}

Our second main result (Theorem~\ref{Thm:QCRigidity}) is a parameter space counterpart to Theorem~\ref{Thm:RRP}. We say that two Newton maps are \emph{combinatorially equivalent} if their \emph{Newton graphs} coincide (see Definition~\ref{Def:CombEquiv} for details); an equivalent way of saying this is that all the components of the basins of the roots are connected to each other in the same way (some examples of Newton dynamical planes are shown in Figure~\ref{Fig:TouchingBasins}).

\begin{MainTheorem}[Parameter rigidity for Newton maps]
\label{Thm:QCRigidity}
If two polynomial Newton maps are combinatorially equivalent, then they are quasiconformally conjugate in a neighborhood of the Julia set provided
\begin{enumerate}
\item
either they are both non-renormalizable, 
\item
or they are both renormalizable, and there is a bijection between their domains of renormalization that respects hybrid equivalence between the little Julia sets as well as their combinatorial position.
\end{enumerate}
The domain of this quasiconformal conjugation can be chosen to include all Fatou components not in the basin of the roots, and its antiholomorphic derivative vanishes on those Fatou components as well as on the entire Julia set.

Moreover, if these Newton maps are normalized so that they attracting-critically-finite (as defined below), then they are even affine conjugate.
\end{MainTheorem}

The conditions in the renormalizable case mean the following: the renormalizable ``little Julia sets'' should correspond to the same polynomial dynamics (up to a quasiconformal conjugation that is conformal on the filled-in Julia set of the polynomials), and they should be connected to the Newton graph at the same combinatorial position. This will be made precise in Section~\ref{Sec:ProofParamRigidity}. 

Finally, a Newton map is called \emph{attracting-critically-finite} if the orbit of every critical point in the basin of a root is eventually fixed; this can be accomplished by a routine quasiconformal surgery on a compact subset of the basins of the roots (see Section~\ref{Sec:Puzzles}).

\begin{figure}[htbp]
\begin{center}
\includegraphics[width=1.\textwidth]{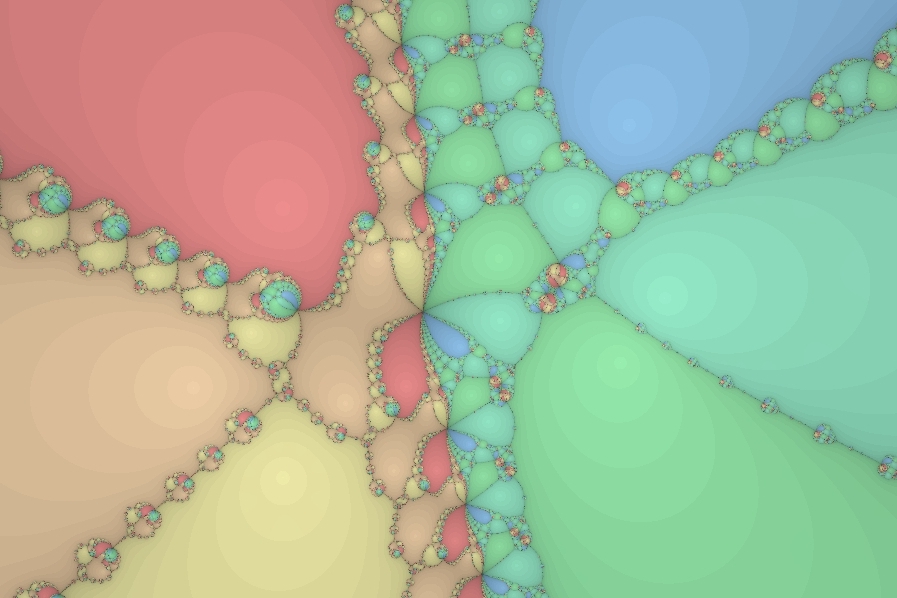}\hfill
\vskip .015\textwidth
\includegraphics[width=.49\textwidth]{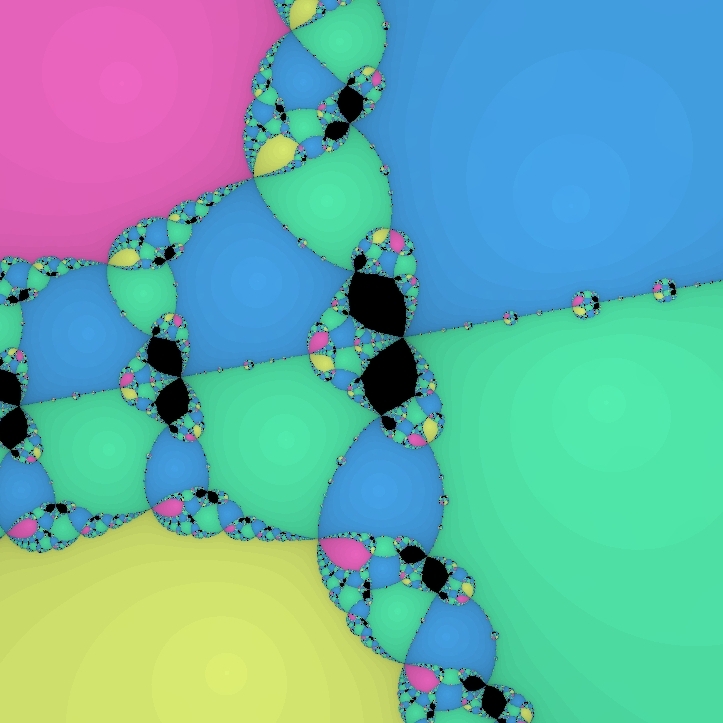}\hfill
\includegraphics[width=.49\textwidth]{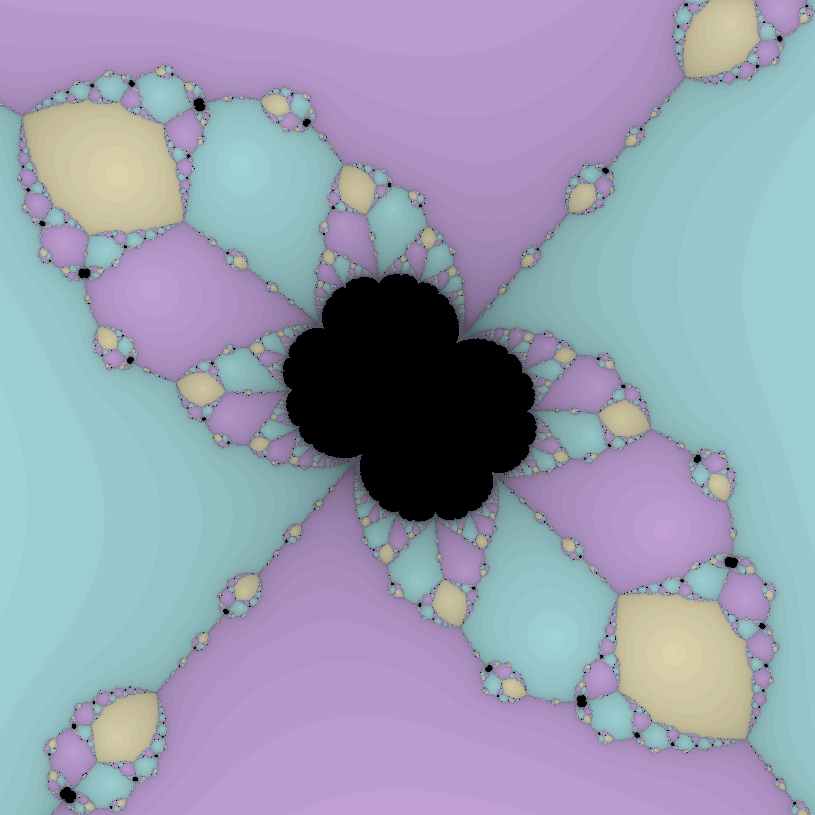}
\caption{The dynamical planes of various Newton maps (of degree 6 (top), 4 (bottom left), and 3 (bottom right)); the basins of different roots are shown in different colors. Different Newton maps can often be distinguished combinatorially in terms of the combinatorial structure of touching components of the basins of roots. Renormalizable parts of the dynamical plane are shown in black.}
\label{Fig:TouchingBasins}
\end{center}
\end{figure}

\subsection{Box mappings}
One of our key tools are \emph{complex box mappings} as introduced by Kozlovski and van Strien \cite{KvS09} (see also~\cite{GS}, and~\cite[Appendix~A]{ALdM} for the similar notion of \emph{puzzle mappings}). These maps are natural generalizations of polynomial-like maps to the case when the domains of definition and the ranges are disconnected (see Figure~\ref{Fig:ExampleBoxMapping}). For any point $z$ in such a box mapping $F$, the fiber $\fib(z)$ is the component containing $z$ of the set of points that have the same symbolic dynamics as $z$ with respect to the connected components for the domain of definition of $F$: that is, $\fib(z)$ consists of points with the same itinerary through all these components. Again, a precise definition will come later (Section~\ref{Sec:ComplexBoxMappings}). Kozlovski and van Strien give sufficient conditions for box mappings to have all their fibers trivial (in different language; see \cite[Theorem 1.4 (1)]{KvS09}, and also~\cite{DKvS}). Our third main result (Theorem~\ref{Thm:RigidityComplexBox}) is an upgrade to their theorem: our result applies to \emph{all} box mappings and provides sufficient conditions for most individual fibers to be trivial. Similarly as for polynomial-like maps, some points can only be iterated finitely many times; we say that such points \emph{escape} (from the box mapping). There also might be components of the domain of a box mapping without escaping points; we call them \emph{\eqref{Item:NE} components} (components with \textbf{N}o \textbf{E}scape). For a cycle of such components, the orbits of all points on this cycle remain there forever. 

\makeatletter
\write\@auxout{\string\newlabel{Item:T}{{T}{\thepage}{bla}{bla}{} }}
\write\@auxout{\string\newlabel{Item:R}{{R}{\thepage}{bla}{bla}{} }}
\write\@auxout{\string\newlabel{Item:NE}{{NE}{\thepage}{bla}{bla}{} }}
\write\@auxout{\string\newlabel{Item:CB}{{CB}{\thepage}{bla}{bla}{} }}
\makeatother

\begin{MainTheorem}[Generalized rigidity for complex box mappings]
\label{Thm:RigidityComplexBox}
Consider an arbitrary box mapping and an arbitrary non-escaping point $z$. Then at least one of the following cases occurs:
\begin{itemize}
\item[\eqref{Item:T}]
$z$ has \textbf{T}rivial fiber;
\item[\eqref{Item:R}]
$z$ belongs, or is mapped by some finite iterate, to the filled Julia set of \textbf{R}enormalizable dynamics (a polynomial-like restriction of the given box mapping with connected Julia set); 
\item[\eqref{Item:CB}]
the orbit of $z$ \textbf{C}onverges to the \textbf{B}oundary of the domain of definition of the box mapping.
\item[\eqref{Item:NE}]
(\textbf{N}o \textbf{E}scape:) the domain of the box mapping contains a periodic component that maps surjectively to itself by some iterate of the box mapping, and $z$ eventually maps to such a component (so the fiber of $z$ is equal to the closure of its component); 
\end{itemize}
\end{MainTheorem}

Observe that the first two possibilities here exactly match the two possibilities in Theorem~\ref{Thm:RRP} for points not in basins of the roots. The last case does not naturally arise in many cases where box mappings are extracted from dynamical systems on $\Cc$ and is admitted by the fairly general definition of box mappings (see Definition~\ref{Def:BoxMap}). We will provide a definition of ``dynamically natural'' box mappings in Definition~\ref{Def:NaturalBoxMapping}.

The concept of renormalization will be discussed in Section~\ref{Sec:ComplexBoxMappings}. Note that there are different  concepts of renormalization: in the sense of Douady--Hubbard, as well as of Kozlovski--van Strien.

\bigskip

\noindent
\emph{Earlier work on Newton's method}.
Newton's method as a dynamical system has been studied by various people for a long time, in many cases with a focus on the cubic case. For cubic Newton maps there is a single free critical point and the parameter space is complex one-dimensional, like the well-studied case of the dynamics of quadratic polynomials and the Mandelbrot set. In particular, we would like to mention the classical work by Tan Lei \cite{TanLeiNewton} with a combinatorial study of the Newton parameter space, with a recent refinement by Roesch, Wang, and Yin in \cite{RWY}. In \cite{Roe}, Roesch has shown that the Newton map of a cubic polynomial has locally connected Julia set in many cases, even when it is renormalizable and the embedded polynomial Julia set is not locally connected. 

There are two recent manuscripts that study Newton's method of arbitrary degrees in a similar spirit as we do here. The main result of Wang, Yin, and Zeng \cite{WYZ} is that immediate basins have locally connected boundaries. 
Roesch, Yin, and Zeng show in \cite{RYZ} that all non-renormalizable Newton maps are rigid (in parameter space). Both are corollaries of our results: the first is Corollary~\ref{Cor:LC}, and the second is Theorem~\ref{Thm:QCRigidity} in the special case of non-renormalizability.

There is also recent work on Newton's method as an efficient root finder. Among the early results are a paper by Przytycki \cite{Pr} that shows that immediate basins are always simply connected (with a generalization by Shishikura \cite{ShishikuraSimplyConnected}), and one by Manning \cite{Manning} that shows where to start the Newton iteration to find some ``exposed roots''. A sufficient small set of starting points that always finds all roots was constructed in \cite{HSS} with a probabilistic version in \cite{BLS}. Estimates on the necessary number of iterations were given in \cite{NewtonIterations,NewtonEfficient,NewtonTodor}; see also the overview in \cite{NewtonOverview}. Finally, experiments that highlight the efficiency of Newton's method for certain polynomials of degrees exceeding one billion were described in \cite{NewtonRobin1,NewtonRobinMarvin}. 

\bigskip

\textbf{Notation.} In order to lighten notation, we write $f^k$ for the $k$-fold iterate of a map $f$, that is $f^k := f^{\circ k} = \underbrace{f \circ \ldots \circ f}_{k \text{ times}}$.

We will also write $\Crit(f)$ for the set of critical points of a map $f$ and $\orb(z) := \left\{f^k(z) \colon k \geqslant 0\right\}$ for the orbit of a point $z$ under the dynamics of $f$. The $\omega$-limit set of $\orb(z)$ is defined as 
\[
\omega(z) := \bigcap_{n \in \mathbb N} \ovl{\{f^k(z) \colon k > n\}}
\;.
\]
Note that $\overline{\orb(z)} = \orb(z) \cup \omega(z)$.  

A set $X$ is \emph{nice} if the orbit of the boundary of $X$ does not intersect the interior of $X$, i.e.\ $f^k(\partial X) \cap \inter X = \emptyset$ for all $k \ge 0$ \cite[$\S 31$]{LyuBook}.

A \emph{component} of a set $X$ is a connected component of $X$. A \emph{critical component} of $X$ is a connected component containing a critical point.

The Lebesgue measure in $\C$ will be denoted as $\area(\cdot)$.

By the diameter of a set $X$ we will usually understand Euclidean or spherical diameter, depending whether $X$ lies in $\C$ or $\Cc$; we will denote it as $\diam X$.

An \emph{annulus} is a doubly-connected domain in $\Cc$; we will write $\modulus(B)$ for the modulus of an annulus $B$.

\section{On general puzzles}
\label{Sec:GeneralPuzzles}

In this preparatory section we start with some general discussion and fix terminology concerning puzzles. The results of this section will be used as a toolbox in the proofs of our main results (Theorems~\ref{Thm:RRP},~\ref{Thm:QCRigidity}, and~\ref{Thm:RigidityComplexBox}).

\subsection{Puzzle pieces, fibers, and the Markov property}

Let $g \colon U \to V$ be a holomorphic map between two open sets $U \subseteq V \subset \Cc$ so that connected components of $U$ resp.\,$V$ have disjoint closures; we do not require that the components of $U$ or $V$ be simply connected. Further assume that $g$ has only finitely many critical points. We describe a setting of \emph{puzzles} in the spirit of the well known Yoccoz puzzles, adapted to the needs of our Newton dynamics. 
Suppose that there exists a nested sequence $\left(S_n\right)_{n=0}^\infty$ of open sets such that $V=S_0 \supset U=S_1\supset S_2\dots$, every component of $S_{n+1}$ is either compactly contained in or coincides with the corresponding component of $S_n$ and for every $n \geqslant 0$  the restriction $g \colon S_{n+1} \to S_n$ is a proper map. Further assume that the closure of each $S_n$ can be represented as a (not necessarily finite) union of closed topological disks  $P_n^i$ ($i$ runs over some finite or countable index set $I_n$) that can only intersect along their boundaries (see Figure~\ref{Fig:GeneralPuzzles}). We call each $P_n^i$ a \emph{puzzle piece of depth $n$}. The union of all puzzle pieces of depth $n$ comprises the \emph{puzzle partition (of $\ovl S_n$) of depth $n$}. An \emph{open puzzle piece of depth $n$} is the interior $\inter P_n^i$ of a puzzle piece of depth $n$. We will call the topological graph $\Gamma_n := \bigcup_{i \in I_n} \partial P_n^i$ the \emph{puzzle boundary of depth $n$}: \emph{vertices} of this graph are either points on $\partial S_n$ where at least two puzzle pieces meet, or points in $S_n$ where at least three puzzle pieces meet (note here that $\partial S_n \subset \Gamma_n$ for every $n$); an \emph{edge} of $\Gamma_n$ is a vertex-free set homeomorphic to an interval that connects two vertices. For simplicity we assume that all edges in all $\Gamma_n$ are smooth and the boundary of each puzzle piece contains finitely many vertices. In general, every puzzle piece of depth $n$ shares an edge with one or several further puzzle piece of the same depth $n$, and hence every component of $\ovl S_n$ contains many puzzle pieces of depth $n$. We allow the special case that a component of $\ovl S_n$ consists of a single puzzle piece, say $Y$; in this case the definition of edges and vertices does not apply, and we choose an  arbitrary point on $\partial Y$ as a vertex and let the rest of $\partial Y$ be an edge that connects the vertex to itself.  

\begin{figure}[htbp]
\begin{center}
\includegraphics[scale=.5, trim=0 0 0 0, clip]{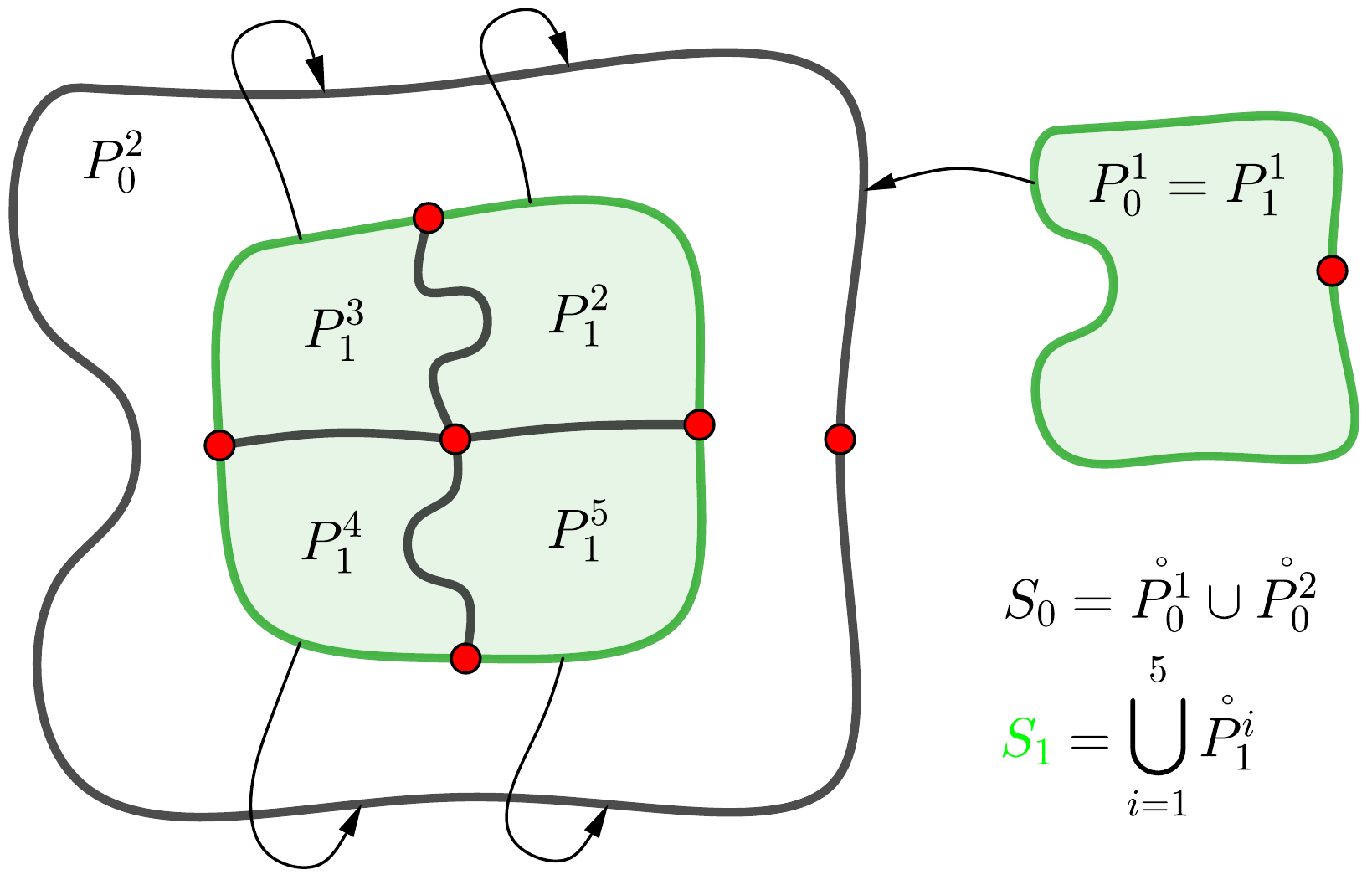}
\caption{An schematic picture of general puzzles. The puzzle partition of depth $0$ consists of $2$ puzzle pieces, both are connected components of $\ovl S_0$; the puzzle partition of depth $1$ consists of $5$ puzzle pieces (filled in green). The red dots indicate vertices on the puzzle boundaries.}
\label{Fig:GeneralPuzzles}
\end{center}
\end{figure}

\begin{remark}
The previous paragraph describes, in fairly large generality, a construction that includes not only the setting of the well known Yoccoz puzzles (where each $S_n$ consists of finitely many puzzle pieces), but it also caters for two settings that will be specified in the upcoming sections. First, we will be interested in puzzles for complex box mappings (see Definition~\ref{Def:BoxMap}). In that case, $g$ will be a box mapping, where each $S_n$ will be a (possibly infinite) union of open topological disks with disjoint closures; the closure of each of the disks will serve as a puzzle piece of depth $n$. In other words, for a box mapping and for any given $n$ the set of puzzle pieces of depth $n$ equals the set of closures of the connected components of $S_n$; see Definition~\ref{Def:BoxPuzzle} for details. For the second time the construction in the previous paragraph will be specified for polynomial Newton maps $N_p$ (see Section~\ref{Sec:Puzzles}). There $g$ will stand for a particularly chosen iterate of the Newton map, while $S_n$ will be the Riemann sphere minus finitely many suitably chosen closed topological disks bounded by equipotentials in the respective basins of roots of $p$; each of these removed disks is a neighborhood of either a root of $p$ or an iterated preimage of such a root for a bounded number of iterations (see Definition~\ref{Def:NewtonPuzzles} for details).
\end{remark}

\begin{definition}[Markov property]
\label{Def:Markov}
The union of all puzzle pieces of all depths has the \textit{Markov property} if:
\begin{enumerate}
\item
any two puzzle pieces are either nested or have disjoint interiors; in the former case the puzzle piece of bigger depth is contained in the puzzle piece of a smaller depth;
\item
the image of each puzzle piece $P_n^i$ of depth $n>0$ is a puzzle piece $P_{n-1}^j$ of depth $n-1$, and the restriction $g \colon P_n^i \to P_{n-1}^j$ is a branched covering.
\end{enumerate}
Equivalently, the Markov property can be stated in terms of puzzles: the union of all puzzle pieces of all depths has the Markov property if $g(\Gamma_n \cap S_{n+1}) \subset \Gamma_n$ for all $n \geqslant 0$. (Note that for puzzles coming from box mappings this condition is automatically satisfied since $\Gamma_n \cap S_{n+1} = \emptyset$, see Definition~\ref{Def:BoxPuzzle}). 
\end{definition}

We will say that $g$ is a \emph{holomorphic map with well-defined Markov partition} if $g \colon U \to V$, with $U \subseteq V \subset \Cc$, is a holomorphic map as described at the beginning of the section and for which there exists a nested sequence of open sets $V = S_0 \supset U=S_1 \supset \ldots$ with a well-defined puzzle partition into puzzle pieces the union of which has the Markov property.

For a holomorphic map $g$ with a well-defined Markov partition, each puzzle piece $P_n$ of depth $n$ is a nice set: the orbit of the boundary of $P_n$ does not intersect the interior of $P_n$. This follows from the conditions $g(\Gamma_n \cap S_{n+1}) \subset \Gamma_n$ and $S_n \supset S_{n+1}$ ($n=0,1,\ldots$).

\begin{remark}
Alternatively, one can define puzzle pieces as pull-backs of  a certain initial set that is nice. The components of this initial set are declared to be the puzzle pieces of zero depth, and puzzle pieces of depth $n$ are components of the $n$-fold pull-back of components of the initial set.
\end{remark}

\begin{remark}
An important example of a nice set, apart from individual puzzle pieces, is a union of puzzle pieces of the same depth. On the other hand, if $P$ and $Q$ are puzzle pieces such that $\inter P \cap \inter Q = \emptyset$ and $g^k(P) \subsetneq Q$ for some $k \ge 0$, then $P \cup Q$ is \emph{not} a nice set. 
\end{remark}

\begin{definition}[Puzzle piece centered at a point]
\label{Def:PuzzleCentered}
Given a point $x \in S_n$, define $P_n(x)$ to be the union of all puzzle pieces of depth $n$ containing $x$.
\end{definition}

From the definition above it is clear that if $x$ is not on the boundary of any puzzle piece of depth $n$ (equivalently, if $x \not \in \Gamma_n$), then $P_n(x)$ is the unique puzzle piece of depth $n$ containing the point $x$. Otherwise, $P_n(x)$ is a union of puzzle pieces with $x$ in their common boundary. Note that these sets do not form a Markov partition: it may be that $P_n(x)$ and $P_n(y)$ are different with intersecting interiors if $x$ or $y$ are in $\Gamma_n$. However, it is still true that the restriction $g \colon P_n(x) \to P_{n-1}\left(g(x)\right)$ is a branched covering. 

Let us spell out an elementary argument that will be used several times below without explicit mention. If $x$ is a point that does not belong to the puzzle boundary of any depth, then every point on the orbit of $x$ also does not belong to the puzzle boundary of any depth, and hence $P_n(g^k(x))$ is a puzzle piece (of depth $n$) for all $n$ and $k$ (as long as $x$ can be iterated $k$ times).

We say that a point $x \in U$ \emph{escapes} if $x \in S_n \sm \ovl S_{n+1}$ for some $n \geqslant 0$. Thus the set of non-escaping points of $g$ (the \emph{non-escaping set of $g$}) is precisely 
 $\bigcap_{n=0}^\infty \ovl S_n$; this is the set of points that can be iterated infinitely often. 
 
\begin{definition}[Fiber, trivial fiber]
\label{Def:Fiber}
For a non-escaping point $x$, the set 
\[
\fib(x) := \bigcap_{n \geqslant 0} P_n(x)
\]
is called the \emph{fiber} of $x$ (with respect to the partition of $S_n$).
We say that $x$ has \emph{trivial fiber} if $\fib(x) = \{x\}$.
\end{definition}

The Markov property of puzzle partitions is a powerful combinatorial property allowing us to study maps from the point of view of symbolic dynamics.
We define the \emph{itinerary at level $n$ of a point $x$} as the sequence  $\left(P_n(g^i(x))\right)_{i=0}^\infty$.  Two points have the same itinerary if their itineraries at all levels coincide.  In particular, the fiber $\fib(x)$ consists of all points that have the same itineraries at all levels. 
In other words, $\fib(x)$ consists of all points that are \emph{dynamically indistinguishable} with respect to our puzzle partition. Hence, if the fiber is trivial, then the dynamics of $g$ at $x$ is \emph{rigid}: there is no point other than $x$  with the same itinerary as $x$.

We will also say that the fiber $\fib(x)$ is \emph{periodic} if $x$ has periodic itinerary; this property is independent of a particular choice of a point in the fiber. 

A point $x$ is called \emph{combinatorially recurrent} if $x$ does not belong to the puzzle boundary of any depth and the orbit of $g(x)$ under $g$ intersects $\inter P_n(x)$ for every $n$. This implies that the orbit of $g(x)$ intersects every $\inter P_n(x)$ infinitely often: if this was not true, then for some $n$ there was a largest $k$ with $g^k(x)\in \inter P_n(x)$, so $g^k(x)\in\inter P_m(x)$ for all $m\ge n$; hence $g^k(x)\in \fib(x)$. But then $g^j(x)$ and $g^{k+j}(x)$ are in the same fiber for every $j\ge 0$, so $g^{jk}(x)\in\fib(x)$ for all $j$.

\subsection{The first return construction}

In many instances, while working with puzzle pieces one can control critical orbits by considering the first moment these orbits enter a given collection of puzzle pieces. The first lemma in this subsection, although fairly easy, provides a first example of this strategy and is the basis for many constructions of partial maps that we will later work with.

\begin{lemma}[First entry maps have uniformly bounded degrees]
\label{Lem:FirstTime}
For every holomorphic map $g$ with well-defined Markov partition there is a constant $D\in\N$ with the following property: for every puzzle piece $Y$ of any depth $n$ and for every point $z$ that does not belong to the puzzle boundary of any depth, if $k \ge 1$ is the least index so that $g^{k}(z)\in Y$, then the map $g^{k} \colon P_{n+k}(z) \to P_n(g^{k}(z)) = Y$ has degree bounded by $D$ (independent of $n$ and $k$). 
\end{lemma}

\begin{proof}
Consider the sequence of puzzle pieces $\left(P_{n+k - i}(g^{i}(z))\right)_{i=0}^{k-1}$. We claim that these $k$ puzzle pieces have disjoint interiors. If not, then by the Markov property (Definition \ref{Def:Markov}) we have $P_{n+k-i}(g^{i}(z)) \subset P_{n+k-j}(g^{j}(z))$ for some $i<j<k$ and $g^{k-j}(P_{n+k-j}(g^{j}(z)))=P_n(g^{k}(z))=Y$, hence $g^{k-j}(P_{n+k-i}(g^{i}(z)))\subset Y$, so $g^{k-j+i}(z)\in Y$ in contradiction to minimality of $k$ (see Figure~\ref{Fig:FirstReturnDegree}).

\begin{figure}[htbp]
\begin{center}
\includegraphics[scale=.9, trim=30 35 30 35]{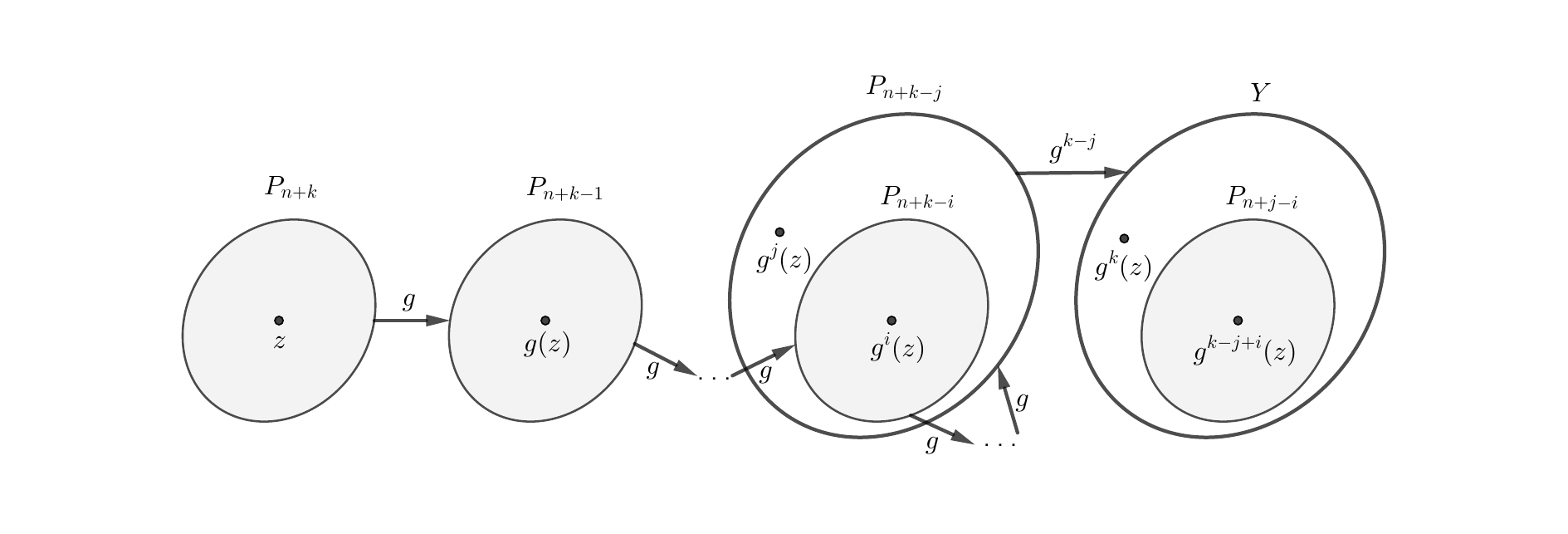}
\caption{An illustration on how to conclude contradiction to minimality of $k$ in the proof of Lemma~\ref{Lem:FirstTime}.}
\label{Fig:FirstReturnDegree}
\end{center}
\end{figure}

In particular, each critical point of $g$ can lie in the interior of at most one puzzle piece in this sequence. Therefore, the claim follows with $D$ equal to the product of the degrees of all critical points of $g$.
\end{proof}

In the coming sections we will use the notions of \emph{first return map} and \emph{first entry map} as follows. For an open set $W\subset U$, following \cite{KSS} we define 
\[
\DomE(W) := \left\{z \in U \colon \exists n \ge 1: g^n(z) \in W \right\} \text{ and } \Dom(W) := \DomE(W) \cap W.
\]
Whenever the set $W$ is understood, we will drop it from the notation and simply write $\DomE$ and $\Dom$. We define the \emph{first entry map} (to $W$) as the map $E\colon \DomE\to W$  via $E(z):=g^n(z)$  for the smallest possible $n$, and in particular the \emph{first return map} $R\colon\Dom \to W$ as $R:=E|_{\DomE \cap W}$. Finally, we define the \emph{first landing map} $L\colon\DomE\cup W\to W$ via $L=\id$ on $W$ and $L=E$ on $\DomE \sm W$; we will write $\DomL(W) = \DomE \cup W$ for the domain of the first landing map to $W$.

\begin{remark}
The first return map construction is a key tool that we will use in order to produce box mappings (see Definition~\ref{Def:BoxMap}).
\end{remark}

\begin{lemma}[First return, entry and landing maps to a nice set]
\label{Lem:FirstReturnConstruction}
Let $g \colon U \to V$ be a holomorphic map with well-defined Markov partition, and let $ W \subset U$ be a nice set consisting of the interiors of puzzle pieces. Then the first entry map $E \colon \DomE \to W$, the first return map $R\colon\Dom\to W$, and the first landing map $L\colon\DomL \to W$ have the following properties.:
\begin{enumerate}
\item
\label{It:FRM1}
$\DomE$ and $ \Dom$ are disjoint unions of interiors of puzzle pieces, and hence are open;
\item
\label{It:FRM2}
for every component $Y'$ of $\DomE$ there is a component $W'$ of $W$ and an integer $n$ so that $E|_{Y'} = g^n|_{Y'}$, and $E\colon Y'\to W'$ is a proper map.
The same property follows for $R$ and $L$ by restriction.

\item
\label{It:FRM3}
the local degrees of the  maps $R$, $E$, and $L$ are bounded in terms of $g\colon U\to V$. Moreover, if $\Crit(g) \subset  W$, then $\Crit(R)$, $\Crit(E)$ and $\Crit(L)$ are contained in  $\Crit(g)$;
\item
\label{It:FRM4}
every $g$-orbit that intersects $ W$ infinitely often lands in the set $\{z \in  \Dom \colon R^n(z) \in  \Dom \text{ for all }n\ge0\}$ (the \emph{non-escaping set} of $R$).  
\end{enumerate}
\end{lemma}

\begin{remark}
If the set $W$ in the lemma above does not satisfy the condition $\Crit(g) \subset W$, then it might happen that the first return map to $W$, as well as the first entry and landing maps, have infinitely many critical points, even though by hypothesis $g$ does not. 
\end{remark}

\begin{proof}
For given $z \in  \DomE$, let $n$ be minimal such that $g^n (z) \in  W$; let $W'$ be the component of $ W$ containing $g^n(z)$, and let $k > 0$ be the depth of $\ovl{W'}$. Then by the Markov property $g^n$ maps $\inter P_{n+k}(z)$ onto $W'$. If $z \in  \Dom  \sm W'$, then $\inter P_{n+k}(z) \subset  \Dom $ because $ W$ is a nice set. 
If $z \in  \DomE  \sm ( \Dom  \sm W')$, then $\inter P_{n+k}(z) \subset  \DomE $ by the definition of the domain of a first entry map and the Markov property. Hence, it follows that $\inter P_{n+k}(z) \subset  \DomE $; moreover, $\partial P_{n+k}(z) \cap  \DomE  = \emptyset$. Therefore, every component of $ \DomE $, and thus of $ \Dom  \subset  \DomE $, is the interior of a puzzle piece of the original map $g$, and by the Markov property, they are disjoint. This establishes property \eqref{It:FRM1}. Observe that some components of $ \Dom $ might coincide with some components of $ W$. 

Property \eqref{It:FRM2} follows directly from \eqref{It:FRM1} by choosing the least $n$ for every $z$; therefore $E$ restricted to a component $Y'$ of $ \DomE$ is a proper map $g^n \colon Y' \to W'$ from $Y'$ onto a component $W'$ of $W$. Similarly for $R$ and $L$ by restriction.

For property \eqref{It:FRM3}, observe that by construction every $g$-orbit of a component of $ \Dom $, resp.\ $ \DomE $, can intersect every critical point of $g$ at most once until it reaches $ W$
(compare the proof of Lemma~\ref{Lem:FirstTime}). This implies the first claim. Since, by hypothesis, all critical points of $g$ are already in $ W$, we conclude that every critical point of $R$, resp.\ $E$ and $L$, must be a critical point of $g$. The claim follows. 

Finally, suppose that $\orb(z)$ intersects $ W$ infinitely often. Then for every $x \in \orb(z) \cap  W$ there is a minimal $s > 0$ such that $g^s(x) \in  W$. By construction of $R$ it follows that $x \in  \Dom $, and $R$ restricted to the component of $ \Dom $ containing $x$ is equal to $g^s$. Hence, $\orb(z) \cap  W$ lies in the non-escaping set of $R$ (equivalently, $\orb(z) \cap ( W \sm  \Dom ) = \emptyset$), and property \eqref{It:FRM4} follows.
\end{proof}


We say that a puzzle piece $P_n$ of depth $n$ is \emph{weakly protected by a puzzle piece $P_m$}, necessarily of depth $m < n$, if $P_n \subset \inter P_m$. If $m=n-1$, then we say that $P_n$ is \emph{protected}. The following lemmas guarantee compact containment of pullbacks of certain weakly protected puzzle pieces.

\begin{lemma}[First return to weakly protected puzzle piece]
\label{Lem:FRD}
\label{Lem:FRD2}
Let $g$ be a holomorphic map with well-defined Markov partition, and $z$ be a point that does not belong to the puzzle boundary of any depth. Suppose that there exists a depth $n$ and an integer $k\ge 1$ so that $P_{n+k}(z)$ is weakly protected by $P_n(z)$ and one of the following is true:
\begin{enumerate}
\item
\label{It:NeverReturns}
the orbit of $g(z)$ never enters $P_n(z)$;
\item
\label{It:Returns}
the orbit of $g(z)$ intersects $P_n(z)$ and $k$ is minimal so that $g^k(z)\in \inter P_n(z)$.
\end{enumerate}
Then every connected component of $\Dom\left(\inter P_{n+k}(z)\right)$ is compactly contained in $\inter P_{n+k}(z)$. 
\end{lemma}

\begin{proof}

By Lemma~\ref{Lem:FirstReturnConstruction}, the closure of any component of the domain of the first return map to $\inter P_{n+k}(z)$ is a puzzle piece of depth at least $n+k+1$. Let $Y \subset P_{n+k}(z)$ be one of these puzzle pieces, say at depth $n+k+l$ with $l \geqslant 1$. By construction, $g^l \colon Y \to P_{n+k}(z)$ is a (branched) covering. Let us show that $Y\subset \inter P_{n+k}(z)$. The argument is slightly different in cases \eqref{It:NeverReturns} and \eqref{It:Returns}.

If the orbit of $g(z)$ never enters $P_n(z)$, then for every $s \in \{1,\ldots, k\}$ the puzzle piece $P_{n+k-s}\left(g^s(z)\right) = g^s\left(P_{n+k}(z)\right)$ is disjoint from $\inter P_n(z)$. The same is true for $g^s(Y)$ because $Y \subset P_{n+k}(z)$; hence $l \ge k+1$. Since $P_{n+k}(z)$ is weakly protected by $P_n(z)$ and $g^l(Y) = P_{n+k}(z)$, by pulling back $P_n(z)$ we conclude that $Y$ is weakly protected by the puzzle piece of depth $n+l$. Since $l \ge k+1$, this puzzle piece lies in $P_{n+k}(z)$. Therefore, $Y\subset\inter P_{n+k}(z)$.

If the orbit of $g(z)$ intersects $P_n(z)$ and $k$ is the first time this orbit enters $\inter P_n(z)$, then the puzzle pieces $(P_{n+k-i}(g^i(z)))_{i=0}^{k-1}$ have disjoint interiors. Therefore, $l \ge k$. Assume that $Y$ is not contained in $\inter P_{n+k}(z)$. Then $\partial g^k(Y)$ intersects $\partial g^k\left(P_{n+k}(z)\right) =\partial P_n(g^k(z))= \partial P_n(z)$ at some point $w \in \partial (g^k(Y)) \cap \partial P_n(z)$. Since puzzle pieces are nice sets, i.e.\ $g^m(\partial P_n(z)) \cap \inter P_n(z) = \emptyset$ for all $m\ge 0$, we have in particular $g^{l-k}(w)\not\in \inter P_n(z)$. But $g^{l-k}(w)\in \partial g^{l-k}(g^k(Y))=\partial g^l(Y)=\partial P_{n+k}(z)\subset \inter P_n(z)$, a contradiction.
\end{proof}

\begin{remark}
It is possible to show that if $P_n$ is protected, then every component of the first return domain to $\inter P_n$ is compactly contained in $\inter P_n$, see \cite[$\S 31$]{LyuBook}.
\end{remark}

\subsection{Some standard pullback and Koebe-type lemmas}
\label{SSec:Shape}

In Subsection~\ref{SSec:AccumAtPer} we will show that certain fibers are trivial. This, as well as many other constructions later in the paper, will be done by controlling moduli of annuli under pullbacks. In the present subsection, we collect some of the standard results in this direction.  

\begin{lemma}[Annulus pull-back under branched covering]
\label{Lem:AnnulusLemma}
Let $f \colon Y' \to Y$ be a branched covering of degree at most $D$ between two closed topological disks. Suppose $Y_2\subset Y_1\subset Y$ are two further closed topological disks so that $A := \inter Y_1 \sm Y_2 \subset Y$ is an annulus with $\modulus(A) = \mu > 0$. Moreover, assume that $Y_1'$ and $Y_2'$ are preimage components of $Y_1$, resp.\ $Y_2$ under $f$ such that $Y_1' \supset Y_2'$. Set $A':=\inter Y_1' \sm Y_2'$. Then $\modulus(A') \geqslant \mu/D^2$.
\end{lemma}
\begin{proof}
The branched cover $f \colon Y' \to Y$ has at most $D-1$ critical points. Hence the annulus $A$ has a parallel sub-annulus $B$ of modulus $\mu/D$ that avoids all critical values (recall that $B$ is a parallel sub-annulus of an annulus $A$ if a biholomorphic map that uniformizes $A$ to a round annulus $A_0$ sends $B$ to a concentric round sub-annulus $B_0$ of $A_0$). Then all $f$-preimages of $B$ are annuli that map to $B$ by unbranched covering maps of degrees at most $D$. One of them, say $B'$, is an essential sub-annulus in $A'$, and thus $\modulus(A') \geqslant \modulus(B') \geqslant \mu/D^2$.
\end{proof}

\begin{remark}
In fact, in the previous lemma one can prove the stronger bound $\modulus(A') \geqslant \mu/D$, see \cite[Lemma 4.5]{KL09}.
\end{remark}

An open topological disk $U$ in $\C$ is said to have \emph{$\eta$-bounded geometry} \cite{KSS} if it contains a Euclidean disk of radius $\eta \cdot \diam U$. The lemma above together with the Koebe Distortion Theorem implies the following fact that gives control over geometric shapes of disks under pullbacks (see, for example, \cite[Fact 6.2]{KSS}).

\begin{lemma}[Easy geometry control]
\label{Lem:Fact}
Let $f \colon U' \to V'$ be a branched covering of degree at most $D$ between open topological disks, and suppose $V \subset V'$ is a topological disk that has $\eta$-bounded geometry and $\modulus(V' \sm \ovl V) \ge \delta > 0$. Let $U$ be a component of $f^{-1}(V)$. Then $U$ has $\eta' = \eta'(\eta, \delta, D)$-bounded geometry. \qed
\end{lemma}

\begin{remark}
Let $U \subset \C$ be an open topological disk, and $x \in U$ be a point. Denote by $R(x)$  the radius of the smallest Euclidean disk centered at $x$ that contains $U$, and by $r(x)$ the radius of the largest Euclidean disk centered at $x$ that is contained in $U$. We say that $U$ has \emph{$C$-bounded shape with respect to $x$} if $R(x)/r(x) \le C$ \cite[$\S 4$]{LyuBook}. More generally, $U$ has \emph{$C$-bounded shape} if this is so with respect to some point in $U$. It is straightforward to see that $U$ has $\eta$-bounded geometry if and only if it has $1/\eta$-bounded shape.  
\end{remark}

Another application of the Koebe Theorem is the following.

\begin{lemma}[Easy diameter control]
\label{Lem:Diam}
Let $f \colon U \to V$ be a branched covering of degree at most $D$ between two open topological disks. Suppose $B \subset B' \subset V$ is a pair of open round disks of radii $0<r<r'$ and so that $\modulus (V \sm \ovl B') = \delta' > 0$, $\modulus(B' \sm \ovl B) = \delta > 0$. Let $Y'$ be a component of $f^{-1}(B')$ and $Y$ be a component of $f^{-1}(B)$ such that $Y \subset Y'$ (see Figure~\ref{Fig:DiamControl}). Then there exists $C = C(r,r',\delta,\delta',D)> 0$ such that
\[
C \cdot \diam Y' \le \diam Y \le C^{-1} \cdot \diam Y'.
\] 
\end{lemma}

\begin{figure}[htbp]
\begin{center}
\includegraphics[scale=.55, trim=0 0 0 0, clip]{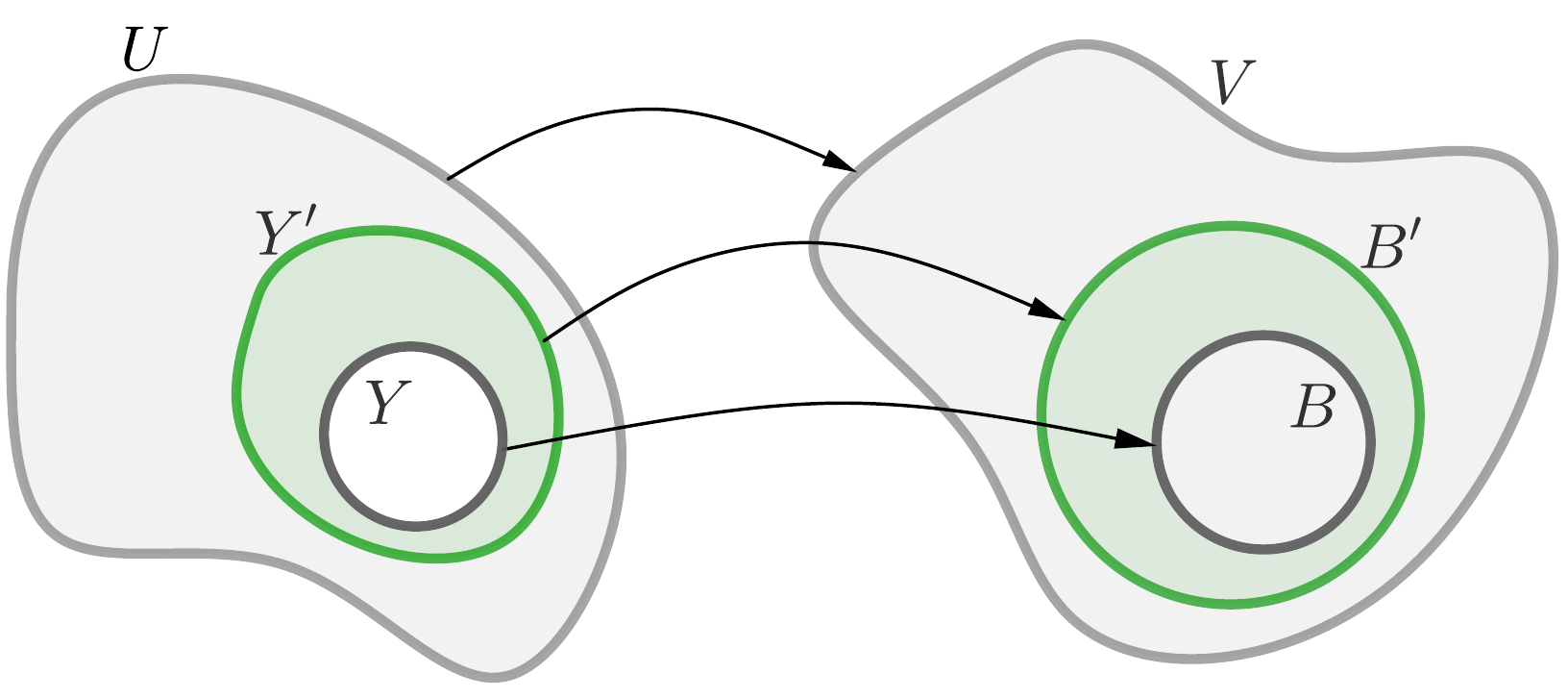}
\caption{An illustration for Lemma~\ref{Lem:Diam}.}
\label{Fig:DiamControl}
\end{center}
\end{figure}

\begin{proof}
The proof follows from~\cite[Lemma 11.2]{KvS09}.
\end{proof}

\subsection{Triviality of fibers of points accumulating on periodic fibers}
\label{SSec:AccumAtPer}

In this subsection we prove that fibers are trivial if they accumulate at periodic fibers under iteration. To do this, we need a refinement of Lemma~\ref{Lem:FirstTime} in order to gain some additional control on the degree of the first entry map to the union of puzzle pieces that contains all critical puzzle pieces. 

We will say that a fiber is \emph{critical} if it contains a critical point. 

\begin{lemma}[First entry to union of critical puzzle pieces has uniformly bounded degree]
\label{Lem:FirstTimeToUnion}
Let $(x_i)_{i \in I}$ be a finite set of points with distinct fibers that includes all critical fibers of $g$. Suppose that there exists a depth $m \ge 0$ so that all puzzle pieces $P_m(x_i)$ of depth $m$ are pairwise disjoint, and an integer $s > 0$ so that all  $(\inter P_m(x_i) \sm P_{m+s}(x_i))_{i \in I}$ are non-degenerate annuli. Then there is a constant $\mu>0$ with the following property: for every $y\in U$ for which there exists a $k\ge 0$ so that $g^k(y) \in \bigcup_i P_{m+s}(x_i)$, let $k=k(y)$ be minimal with this property; then there exists an essential open annulus $A \subseteq \inter P_{m+k}(y) \sm P_{m+s+k}(y)$ such that 
$\modulus(A) \geqslant \mu$.
\end{lemma}

\begin{proof}
Consider an arbitrary $y'\in U$ for which there exists a $k'\le s$ so that $g^{k'}(y')\in \bigcup_i P_{m+s}(x_i)$, and suppose again that $k'$ is minimal with this property, and that $y'$ is not on the boundary of a puzzle at any depth. To fix notation, suppose that $x_0$ is a point in $(x_i)_{i\in I}$ with $g^{k'}(y')\in P_{m+s}(x_0)$.

We claim that then the set $\inter P_{m+k'}(y') \sm P_{m+s+k'}(y')$ contains an annulus that separates $P_{m+s+k'}(y')$ from $\partial P_{m+k'}(y') $ and that has modulus bounded below.

We have $P_{m+s}(g^{k'}(y'))=P_{m+s}(x_0)$ and hence $\inter P_{m}(g^{k'}(y')) \sm P_{m+s}(g^{k'}(y')) = \inter P_{m}(x_0) \sm P_{m+s}(x_0)$, and by hypothesis this is a non-degenerate annulus of some modulus, say $\mu(x_0)>0$. 

Now we take a preimage of this annulus under $g^{k'}$. The map $g^{k'}$ sends $\inter P_{m+k'}(y')$ to the puzzle piece $\inter P_m(g^{k'}(y'))=\inter P_m(x_0)$ at depth $m$, and this is a branched cover of degree bounded in terms of $m$ and $k'\le s$ and the degrees of the critical points of $g$. 

 Therefore,  
\[
g^{-k'}\left( \inter P_m(g^{k'}(y'))\sm P_{m+s}(g^{k'}(y')) \right)\cap \inter P_{m+k'}(y' )
\] 
will in general not be an annulus, but an open disk with several closed disks removed. However, it does contain an annulus that separates $P_{m+s+k'}(y')$ from $\partial P_{m+k'}(y')$, and that is an essential annulus with modulus bounded below in terms of $\mu(x_0)$, $m$, $s$ and the degrees of the critical points of $g$. Since there are only finitely many $x_i$, this modulus is bounded below by a number $\mu>0$ that depends on $g$, $s$, $m$ and the set $\{x_i\}$, but not on $y'$.

Now consider a point $y$ for which there exists a minimal $k=k(y)$ as required in the lemma, and so that $y$ is not on the puzzle boundary at any depth. If $k \le s$, then $y$ is one of the $y'$ discussed earlier.

The proof for the case $k>s$ is similar to Lemma~\ref{Lem:FirstTime}. Again, consider the ``orbit of puzzle pieces'' $\left(P_{m+k-t} \left(g^t(y)\right)\right)_{t=0}^{k}$. For  $t<k$, the point $g^t(y)$ does not visit any critical puzzle piece of depth $m+s$. Since for $t<k-s$, the depth of the surrounding pieces  $P_{m+k-t} \left(g^t(y)\right)$ exceeds $m+s$, the entire pieces $P_{m+k-t} \left(g^t(y)\right)$ are non-critical. Therefore, the map
\[
g^{k-s}\colon \inter P_{m+k}(y)\to \inter P_{m+s}(g^{k-s}(y))
\]
is biholomorphic. In particular, $\inter P_{m+k}(y) \sm P_{m+s+k}(y)$ is conformally equivalent to 
\[
\inter P_{m+s}(g^{k-s}(y)) \sm P_{m+2s}(g^{k-s}(y))
\;.
\] 
The claim now follows from the first part, applied to $y'=g^{k-s}(y)$ and $k'=s$.
\end{proof}

\begin{lemma}[Accumulation at periodic fiber implies trivial fiber]
\label{Lem:AccumAtPeriod}
Let $g$ be a holomorphic map with well-defined Markov partition. Suppose that $z$ is a non-escaping point of $g$  so that the $\omega$-limit set of $z$ intersects the fiber $\fib(y)$ of some periodic point $y$ but the orbit of $z$ is disjoint from $\fib(y)$. Assume additionally that $\fib(y)$, as well as all those critical fibers of $g$ that intersect $\omega(z)$, are contained in the interiors of the corresponding puzzle pieces of any depth. Then $\fib(z) = \{z\}$. 
\end{lemma}

\begin{proof}
\setcounter{stepctr}{0}

Our proof goes along the lines of the proof of \cite[Lemma~3]{RY08}, except for the final step where Lemma~\ref{Lem:FirstTimeToUnion} will provide us with the suitable annuli to pull back. 

The proof itself may look a bit technical in notation, but the underlying idea is simple: as long as the orbit of $z$ stays sufficiently close to $\fib(y)$, that is in some fixed puzzle piece $P_{n_0}(y)$ that contains no further critical points other than those are already in $\fib(y)$, the puzzle pieces along this orbit are mapped forward injectively. When the orbit leaves $P_{n_0}(y)$ and later returns back ($z$ accumulates on $\fib(y)$ by hypothesis), it does so with uniformly bounded degree by Lemma~\ref{Lem:FirstTime}. This allows us, by pulling back suitable annuli (given by Lemma~\ref{Lem:FirstTimeToUnion}), to conclude that $\fib(z) = \{z\}$,  whether or not the fiber of $y$ is trivial. 

Up to passing to an iterate of $g$, assume that $y$ is a fixed point, and let us adopt the notation $f$ for this iterate of $g$; thus $f^k (P_{n+k}(y)) = P_n(y)$ for all $n$ and $k$.

Let $(\cfib_i)_{i \in I}$ be the set of all critical fibers of $f$, different from $\fib(y)$, that intersect $\omega(z)$ (if the fiber of $y$ is not critical, then this is just the set of all critical fibers of $f$ that intersect $\omega(z)$; here $I$ is some finite index set, which in the simplest case might be empty). For every critical fiber $\cfib_i$ pick a critical point $c_i \in \cfib_i$ representing this fiber (this choice might not be unique). Let us choose $n_0$ so that $P_{n_0}(y) \cap \Crit(f) \subset \fib(y)$ and $P_{n_0}(c_i) \cap \Crit(f) \subset \cfib_i$ for every $i \in I$; this is possible by definition of a fiber. By increasing $n_0$ if necessary, we also assume that $\orb(z)$ does not intersect any critical puzzle piece of depth $n_0$ except those around $c_i$, $i \in I$ and, possibly, $y$. Up to an index shift, assume $n_0=0$. Further on, fix a depth $s > 0$ such that all the annuli $\inter P_0 (y) \sm P_s(y)$ and $\inter P_0(c_i) \sm P_s(c_i)$ are non-degenerate. The depth $s$ exists by the assumption of the lemma: all the fibers $\fib(y)$ and $\cfib_i$ are contained in the interiors of the corresponding puzzle pieces of any depth. Define $\mathcal P := P_s(y) \sqcup \bigsqcup_{i \in I} P_s(c_i)$; this is the union of the puzzle pieces of depth $s$ containing $y$ and all critical fibers on which the orbit of $z$ accumulates.

For given $n$, let $k_n$ be the smallest integer such that $f^{k_n}(z) \in P_n(y)$; such an index exists because $z$ accumulates on $\fib(y)$. However, since the orbit of $z$ never enters $\fib(y)$ by hypothesis, there exists a smallest integer $m_n>n$ such that $f^{k_n}(z) \not \in P_{m_n}(y)$, hence $f^{k_n}(z)\in P_{m_n-1}(y)\sm P_{m_n}(y)$. Finally, let $l_n\ge 0$ be minimal so that $f^{k_n+m_n+l_n}(z)\in \mathcal P$; again, such an index exists because $z$ accumulates on $\fib(y)$, critical fibers $\fib(c_i)$, $i \in I$, and by the choice of what we call the zero depth puzzle pieces. However, it might happen that $f^{k_n+m_n+l_n}(z)$ lands not in $P_s(y)$ but in a critical puzzle piece in $\mathcal P$; denote by $c=c(n)$ the point from the set $\{y\} \cup \bigcup_{i \in I} \{c_i\}$ such that $f^{k_n+m_n+l_n}(z) \in P_s(c)$ (see Figure~\ref{Fig:AccumOnFiber} for a schematic drawing of the puzzle pieces involved).

\begin{figure}[htbp]
\begin{center}
\includegraphics[scale=.6, trim=90 35 30 35]{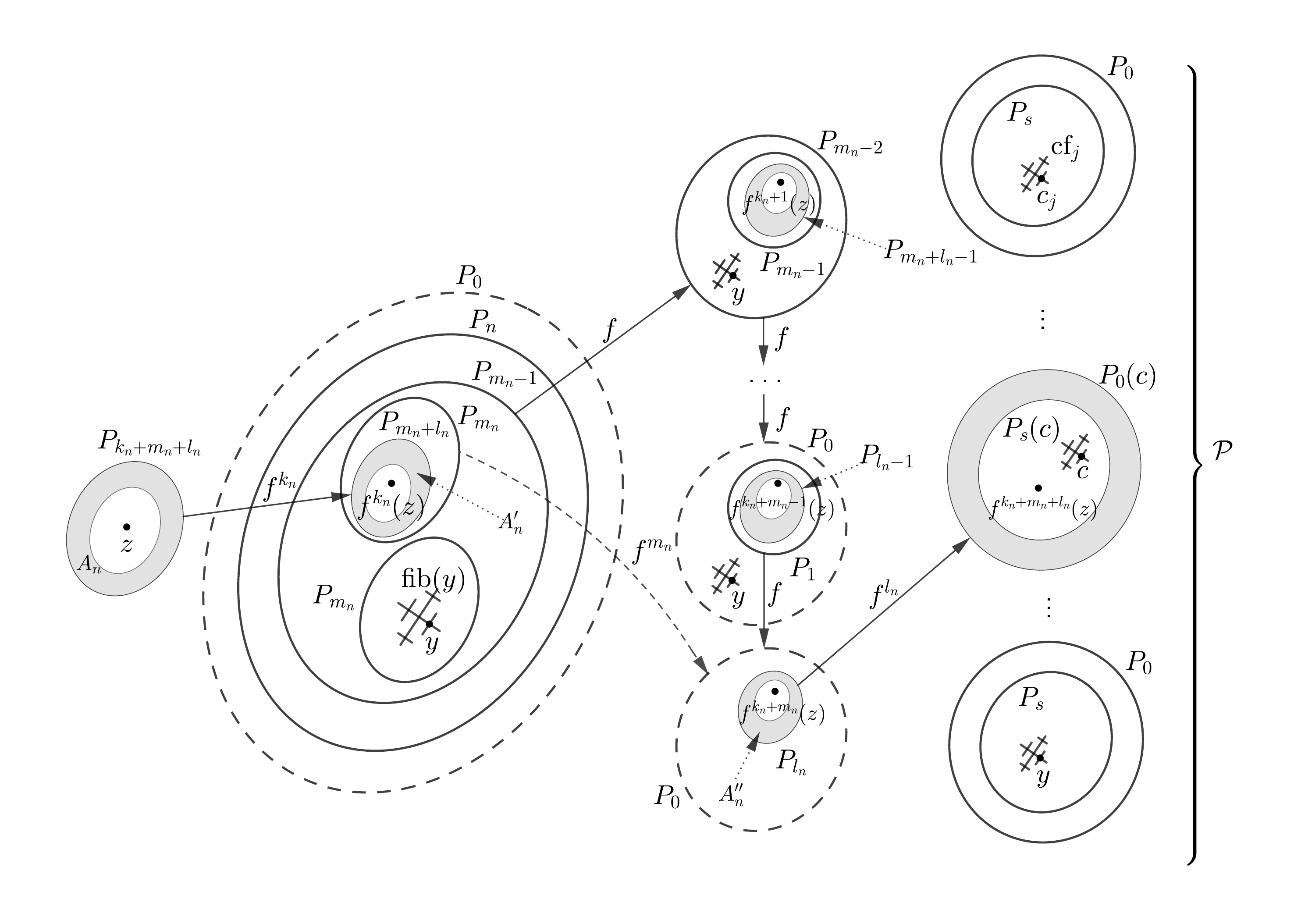}
\caption{Puzzle pieces of various depths involved in the proof of Lemma~\ref{Lem:AccumAtPeriod}. Observe that the fiber $\fib(y)$ may or may not belong to some of the puzzle pieces in the sequence $\left(P_{m_n - i}\left(f^{k_n+i}\left(z\right)\right)\right)_{i=1}^{m_n-1}$; the picture shows the case when it never happens. Moreover, the first landing of the orbit of $z$ after the time $k_n+m_n$ to the union $\mathcal P = P_s(y) \cup \bigcup_j P_s(c_j)$ may or may not be in the puzzle piece $P_s(y)$; the picture shows the case when the orbit of $z$ lands in some other puzzle piece $P_s(c) \in \mathcal P$. }
\label{Fig:AccumOnFiber}
\end{center}
\end{figure}

We claim that there exists an essential open sub-annulus 
\[
A_n \subset \inter P_{k_n+m_n+l_n}(z) \sm P_{k_n+m_n+l_n + s}(z)
\]
such that 
\begin{equation}
\label{Eq:ModBound}
\modulus (A_n) \geqslant \frac{\mu}{D^2}, 
\end{equation}
where $D$ is given by Lemma~\ref{Lem:FirstTime} and $\mu$ is given by Lemma~\ref{Lem:FirstTimeToUnion}, and hence the factor $\mu/D^2$ is independent of $z$ and $n$. We will do this in three steps; since we are pulling back, they come in reverse order. The third step is $f^{k_n}\colon P_{k_n + m_n+l_n}(z)\to P_{m_n+l_n}(f^{k_n}(z))$ controlled by Lemma~\ref{Lem:FirstTime}; the second step is a sequence of $m_n$ conformal iterates to $P_{l_n}(f^{k_n+m_n}(z))$; and in the first step this puzzle piece is sent by $f^{l_n}$ to $P_0(f^{k_n+m_n+l_n}(z))$, controlled by Lemma~\ref{Lem:FirstTimeToUnion} again. These three steps are illustrated in Figure~\ref{Fig:AccumOnFiber} (left, center, and right); the annuli we are pulling back are contained in the shaded rings. 

\Newpage

\begin{step}
\label{St:Ann1}
Since $l_n$ was chosen to be minimal so that $f^{k_n+m_n+l_n}(z)\in \mathcal P$, and $c$ is such that $f^{k_n+m_n+l_n}(z) \in P_s(c) \subset \mathcal P$, Lemma~\ref{Lem:FirstTimeToUnion} guarantees that there exists an essential open sub-annulus 
\[
A_n'' \subset \inter P_{l_n}\left(f^{k_n+m_n}(z)\right) \sm P_{l_n+s}\left(f^{k_n+m_n}(z)\right)
\]
such that
\begin{equation}
\label{Eq:FirstAnnulus}
\modulus(A_n'') \geqslant \mu,
\end{equation}
where $\mu$ does not depend on $z$ and $n$. (Strictly speaking, in order to apply Lemma~\ref{Lem:FirstTimeToUnion}, we have to enlarge $\mathcal P$ so that it would contain all critical puzzle pieces of depth $s$; but since, by construction, the orbit of $z$ visits only those critical puzzles already in $\mathcal P$, this enlargement of $\mathcal P$ does not alter the conclusion.)
\end{step}

\begin{step}
\label{St:Ann2}
We claim that there exists an open essential sub-annulus
\[
A_n' \subset \inter P_{m_n+l_n}\left(f^{k_n}(z)\right) \sm  P_{m_n+l_n + s}\left(f^{k_n}(z)\right)
\]
such that $A_n'$ is a conformal copy of $A_n''$, and hence
\begin{equation}
\label{Eq:SecondAnnulus}
\modulus(A_n') = \modulus(A_n'').
\end{equation}

We argue as follows. The puzzle pieces around $y$ are, as always, nested like $P_0(y)\supset P_1(y)\supset P_2(y)\supset\dots$, and since $y$ is a fixed point, each one is the image of the next one under $f$. Since $f^{k_n}(z)\in P_{m_n-1}(y)\sm P_{m_n}(y)$, all the points $f^{k_n}(z)$, $f^{k_n+1}(z)$, \dots, $f^{k_n+m_n}(z)$ are in $P_0(y)\sm P_{m_n}(y)$. But by construction all critical points in $P_0(y)$ are already in $\fib(y)$ and hence in $P_{m_n}(y)$. Therefore, for $i\in\{0,1,\dots,m_n-1\}$, the puzzle pieces of depth $m_n-i$ around $f^{k_n+i}(z)$ do not contain critical points. Together, this shows that the map $f^{m_n}\colon P_{m_n}\left(f^{k_n}(z)\right)\to P_0\left(f^{k_n+m_n}(z)\right)$ has degree $1$, and the same is true for its restriction $f^{m_n}\colon P_{m_n+l_n}\left(f^{k_n}(z)\right)\to P_{l_n}\left(f^{k_n+m_n}(z)\right)$, and hence the claim in {Step~\ref{St:Ann2}} follows with $A_n'$ as the conformal pull-back of $A_n''$ under this restricted map.
\end{step}

\begin{step}
\label{St:test}
Similarly to {Step~\ref{St:Ann1}}, since $k_n$ is the first iterate so that $f^{k_n}(z)\in P_n(y)$, the map $f^{k_n}\colon P_{k_n+n}(z)\to P_n(y)$ has degree at most $D$ by Lemma~\ref{Lem:FirstTime}. The same is then true for its restriction 
\begin{equation*}
f^{k_n}\colon P_{k_n+m_n+l_n}(z)\to P_{m_n+l_n}\left(f^{k_n}(z)\right).
\end{equation*}
We will construct an annulus $A_n$ in order to apply Lemma~\ref{Lem:AnnulusLemma} as follows: let $Y_1$ and $Y_2$ be the closed disks such that $A'_n=\inter Y_1\sm Y_2$, pick two more disks $Y'_{1}\subset P_{k_n+m_n+l_n}(z)$ and $Y'_{2}\supset P_{k_n+m_n+l_n+s}(z)$ such that $Y'_2\subset Y'_1$ and $f^{k_n}(Y'_{i})=Y_{i}$ for $i=1,2$, and set $A_n:=\inter Y'_1\sm Y'_2$. Then $A_n$ is an essential sub-annulus in $\inter P_{k_n+m_n+l_n}(z) \sm P_{k_n+m_n+l_n+s}(z)$. Then by Lemma~\ref{Lem:AnnulusLemma},
\begin{equation}
\label{Eq:ThirdAnnulus}
\modulus(A_n) \geqslant \frac{\modulus(A_n')}{D^2}. 
\end{equation}
\end{step}

Combining (\ref{Eq:FirstAnnulus}), (\ref{Eq:SecondAnnulus}) and (\ref{Eq:ThirdAnnulus}) we obtain (\ref{Eq:ModBound}).

This argument can be carried out for infinitely many $n$: we choose a sequence $n_j$ so that once $A_{n_{j-1}}$ is fixed, the value of $n_j$ is chosen so that $P_{k_{n_j}+m_{n_j}+l_{n_j}}(z)$ is contained in the bounded component of $\C\sm A_{n_{j-1}}$. This way, we obtain infinitely many disjoint annuli with moduli bounded below that all separate $z$ from all previous annuli, and using the standard Gr\"otzsch inequality this implies that the fiber of $z$ is trivial.
\end{proof}

\begin{corollary}[Accumulation at periodic fiber implies trivial fiber, revisited]
\label{Cor:AccumAtPeriodDeg}
Under the hypothesis of Lemma~\ref{Lem:AccumAtPeriod}, there exist an increasing sequence of integers $(\nu_j)_{j \ge 0}$ and two puzzle pieces $P_s, P_0$ with $\inter P_s \subset P_0$ such that $g^{\nu_j}(z) \in \inter P_s$ for every $j \ge 0$ and the degrees of the maps $g^{\nu_j} \colon \inter P_{\nu_j}(z) \to \inter P_0$ are uniformly bounded.
\end{corollary}

\begin{proof}
The claim follows from the proof of Lemma~\ref{Lem:AccumAtPeriod} with $(\nu_j)$ a subsequence of $(k_{n_j} + m_{n_j} + l_{n_j})$ (in the notation of the lemma) chosen so that the corresponding iterate of $z$ lands in the same puzzle piece $P_s \subset P_0$ for all $j$. Clearly, such a subsequence exists because we have a finite choice of ``target'' puzzle pieces. 
\end{proof}

\section{Complex box mappings and rigidity (Theorem~\ref{Thm:RigidityComplexBox})}
\label{Sec:ComplexBoxMappings}

In this section we review the notion of \emph{complex box mappings} introduced in \cite{KSS, KvS09} and prove a generalized version of triviality of fibers for such mappings (Theorem~\ref{Thm:RigidityComplexBox}). This result is of interest in its own right, and it is a key ingredient in the proof of our Dynamical Rigidity for Newton maps (Theorem~\ref{Thm:RRP}). In what follows, we use definitions and results from \cite{KvS09}, clarified and spelled out in~\cite{DKvS}. 

\begin{definition}[Complex box mapping]
\label{Def:BoxMap}
A holomorphic map $F \colon \mathcal U \to \mathcal V$ between two open sets $\mathcal U \subset \mathcal V \subset \Cc$ is a \emph{complex box mapping} if the following holds:
\begin{enumerate}
\item
\label{It:BM1}
$F$ has finitely many critical points;
\item
\label{It:BM2}
$\mathcal V$ is the union of finitely many open Jordan disks with disjoint closures, while $\U$ is the union of finitely or infinitely many open Jordan disks with disjoint closures;
\item
\label{It:BM3}
every component $W$ of $\mathcal V$ is either a component of $\mathcal U$, or $W \cap \mathcal U$ is a union of Jordan disks with pairwise disjoint closures, each of which is compactly contained in $W$;
\item
\label{It:BM4}
for every component $U$ of $\mathcal U$ the image $F(U)$ is a component of $\mathcal V$, and the restriction $F \colon U \to F(U)$ is a proper map.
\end{enumerate}
\end{definition}

\begin{figure}[htbp]
\begin{center}
\includegraphics[scale=.7, trim=60 35 30 35]{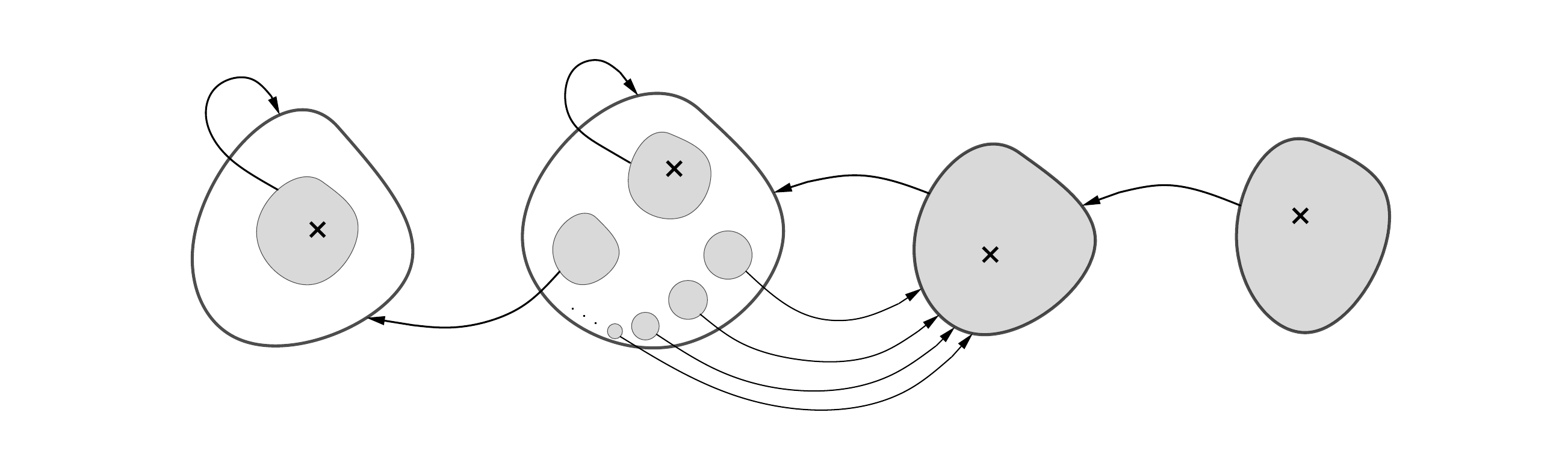}
\caption{An example of a box mapping $F \colon \U \to \V$. The components of $\U$ are shaded in grey, there might be infinitely many of those; the components of $\V$ are drawn with thick black boundary. The critical points of $F$ are marked with a cross.}
\label{Fig:ExampleBoxMapping}
\end{center}
\end{figure}

Following Douady and Hubbard~\cite{PolyLike}, a proper holomorphic map $f \colon U \to V$ of degree $d \geqslant 2$ between two open topological disks $U$ and $V$ with $\ovl U \subset V \subset \C$ is called a \emph{polynomial-like map}. By the straightening theorem, such a polynomial-like map is hybrid equivalent to a polynomial of degree $d$, and this polynomial is unique (up to affine conjugation) if the filled Julia set $K(f) := \bigcap_{n\geqslant 1} f^{-n}(V)$ is connected. Moreover, connectivity of $K(f)$ is equivalent to the condition that all critical points of $f$ are contained in $K(f)$. 

When $\mathcal V$ is connected and $\mathcal U$ has only finitely many components, and all of these are compactly contained in $\mathcal V$, then the corresponding box mapping may be regarded as a polynomial-like map in the sense of Douady--Hubbard (generalized to several components of $\U$). For general box mappings, however, $\mathcal U$ is allowed to have infinitely many components, and in many applications this is important. Such generality in the definition of a box mapping results in phenomena that do not occur for polynomial-like maps. For example, a box mapping might have wandering domains, or it might have a filled Julia set that is all of $\mathcal U$ (see~\cite{DKvS}). 

\begin{definition}[Puzzle piece of box mapping]
\label{Def:BoxPuzzle}
For a box mapping $F \colon \mathcal U \to \mathcal V$ and $n \ge 0$, we define a \emph{puzzle piece of depth $n$} to be the closure of a component of $F^{-n}(\mathcal V)$. 
A puzzle piece is called \emph{critical} if it contains at least one critical point.
\end{definition}

\begin{remark}
We keep our convention for puzzle pieces to be closed sets, following Yoccoz and Hubbard \cite{HY}. This is in contrast to~\cite{KvS09, KSS}, and also \cite{DKvS}, where puzzle pieces are defined to be open sets. The transition between the two conventions is straightforward: a puzzle piece in the sense of the cited papers is the interior of a puzzle piece in the sense of Definition~\ref{Def:BoxPuzzle}, and vice versa, the closure of a puzzle pieces in the sense of \cite{KvS09, KSS, DKvS} is a puzzle piece in our sense.   
\end{remark}

The set $K(F) := \{z \in \mathcal U \colon F^n(z) \in \mathcal U \text{ for all }n \geqslant 0\}$ is the \emph{filled Julia set} of the box mapping $F \colon \mathcal U \to \mathcal V$; this is the set of \emph{non-escaping points}. Similarly, the \emph{Julia set} is defined as $J(F) := \partial K(F) \cap \U$.

The following lemma describes an easy way how to construct a box mapping using the first return construction.
Recall that a set is nice if the forward orbit of its boundary does not intersect its interior. By the Markov property, every union of puzzle pieces of the same depth is a nice set in this sense. However, a union of puzzle pieces not necessarily of the same depth might or might not be nice.

\begin{lemma}[Simple construction of box mappings]
\label{Lem:Simple}
Let $g \colon U \to V$ be a holomorphic map with well-defined Markov partition as in Section~\ref{Sec:GeneralPuzzles}, and let $W$ be a disjoint union of the interiors of finitely many puzzle pieces. Suppose $W$ is nice. Let $\V$ be the union of $W$ and all the components of $\DomE(W)$ that intersect $\Crit(g)$, and let $F \colon \Dom(\V) \to \V$ be the first return map to $\V$. Then $F$ is a complex box mapping and $\Crit(F) \subset \Crit(g)$.
\end{lemma}

\begin{proof}
Let us show that $F$ satisfies Definition~\ref{Def:BoxMap}. First notice that $\V$ is a nice open set. Indeed, by definition of $\DomE(W)$ the orbit of the boundary of a component of $\V \sm W$ maps over the boundary of some component of $W$ without intersecting other components of $\V \sm W$; the claim then follows because $W$ is nice by hypothesis. Since $\V$ is nice, we can apply Lemma~\ref{Lem:FirstReturnConstruction}, and the properties \eqref{It:BM2}, \eqref{It:BM3}, and \eqref{It:BM4} of Definition~\ref{Def:BoxMap} follow automatically by that lemma.

In order to see that $F$ has only finitely many critical points (property~\eqref{It:BM1} of Definition~\ref{Def:BoxMap}), and those critical points are among the critical points of $g$, let us pick a critical component of $\Dom(\V)$ for $F$, say $Y$, and assume $F|_Y = g^k|_Y$ for some $k \ge 1$. Since $F$ is the first return map to $\V$, each critical point of $g$ can be seen at most once in the sequence $Y, g(Y), \ldots, g^{k-1}(Y)$ of open puzzle pieces (see Lemma~\ref{Lem:FirstTime}). Since $Y$ is a critical component for $F$, there exists $m \in \{0, \ldots, k-1\}$ so that $g^m(Y)$ contains a critical point of $g$; call this point $c$. The $g$-orbit of $c$ intersects $W$ and hence $c$ belongs to $\DomE(W)$. By definition of $\V$, this implies $g^m(Y) \subset \V$ and thus $m=0$. We conclude that each critical point of $F$ is a critical point of $g$.
\end{proof}

\begin{definition}[Box renormalizable box mappings]
\label{Def:RenComplexBoxMap}
We call a complex box mapping $F\colon \mathcal U\to\mathcal V$  \emph{box renormalizable around a critical point $c \in \mathcal U$} if there exists a puzzle piece $W$ at some depth containing $c$, and an integer $s \ge 1$ such that $F^{sk}(c') \in \inter W$ for every critical point $c'\in W$ and every $k \ge 0$; the minimal such $s$ is called the \emph{period of the renormalization}. The \emph{filled Julia set of this box renormalization} is defined analogously as $\left\{z\in\mathcal U\colon F^{sk}(z)\in \inter W \text{ for all }k\ge 0\right\}$. In this situation we call $c$ a \emph{box renormalizable critical point}. 

A complex box mapping $F \colon \mathcal U \to \mathcal V$ is called \emph{box renormalizable} if it is box renormalizable around at least one critical point in $\mathcal U$, and \emph{non-box renormalizable} otherwise.
\end{definition}

\begin{definition}[Renormalizable box mappings]
\label{Def:RenComplexBoxMapDH}
We call a box mapping $F\colon\U\to\V$ \emph{renormalizable around a critical point $c\in\U$} if there exists a puzzle pieces $W$ and an integer $s\ge1$ such that $F$ is box renormalizable around $c$ with period $s$ in the sense of Definition~\ref{Def:RenComplexBoxMap}, and such that there exists a puzzle piece $Y\subset \inter W$ with $F^s(Y)=W$, and $F^{sk}(c') \in \inter Y$ for all critical points $c'$ in $\inter Y$, and for all $k\ge 0$. 
\end{definition}

When we speak of renormalizable box mappings, we mean Definition~\ref{Def:RenComplexBoxMapDH} unless explicitly speaking of ``box renormalization''.

\begin{remark}
In the case that a puzzle piece $W_0$ contains several critical points among which some have their entire $F^s$-orbits in $W_0$ and others do not, then one can shrink $W_0$ to a puzzle piece $W$ of greater depth that contains only those critical points that do not escape, and then $F$ is renormalizable around these critical points.
\end{remark}

\begin{remark}
Let us clarify the difference between these two definitions. If $W$ is a puzzle piece of some depth $n$ so that $F$ is box renormalizable around $c$ with period $s$, and $Y$ is the puzzle piece of depth $n+s$ containing $c$, then either $Y = W$ or $Y \subset \inter W$ by the puzzle set-up. Both cases are allowed for box renormalization according to \cite[Definition 1.3]{KvS09}. In the second case, the restriction $F^s\colon \inter Y \to \inter W$ is a polynomial-like map in the sense of Douady--Hubbard. It follows that all critical points $c'\in Y$ have $F^{sk}(c') \in \inter Y$ for all $k\ge 1$, and hence the filled Julia set of the renormalization around $c$ is connected by the standard theory of polynomial-like maps mentioned above.

If $Y = W$, then $F^s \colon \inter Y \to \inter W$ is a proper self-map of a disk without escaping points (hence an \eqref{Item:NE} component), and thus the filled Julia set of $F^s$, restricted to $Y$, is equal to $Y$. This is a situation that may occur for box renormalizable maps, but it is not included in Definition~\ref{Def:RenComplexBoxMapDH}. It does not occur in a number of interesting cases arising from dynamics on $\Cc$. We will thus introduce, in Definition~\ref{Def:NaturalBoxMapping}, the notion of \emph{dynamically natural box mappings} for which this is excluded.
\end{remark}

The following lemma relates fibers of renormalizable critical points to little filled Julia sets of the corresponding renormalizations. 

\begin{lemma}[Renormalizable fibers equal little filled Julia sets]
\label{Lem:FibersJulia}
If $c$ is a renormalizable critical point of a complex box mapping $F$, and $\rho:= F^s \colon \inter Y \to \inter W$ is the corresponding polynomial-like renormalization of $F$, then the filled Julia set of $\rho$ is equal to $\fib(c)$.
\end{lemma}
\begin{proof}
Since by our definition of renormalization the filled Julia set $K(\rho)$ of $\rho$ is connected (see the remark above), $K(\rho)$ is equal to the nested intersection of open disks $\rho^{-n}(\inter W)$. As $Y \subset \inter W$, the same is true for the closures of these disks. Each of these closed disks is a puzzle piece of $F$ because $\rho$ is a restriction of $F$ and $Y$, $W$ are puzzle pieces of $F$. Hence $\fib(c)$, as a nested intersection of all puzzle pieces containing $c \in K(\rho)$, is equal to $K(\rho)$.
\end{proof}

\begin{lemma}[Fibers contained in puzzle interior]
\label{Lem:FibersCompactlyContained}
Consider a box mapping $F\colon \U\to\V$ and a non-escaping point $z$. If the orbit of $z$ does not eventually land in an \eqref{Item:NE} component, then the fiber of $z$ is contained in the interior of every puzzle piece it is contained in. 
\end{lemma}
\begin{proof}
It suffices to prove that the interior of every puzzle piece around $z$ contains another puzzle piece around $z$ at greater depth.   To do this, let $W_n$ be the puzzle piece of some depth $n$ containing $z$ and for $k\le n$ denote by $W_{n-k}=F^k(W_n)$ the puzzle piece around $F^k(z)$ at depth $n-k$, so that $W_{0}$ is a component of $\U$.

If any $W_{n-k}$ contains a puzzle piece around $F^k(z)$ at greater depth than $n-k$ that is a proper subset, then this proper subset must  be contained in $\inter W_{n-k}$, and the claim follows by pull-back to $W_n$. Otherwise, in particular $W_0$ is not only a component of $\U$ but also a component of $\V$. As we iterate forward, we cannot keep visiting components of $\U$ that are also components of $\V$ (by finiteness of $\V$ this would yield a cycle of \eqref{Item:NE}-components which is excluded by hypothesis), so we must reach a component of $\U$ that is compactly contained in its component of $\V$, and the claim follows.
\end{proof}

From Lemma~\ref{Lem:FibersCompactlyContained} it follows that if a box mapping $F$ has no \eqref{Item:NE} components, then every component of $K(F)$ is compact and $J(F) = \partial K(F)$. This is of course what is expected in dynamics, and it is the case for dynamically natural box mappings (see Definition~\ref{Def:NaturalBoxMapping}). 

This is true, in particular, for non-renormalizable box mappings (which are, in turn, non-box renormalizable), as well as for the box mappings that we will extract in Section~\ref{Sec:RRPProof} for Newton maps. 

We will be using one of the rigidity theorems by Kozlovski and van Strien. This result plays a crucial role in their study of rigidity for multicritical complex \cite{KvS09} and real \cite{KSS} polynomials (see also \cite{KL09} for the original proof of \emph{the Kahn--Lyubich Covering Lemma}, a crucial technical ingredient used to obtain the all-important complex bounds in the complex case). In order to quote it, we need some additional notation.

For a point $x \in \U$, set $m_F(x) := \modulus (\inter P_0(x) \sm P_1(x))$. Define 
\begin{equation*}
\begin{aligned}
K_\delta (F) &:= \left\{y \in K(F) : \limsup_{k \ge 0} m_F(F^k(y)) > \delta \right\} \text{ and }  
J_\delta(F) := \partial K_\delta (F) \cap \U; \\
&\Kwi(F) := \bigcup_{\delta>0} K_\delta(F) \text{ and } \Jwi(F) := \bigcup_{\delta>0} J_\delta(F).
\end{aligned}
\end{equation*}
The set $\Kwi(F)$ is invariant under $F$ and consists of all points in $K(F)$ whose orbits visit from time to time some components of $\U$ that are ``well inside'' the corresponding components of $\V$. In other words, a point $x \in K(F)$ fails to lie in $\Kwi(F)$ if the orbit of $x$ converges to the boundary of $\U$ in such a way that the modulus of the largest annulus within $\V$ around the components of $\U$ through which the orbit travels goes to zero. Clearly, $\Kwi(F)$ does not contain points that are mapped to an \eqref{Item:NE} component.  

Using this notation, there is the following rigidity result for complex box mappings from \cite{DKvS} (rephrased using the language of fibers):

\begin{theorem}[Dynamical rigidity for complex box mappings]
\label{Thm:RigidityJo}
If $F \colon \mathcal U \to \mathcal V$ is a non-renormalizable complex box mapping without \eqref{Item:NE} components and so that all periodic points are repelling, then each point in its Julia set has trivial fiber or converges to the boundary of\/ $\mathcal U$. In particular, 
each point in $\Jwi(F)$ has trivial fiber. \qed
\end{theorem}

Now we are able to prove the first of our three main results, Theorem~\ref{Thm:RigidityComplexBox} which claim that for any box mapping and an arbitrary point $z$, at least one of the following holds: $z$ has trivial fiber, it is renormalizable, it converges to the boundary, or it lands in an \eqref{Item:NE} component. 


\begin{proof}[Proof of Theorem~\ref{Thm:RigidityComplexBox}] 

Let $F \colon \U \to \V$ be the given box mapping, and $z \in K(F)$ be a point in the non-escaping set of $F$. If the orbit of $z$ lands in one of the \eqref{Item:NE} components of $F$, or if it converges to the boundary, then we are in case \eqref{Item:NE} or \eqref{Item:CB} of the theorem, and we are done. Excluding these two possibilities, we will show that either the orbit of $z$ lands in a renormalizable fiber (case \eqref{Item:R}), or $\fib(z)$ is trivial (case \eqref{Item:T}). 

If $F$ has at least one renormalizable critical point $c$, and the orbit of $z$ lands in $\fib(c)$ (which is the filled Julia set of the corresponding renormalization by Lemma~\ref{Lem:FibersJulia}), then we are in case \eqref{Item:R} of the theorem. If the orbit accumulates on $\fib(c)$ but does not land there, then we are in case \eqref{Item:T}: since $\orb(z)$ and hence $\omega(z)$ are disjoint from \eqref{Item:NE} components, the fiber of each point in $\omega(z)$ is contained in the interior of every puzzle piece in which it is contained (by Lemma~\ref{Lem:FibersCompactlyContained}), 
so by Lemma~\ref{Lem:AccumAtPeriod} it follows that $\fib(z) = \{z\}$.

Otherwise, defining $\mathcal C \subset \Crit(F)$ as the set of all renormalizable critical points, as well as those critical points whose orbits land in renormalizable fibers, there is a depth $s \ge 1$ such that the orbit of $z$ is disjoint from the set $B := \bigcup_{c \in \mathcal C} P_s(c)$. Since $\orb(z)$ does not converge to the boundary of $\U$, it must have at least one accumulation point $x \in \U$ such that $P_s(x)$ is disjoint from $B$. By choosing $r > s$, we can assure that the orbits of the critical points in $\mathcal C$ are disjoint from $P_r(x)$.

Write $V := \inter P_r(x)$ and construct a box mapping as in Lemma~\ref{Lem:Simple}: define $\V'$ to be the union of $V$ and all the components of $\DomE(V)$ intersecting $\Crit(F)$, set $\U' := \Dom(\V')$, and let $F' \colon \U' \to \V'$ be the first return map to $\V'$. By Lemma~\ref{Lem:Simple}, $F'$ is a complex box mapping. Since $F'$-puzzle pieces are also $F$-puzzle pieces, the notions of fiber for $F$ and for $F'$ coincide for $K(F')$.

By Lemma~\ref{Lem:FirstReturnConstruction}~\eqref{It:FRM4}, since $F'$ is a first return map, the $F$-orbit of $z$ intersects $K(F')$. Moreover, $\Crit(F') \subset \Crit(F)$ (Lemma~\ref{Lem:Simple}), and the choice of $r$ implies that all the critical points of $F'$ are non-renormalizable.

The box mapping $F'$ has no \eqref{Item:NE} components, it is not renormalizable and has only repelling periodic points: the latter property follows because every sufficiently deep puzzle piece around a non-repelling periodic point gives rise to a polynomial-like map of degree at least $2$, and this requires a renormalizable critical point. 
Hence, by Theorem~\ref{Thm:RigidityJo}, each point in $J(F')$ either has trivial fiber, or converges to $\partial \U'$. Since $F^k(z) \in K(F') = J(F')$ for some $k$, and since we have already excluded case~\eqref{Item:CB}, we conclude that $\fib(F^k(z)) = \{F^k(z)\}$, and the same is true for $z$. \end{proof}

We end this section by quoting some further definitions from~\cite{DKvS}. We will use these notions in later sections.

For a complex box mapping $F \colon \U \to \V$, let $A \subset K(F)$ be a finite set and $W$ be a union of finitely many open puzzle pieces of $F$. We say that $W$ is an \textit{open puzzle neighborhood} of $A$ if $A \subset W$ and each component of $W$ intersects the set $A$.

Define $\Koc(F)$ to be the set of all points $x \in K(F)$ such that the orbit of $x$ is disjoint from some open puzzle neighborhood of $\Crit(F)$. Similarly, define $\Joc(F) = J(F) \cap \Koc(F)$. 

\begin{definition}[Dynamically natural complex box mapping]
\label{Def:NaturalBoxMapping}
A complex box mapping $F \colon \U \to \V$ is \emph{dynamically natural} if it satisfies the following assumptions:
\begin{enumerate}
\item
\label{It:Nat1}
$F$ does not have components of \eqref{Item:NE} type;
\item
\label{It:Nat2}
$K(F) = \Kwi(F)$;
\item
\label{It:Nat3}
the Lebesgue measure of the set $\Koc(F)$ is zero.
\end{enumerate}
\end{definition}

As a heuristic principle, the box mappings that arise \emph{naturally} as restrictions of some globally defined rational maps are dynamically natural in the sense of the definition above. This is the case for box mappings constructed for complex polynomials  in \cite{KvS09} and \cite{ALdM}. In Section~\ref{Sec:RRPProof}, we show that the box mappings that are induced by the Newton dynamics are also dynamically natural. We use this fact in Section~\ref{Sec:ProofParamRigidity} to conclude parameter rigidity for Newton maps.

Dynamically natural box mappings may have orbits that converge to the boundary, i.e.\ of \eqref{Item:CB} type. However, condition~\eqref{It:Nat2} imposes a restriction on those orbits: the convergence to the boundary cannot be ``sudden''.

\section{Puzzles for Newton maps}
\label{Sec:Puzzles}

\subsection{Construction of Newton puzzles.}
\label{SSec:ConstructionOfNewtonPuzzles}

The proof of the Dynamical Rigidity for Newton maps (Theorem~\ref{Thm:RRP}) will rely on the puzzle construction for Newton maps introduced in \cite{DLSS}. We will review the key steps of this construction, together with some of its properties that we will use in the further sections.

In the polynomial case, puzzles were constructed by Branner--Hubbard  and Yoccoz (with much further work since then) starting with neighborhoods $S_n$ of the filled-in Julia set that are bounded by suitably chosen equipotential curves. These are subdivided by finitely many pairs of dynamic rays landing at common repelling periodic or preperiodic points (for quadratic polynomials, usually at the $\alpha$-fixed point and its iterated preimages). The closures of complementary components of $\ovl{S_n}$ minus the ray pairs are called puzzle pieces of depth $n$) and they form a Markov partition. 

However, for rational maps it is not at all clear how to carry over such a construction; in particular, there are no obvious substitutes for the basin of infinity and the B\"ottcher coordinates available in the polynomial case that give rise to good Markov partitions.


A notable exception are Newton maps of polynomials. For these, puzzles with similar properties as in the polynomial case have recently been constructed in \cite[Theorem~B]{DLSS}. This result will be one of the main ingredients in the construction leading to our Theorem~\ref{Thm:RRP}.

\begin{theorem}[Newton puzzles for Newton maps of polynomials]
\label{Thm:NewtonPuzzles}
Every Newton map $N_p$ has an iterate $g = N_p^M$ for which there exists a finite graph $\Gamma \subset \Cc$  that is $g$-invariant (except possibly in a Fatou neighborhood of the roots), and so that for every $n \geqslant 0$  the complementary components of $g^{-n}(\Gamma)$ that intersect the Julia set are Jordan disks that satisfy the Markov property under $g$. These disk components define a Newton puzzle partition of depth $n$. \qed
\end{theorem}

We will unwrap they key steps in Theorem~\ref{Thm:NewtonPuzzles}, in particular, how we construct $\Gamma$. But first let us discuss the possible exception to forward invariance of $\Gamma$: every root has a compact and forward invariant neighborhood within its Fatou component (the immediate basin) in which $\Gamma$ may fail to be  invariant. 

It turns out that it  is much more convenient to work with an important and only mildly restricted class of Newton maps:


\begin{definition}[Attracting-critically-finite Newton maps]
\label{Def:AttractingCriticallyFinite}
A polynomial Newton map is called \emph{attracting-critically-finite} if all critical points that converge to a root actually land on the root after finitely many iterations.
\end{definition}

Theorem~\ref{Thm:NewtonPuzzles} can be strengthened for attracting-critically-finite Newton maps: in this case, the graph $\Gamma$ can be chosen to be $g$-invariant without exception (see \cite[Theorem~3.10]{DLSS}).  

The reason that attracting-critically-finite Newton maps are no serious restriction is that, by the means of a standard quasiconformal surgery, every polynomial $p$ has an associated polynomial $\tilde p$ so that the associated Newton maps $N_p$ and $N_{\tilde p}$ are related as follows \cite[Proposition 2.8]{DLSS}: there exists a quasiconformal homeomorphism $\tau \colon \Cc \to \Cc$ with $\tau(\infty)=\infty$ such that
\begin{enumerate}
\item
$N_{\tilde p}$ is attracting-critically-finite;
\item
$\deg p\ge \deg N_p = \deg N_{\tilde p}=\deg \tilde p $;
\item
$\tau$ conjugates $N_p$ and $N_{\tilde p}$ in neighborhoods of the Julia sets of $N_p$ and $N_{\tilde p}$ union all Fatou components (if any) that do not belong to the basins of roots; moreover, $\tau$ has non-zero dilatation only in the basin of the roots, but not on the Julia set (if it has positive measure) or on Fatou components away from the root basins.
\end{enumerate}

Recall that the \emph{immediate basin} $U_\xi$ of a root $\xi$ is the component of the basin of $\xi$ that contains $\xi$. The quasiconformal surgery that turns $N_p$ into $N_{\tilde p}$ is described in \cite{DMRS}; the idea is as follows. For every root $\xi$, the immediate basin $U_\xi$ is simply connected \cite{Pr}, and the restriction of $N_p$ to $U_\xi$ is a proper self-map of some degree $k=k(\xi)\ge 2$. The quasiconformal surgery consists of replacing a disk neighborhood of the root within $U_\xi$ by a disk with dynamics $z\mapsto z^k$; the degree of the self-map on $U_\xi$ is unchanged by this procedure, but afterwards there is a single critical point in $U_\xi$, and it is a fixed point (hence a simple root of $\tilde p$). The degree of $N_p$ will drop if $\xi$ was a multiple root of $p$. Moreover, the dynamics in the preimage components of $U_\xi$ is adjusted so as to make sure that all critical points in any preimage component coincide and land exactly on the root. All this can be accomplished by a quasiconformal surgery within a compact subset of finitely many components of the basin. 

The condition that $N_{\tilde p}$ is attracting-critically-finite assures that the dynamics of $N_{\tilde p}$ restricted to the basins of the roots is postcritically finite, while keeping the dynamics elsewhere unchanged (there may for instance still be critical points with dense orbits in the Julia set). This condition implies that the roots of $\tilde p$ are simple, so they are critical fixed points of $N_{\tilde p}$, and these are the only critical points in the immediate basins (possibly of higher multiplicities as critical points, but not as roots). Condition (3) implies that the dynamics of fibers of points in the Julia sets are the same, up to quasiconformal conjugation; in particular, triviality of fibers is preserved.

It is known that for every immediate basin $U_\xi$ the boundary point $\infty$ is always accessible through one or several accesses that are invariant up to homotopy; this number of invariant accesses equals $k(\xi)-1$, i.e.\ is one less than the degree of $N_p$ as a self-map of the immediate basin of the root \cite[Proposition 6]{HSS}. For the modified Newton map $N_{\tilde p}$, the accesses to $\infty$ are in fact invariant as curves, without need for a homotopy; this is the point of the surgery. Since $\tau$ must map basins and in particular immediate basins to basins and immediate basins, and it is a conjugation in a neighborhood of $\infty$, it must respect accesses to $\infty$ up to homotopy within immediate basins, and in particular the circular order at $\infty$ of these accesses (quasiconformal homeomorphisms preserve orientation). 

From now on we will assume that our Newton maps are attracting-critically-finite. This property allows us to define the basic combinatorial object associated to a Newton map, the channel diagram and eventually fibers. By definition (see \cite{HSS, DMRS} for a detailed discussion), the \emph{channel diagram} is a finite topological graph $\Delta$ such that its vertex set consists of all fixed points of $N_p$ (that is, $\infty$ and the roots of $p$), and each edge  of $\Delta$ is invariant under the dynamics and connects $\infty$ to some root within the respective immediate basin. The union of the interiors of all  such fixed rays across all immediate basins gives the edge set of $\Delta$. 
 
By construction, the graph $\Delta$ is invariant (as a set) under $N_p$; it encodes the mutual locations of the immediate basins. The \emph{Newton graph at level $n$} (denoted by $\Delta_n$) is the connected component of $N_p^{-n}(\Delta)$ containing $\infty$. Each edge of the Newton graph is an iterated preimage of a fixed ray in an immediate basin, and hence the Newton graph intersects the Julia set of $N_p$ only at $\infty$ and its iterated preimages, and intersects the Fatou set of $N_p$ along basins of roots. By construction, $\Delta_n \subset \Delta_{n+1}$ for all $n \ge 0$.
 
The Newton graph at level $n$ is the foundation for the definition of the puzzle of depth $n$. The construction rests on the following theorem, which is a compilation of two results, \cite[Theorem 3.4]{DMRS} and \cite[Proposition 3.1]{DLSS}.

\begin{theorem}[Connectivity properties of Newton graphs]
\label{Thm:MagicLevels}
Let $N_p$ be an attracting-critically-finite Newton map; then
\begin{enumerate}
\item
\label{Thm:PolesInGraph}
there exists a least integer $N$ so that $\Delta_N$ contains all poles of $N_p$;
\item
\label{Prop:Circles}
there exists a least integer $K > N$ so that for every component $V$ of $\Cc \sm \Delta$ there exists a topological circle $X_V \subset \Delta_K \cap \overline V \cap \C$ that passes through all finite fixed points in $\partial V$, separates $\infty$ from all critical values of $N_p$ in $V$, and does not contain a point on a critical orbit, except the roots (see Figure~\ref{Fig:Circles}). 
\qed
\end{enumerate}
\end{theorem}

\begin{figure}[htbp]
\begin{center}
\includegraphics[scale=0.43]{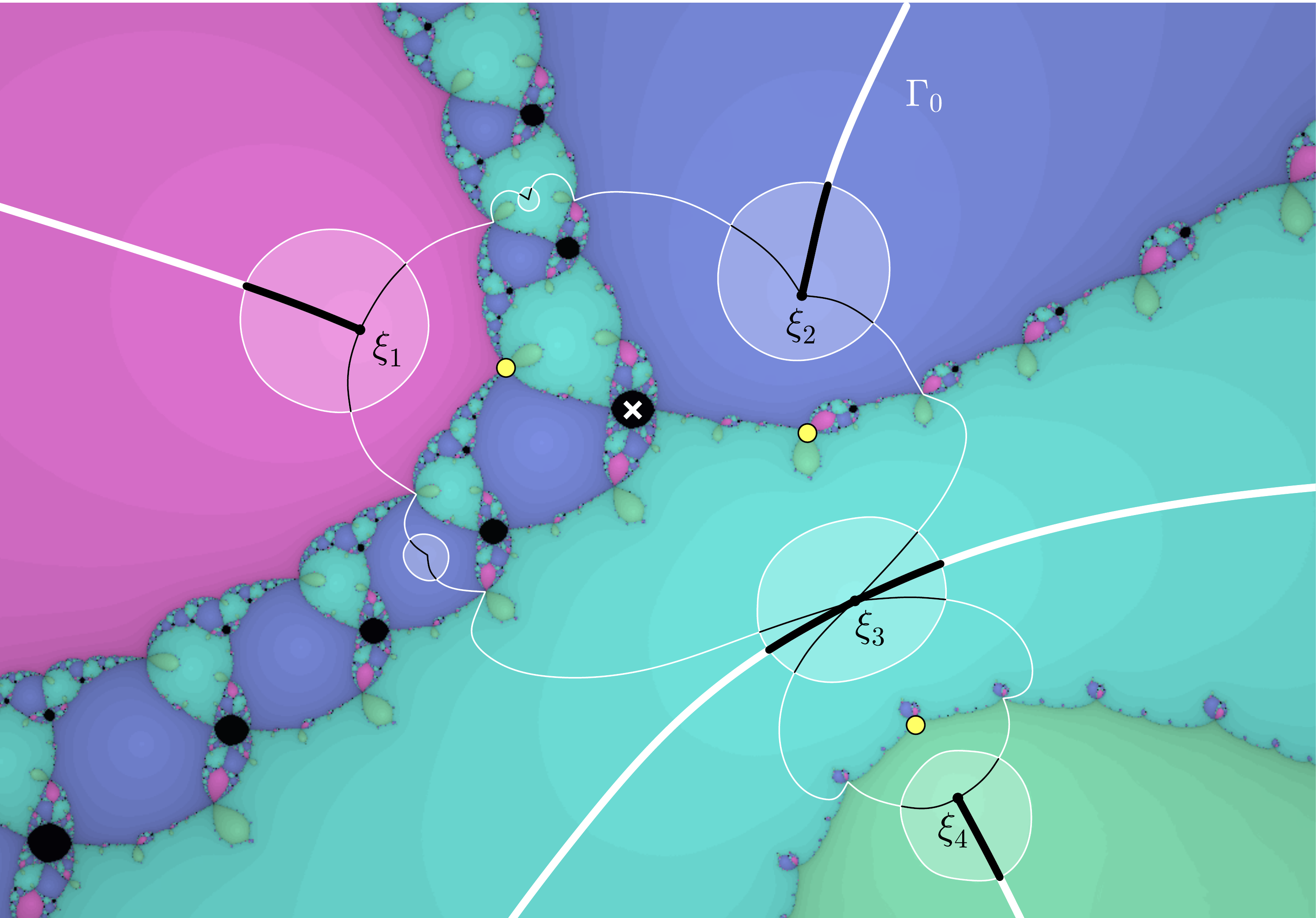}
\caption{Example of the dynamical plane of a degree 4 Newton map. The thick lines (in white and black) form the channel diagram $\Delta$; it connects the roots $\xi_1, \ldots, \xi_4$ to $\infty$. The three yellow dots mark the poles, and the white cross marks the ``free'' critical point: in this example, $\xi_1$, $\xi_2$, and $\xi_4$ are simple critical points, $\xi_3$ is a critical point of multiplicity two, it is responsible for two accesses to $\infty$ in $U_{\xi_3}$, and the remaining ``free'' critical point is renormalizable. The lighter-colored disks around the roots are bounded by suitably chosen equipotentials in the basins of roots. The topological circles passing through the roots, each in its own connected component of $\Cc \sm \Delta$, are shown in thin lines (black within the light disks around roots and their preimages, white otherwise). The set $S_0$ is the complement in $\Cc$ of the closed light disks. The Newton puzzle boundary of depth $0$, denoted $\Gamma_0$, consists of all white lines  (the disk boundaries, the rays connecting them to $\infty$, and the topological circles constructed in Theorem~\ref{Thm:MagicLevels}~\eqref{Prop:Circles}, except the parts within the disks around the roots and their preimages).}
\label{Fig:Circles}
\end{center}
\end{figure}

Finally, we are ready to complete the explanation of Theorem~\ref{Thm:NewtonPuzzles} and present the construction of puzzles that was carried out in \cite[Section 3]{DLSS}. Our puzzles will be defined for a suitable iterate of the Newton map $N_p$. Let $X := \bigcup_V X_V$, where the union is taken over all components $V$ of $\Cc \sm \Delta$ of the circles described in Theorem~\ref{Thm:MagicLevels}~(\ref{Prop:Circles}). Since $X \subset \Delta_K$, it follows that $N_p^K(X)=\Delta$, and $K$ is minimal with this property.

Define $M := N - 1 + K$. By Theorem~\ref{Thm:MagicLevels}~(\ref{Thm:PolesInGraph}), $M$ is the smallest index so that $\Delta_M$ contains all prepoles of level $K$ (where a point $z$ is called a \emph{(pre-)pole of level $n > 0$} if $n$ is minimal such that $N_p^{n}(z) = \infty$; with this definition, $K$ equals to the largest level of (pre-)poles in $X$. Indeed, since $\infty \in \Delta_N$, every component of $N_p^{-1}(\Delta_N)$ contains a pole. Hence, $\Delta_{N+1} = N_p^{-1}(\Delta_N)$ because $\Delta_N \subset \Delta_{N+1}$ and $\Delta_N$ contains all the poles. Therefore, $\Delta_{N+1}$ contains all the prepoles of level 2. Proceeding inductively, we see that $\Delta_M$ contains all prepoles of level $K$, and $M$ is the smallest with such property (as $N$ was also the smallest). Note that $X \subset \Delta_M$ because $X \subset \Delta_K$ and $\Delta_K \subset \Delta_M$. 

Write $g:= N_p^M$ for the $M$-th iterate of $N_p$. This is the iterate for which we construct Newton puzzles. Our basis is the Newton graph $\Delta_n$: the components of $\Cc \sm \Delta_n$ have the Markov property (Definition~\ref{Def:Markov}); this is discussed in \cite[Section 3]{DLSS}. However, and this is the main technical difficulty, these components do not necessarily have Jordan boundaries: some (pre-)poles on the boundary can be accessible in more than one way, like the point $\infty$ for some immediate basins. It turns out that this problem can be remedied by adding $X$ to $\Delta$ and pulling back; passing to the iterate $g$ is required to make the resulting graphs connected and forward invariant. The details are as follows. 

Define $\Delta_0^+ := \Delta \cup X$, and for $n>0$ let $\Delta_n^+$ be the component of $N_p^{-n}(\Delta_0^+)$ containing $\infty$; this is in analogy to the construction of $\Delta_n$. By \cite[Lemma 3.6~(3)]{DLSS},
\begin{equation}
\label{Eq:PuzzleLevels1}
g^{-1}(\Delta_{nM}^+) = N_p^{-M}(\Delta_{nM}^+) = \Delta_{(n+1)M}^+ \text{ for all }n \geqslant 2.
\end{equation}

We set $\Gamma_0' := \Delta_{2M}^+$, and similarly $\Gamma_n' := \Delta_{(n+2)M}^+$ for all $n > 0$. In this notation, Property~(\ref{Eq:PuzzleLevels1}) transforms to 
\begin{equation}
\label{Eq:PuzzleLevels2}
g^{-1}\left(\Gamma_n'\right) = \Gamma_{n+1}' \supset \Gamma_n' \text{ for all } n \geqslant 0,
\end{equation}
where the last inclusion follows from \cite[Lemma 3.6 (2)]{DLSS}. 

In order to construct puzzles, pick an equipotential curve in each of the immediate basins of roots, for some value of the potential, and consider the disks around the roots bounded by these curves. Define $S_0$ as the unique unbounded component in the complement of the union of these disks and all their pullbacks under $N_p$ to the components of root basins that intersect $\Gamma_0'$  (see Figure~\ref{Fig:Circles}). For a given $n \ge 1$, inductively define $S_n$ as the unbounded component of $g^{-1}(S_{n-1})$. Since the roots of $p$ are superattracting fixed points of $g$, we have $S_0 \supset S_1 \supset \ldots \supset S_n \supset \ldots$. Moreover, $\bigcap_{n \ge 0} S_n$ is the complement of the root basins, and hence it contains only those critical points of $g$ that are not mapped to the roots of $p$  by some iterate of $g$ (or equivalently, of $N_p$). 

\begin{definition}[Newton puzzles]
\label{Def:NewtonPuzzles}
For a given $n \geqslant 0$, the \emph{Newton puzzle of depth $n$} is the graph $\Gamma_n := \partial S_n \cup (\Gamma_n' \cap S_n)$.  The closure of a connected component of $\ovl S_n \sm \Gamma_n$ is a \emph{Newton puzzle piece of depth $n$}.
\end{definition}

The Newton puzzle pieces of depth $n$ provide a tilling of $\ovl{S_n}$. Theorem~\ref{Thm:NewtonPuzzles} guarantees that these puzzle pieces are closed topological disks with Jordan boundaries, and that they have the Markov property in the sense of Definition \ref{Def:Markov} (essentially, this follows from (\ref{Eq:PuzzleLevels2})). (See also \cite[Theorem~3.9]{DLSS} for a stronger result.) 

\subsection{Properties of Newton puzzles.}
\label{Subsec:PropNewtonPuz}

Let us review the properties of Newton puzzles proven in \cite{DLSS}. As mentioned above, we will be working with the iterate $g$ of an attracting-critically-finite Newton map $N_p$ for which we have well-defined puzzle pieces. For simplicity, we will keep calling $g$ a Newton map, even though it is an iterate of a Newton map. 

It follows directly from the construction that for every puzzle piece $P_n$ of depth $n$ the ``Julia boundary'' $\partial P_n \cap J(g)$ is a finite set consisting of (pre)poles of level at most $K + (n+2)M$, whereas $\partial P_n$ intersects the Fatou set of $g$ within the basins of the roots, and $\partial P_n$ intersects any particular Fatou component along two pieces of fixed or pre-fixed internal rays and an arc of some equipotential. 

From Definition~\ref{Def:PuzzleCentered} it is clear that if $x\in J(g)$ is neither $\infty$, nor a (pre)pole, then $P_n(x)$ is the unique puzzle piece of depth $n$ that contains $x$. Otherwise, $P_n(x)$ is a finite union of puzzle pieces with $x$ as their common boundary point. 

The following two lemmas are \cite[Theorem 3.9 (4) and (5)]{DLSS}.

\begin{lemma}[Infinity and (pre)poles have trivial fibers]
\label{Lem:FibersInfinity}
If $x$ is $\infty$, a pole or a prepole, then $\fib(x) = \{x\}$.\qed
\end{lemma}

\begin{lemma}[Fibers contained in interior]
\label{Lem:FibersCompContained}
Every $x\in\Cc$ that is not in the basin of a root has the property that its fiber is contained in $\inter P_n(x)$ for every depth $n\ge 0$. Stronger yet, for every $n\ge 0$ there is an $m>n$ so that $P_m(x)\subset\inter P_n(x)$. 
\qed
\end{lemma}

Here comes a result saying that the only possible obstruction for the fiber of a point $z$ to be trivial is when the orbit of $z$ accumulates at some critical fiber that is not pole or prepole. The second part of this result is a statement similar to \cite[Fact 5.1]{KSS}.

\begin{lemma}[Avoiding critical points implies trivial fiber and measure zero]
\label{Lem:AvoidingCritPtsNewton}
\label{Lem:ZeroArea}
Let $\CritNF(g) \subset \Crit(g)$ be the set of all critical points that are not (pre-)poles and do not lie in the basin of any root.

\begin{enumerate}
\item
\label{It:ZAass1}
If $z \in \C$ is not in the basin of any root and its orbit does not accumulate at a fiber of any critical point in $\CritNF(g)$, then $\fib(z) = \{z\}$.
\item
\label{It:ZeroArea}
If $U$ is a finite union of puzzle pieces so that $\inter U \cap \Crit(g) = \CritNF(g)$, and $E(U)$ is the union of all  orbits that avoid $\inter U$ and all root basins, then $E(U)$ is a nowhere dense compact set of zero Lebesgue measure.
\end{enumerate}
\end{lemma}

\begin{proof}
\setcounter{stepctr}{0}
\begin{step}[proving assertion \eqref{It:ZAass1}]
\label{St:ZA1}
Since the orbit of $z$ does not accumulate on critical fibers, except, possibly, at critical (pre-)poles, there exists $n > 0$ such that $\orb(g(z))$ is disjoint from the interiors of all puzzle pieces of depth $n$  that contain critical values in their interiors. Furthermore, since the fiber of $\infty$ is trivial (Lemma~\ref{Lem:FibersInfinity}), we can enlarge $n$ if necessary so that $\inter P_n(\infty)$ contains no critical values.

If $z$ is $\infty$ or a (pre-)pole, then the claim follows from Lemma~\ref{Lem:FibersInfinity}. So we can assume, for the rest of the proof of the first assertion, that $z$ is not $\infty$ or a (pre-)poles.

We first consider the case that there exists a point $w \in \omega(z)$ that does not belong to the puzzle boundary of any depth. In particular, $w$ is neither $\infty$ nor a (pre-)pole; and $w$ cannot be in any root basin since $z$ is not.
By Lemma~\ref{Lem:FibersCompContained}, every $k> n$ has an $l>0$ so that $A := \inter P_k(w) \sm P_{k+l}(w)$ is a non-degenerate annulus. Fix some $k > n$. Since the orbit of $z$ accumulates at $w$, there exists an increasing sequence $(k_i)$ with $g^{k_i}(z) \in \inter P_{k+l}(w)$. 

We claim that for all $i$, the annulus $A_i := \inter P_{k+k_i}(z) \sm P_{k+l+k_i}(z)$ is a conformal copy of $A$. Indeed, we can pull back $\inter P_k(w)\ni g^{k_i}(z)$ univalently for $k_i$ iterates along the orbit from $z$ to $g^{k_i}(z)$ and obtain $\inter P_{k+k_i}(z)$: the only possible obstacle would be a critical value in $\inter P_{k+k_i-j}(g^j(z))$ for $j\in\{1,2,\dots,k_i \}$, but this would mean that $g^j(z)$ was in a puzzle piece of depth $k+k_i-j\ge k>n$ with a critical value in its interior, which is excluded by hypothesis. We thus have $\modulus(A_i)=\modulus(A)$, and the  $A_i$ are nested (possibly after passing to a subsequence of $(k_i)$); hence, by the Gr\"otzsch inequality, $\fib(z) = \{z\}$.

We are thus left with the case when $\omega(z)$ consists only of points in the puzzle boundary, but not in the basin of any root.  By construction of puzzles, this means that $\omega(z)$ consists only of (pre-)poles, and hence contains $\infty$. Since the fiber of $\infty$ is trivial (Lemma~\ref{Lem:FibersInfinity}), we can exclude the case that $z$ is a (pre-)pole itself.

Since the orbit of $z$ accumulates at $\infty$, and $\infty$ is a repelling fixed point, there exist infinitely many points $y \in \orb(z)$ such that $y \in \inter P_{n+2}(\infty)$ and $g(y) \in \inter P_{n+1}(\infty) \sm P_{n+2}(\infty)$. Let $w \in P_{n+1}(\infty) \sm \inter P_{n+2}(\infty)$ be an accumulation point of such $g(y)$'s. By our assumption on $\omega(z)$, the point $w$ must be a (pre-)pole. Let $Y \subset P_{n+1}(\infty) \sm \inter P_{n+2}(\infty)$ be a puzzle piece of depth $n+2$ containing $w$ and such that $\orb(z)$ intersects $\inter Y$ in an infinite set.  

We claim that $B:= \inter P_n(w) \sm Y$ is a non-degenerate annulus. To see this, let $W$ be the puzzle piece of depth $n+1$ such that $Y \subset W$ and $\infty \in \partial W$. From our choice of $n$ it follows that $P_{m+1}(\infty) \subset \inter P_m(\infty)$ for every $m \ge n$: the point $\infty$ is a repelling fixed point, and the set $\inter P_n(\infty)$ lies in the linearizing neighborhood around $\infty$ as it contains no critical values. Therefore, the set $P_n(w)$ is a puzzle piece of depth $n$ (rather than a union of puzzle pieces), $\infty \in \partial P_n(w)$, and we have the inclusion $Y \subset W \subset P_n(w)$. Furthermore, $\partial W \cap \partial P_n(w)$ consists of $\infty$ and two pieces of fixed rays meeting at $\infty$. But since $\partial Y$ is disjoint from $\infty$, and hence cannot contain pieces of fixed rays, it does not intersect $\partial W \cap \partial P_n(w)$. Therefore, $\inter Y \subset P_n(w)$, and hence the annulus $B$ is non-degenerate. The rest of the proof for $B$ is the same as for the annulus $A$. The first assertion is proven.
\end{step}

\begin{step}[proving assertion \eqref{It:ZeroArea}]
Clearly, since $\inter U$ is an open set, $E(U)$ is closed and hence compact. By the first claim of the lemma, $E(U)$ is totally disconnected and hence nowhere dense. It remains to show that it has zero Lebesgue measure.

We can assume that $U$ is composed of puzzle pieces of the same depth. Indeed, if $U' \subset U$ is a union of puzzle pieces of the same depth that satisfies $\inter U' \cap \Crit(g) = \CritNF(g)$, then $E(U)$ is a subset of $E(U')$. We can further assume that the components of $U$ are of depth $n$, with $n$ having the properties specified in {Step~\ref{St:ZA1}} of the proof: $\inter P_n(\infty)$ contains no critical values and $\orb(g(z))$ is disjoint from all open puzzle pieces of depth $n$ that contain at least one critical value. 

Let $E'(U)$ be $E(U)$ minus $\infty$ and all prepoles. Since the difference is countable, it does not affect the measure, so we can focus on $E'(U)$.
Let $z \in E'(U)$ be a Lebesgue density point of $E'(U)$. 

From the construction in part {(A)}, there exists a point $w$ in a puzzle piece $Y$ of depth at least $n+2$ so that $Y \subset \inter P_n(w)$ and the orbit of $z$ accumulates at $w$ from within $Y$. Pick an $\eps > 0$ large enough so that the open round disk $B:=B(w, \eps)$ of radius $\eps$ centered at $w$ intersects $\partial Y \sm J(g)$, but small enough so that $\ovl B\subset \inter P_n(w)$. Since $\orb(z)$ accumulates at $w$, there exists an increasing sequence of integers $(n_{k})_{k \ge 0}$ so that $g^{n_k}(z) \in B$. Let $D_k$ be the component of $g^{-n_k}(B)$ containing $z$; by construction, $D_k \subset \inter P_{n+n_k}(z)$ and the map $g^{n_k} \colon \inter P_{n+n_k}(z) \to \inter P_n(w)$ is univalent.

Since $B$ intersects the boundary of $Y$, from the boundary structure of Newton puzzle pieces it follows that there exists a point $\xi$ in one of the root basins and a small $\delta > 0$ so that $B':=B(\xi, \delta) \subset B \sm J(g)$ (see the discussion in the beginning of the current subsection). Let $D_k'$ be the component of $g^{-n_k}(B')$ contained in $D_k$.

By the result in {Step~\ref{St:ZA1}}, $\diam D_k \to 0$ as $k \to \infty$. By Lemma~\ref{Lem:Fact}, all $D_k$ and $D_k'$ have uniformly bounded shapes, and by Lemma~\ref{Lem:Diam}, the diameters of $D_k$ and $D_k'$ are comparable with the constant independent of $k$. As $k \to \infty$, it follows that
\[
\frac{\area(D_k \cap J(g))}{\area(D_k)} \,\,\not {\!\!\!\!\longrightarrow}\, 1,
\]    
because $\area(D_k \cap J(g)) \le \area (D_k \sm D_k') < C \cdot \area(D_k)$ for some constant $C < 1$ independent of $k$. This is a contradiction to our choice of $z$ as a density point of $E'(U)$. By the Lebesgue Density Theorem we conclude that $E'(U)$ has measure zero. The second assertion follows. 
\end{step}
\end{proof}

A critical point $c$ of a Newton map $g$ is called \emph{combinatorially periodic} if there exists a puzzle piece $W$ at some depth containing $c$, and an integer $s > 1$ such that $g^{sk}(c') \in \inter W$ for every critical point $c' \in W$ and every $k \ge 0$. Combinatorially periodic critical points are never mapped to $\infty$ and do not belong to the basins of roots. Moreover, the following lemma shows that all combinatorially periodic critical points are, in fact, renormalizable in the sense of Douady--Hubbard: each of them lies in the non-escaping set of a suitable polynomial-like restriction (with domain being the interior of a puzzle piece). This lemma is a slight modification of \cite[Proposition~3.16]{DLSS}, and essentially follows from Lemma~\ref{Lem:FibersCompContained}.    

\begin{lemma}[Polynomial-like renormalization at combinatorially periodic points]
\label{Lem:NonRepellingRenormalizable} \looseness-1
If $x$ is a combinatorially periodic critical point, then there exists $n > 0$ and a least $k = k(n) >0$ so that the map $g^k \colon \inter P_{n + k}(x) \to \inter P_{n}(x)$ is a polynomial-like map of degree $d \geqslant 2$ with connected filled Julia set equal to $\fib(x)$. Moreover, for sufficiently large $n$ and for any two combinatorially periodic points $x$ and $x'$ as above, either $\fib(x)=\fib(x')$ or $P_n(x)\cap P_n(x')=\emptyset$.\qed
\end{lemma}

We obtain that in the Newton setting for a critical point to be combinatorially periodic is equivalent to being renormalizable: the inclusion in one direction is given by Lemma~\ref{Lem:NonRepellingRenormalizable}; inclusion in the other direction follows by definition. (Note here that the \emph{renormalization period} $k = k(n)$ from Lemma~\ref{Lem:NonRepellingRenormalizable} can be larger than the least period coming from combinatorial periodicity.) In the sequel, we will use these terms interchangeably. 

As defined in Section~\ref{Sec:GeneralPuzzles}, a point $x$ is combinatorially recurrent if the orbit of $g(x)$ under $g$ intersects every puzzle piece at $x$ (we had shown in Section~\ref{Sec:GeneralPuzzles} that this implies that the orbit visits every such puzzle piece infinitely often). In this case, we can define a strictly increasing sequence $(n_i)$ as follows, starting at arbitrary $n_0\ge 0$: given $n_i$, let $k_i$ be minimal so that $g^{k_i}(x)\in P_{n_i}(x)$ and set $n_{i+1}:=n_i+k_i$. Then $g^{k_i}$ sends $P_{n_{i+1}}(x)$ to $P_{n_i}(g^{k_i}(x))=P_{n_i}(x)$. 

Every combinatorially periodic (and hence renormalizable, see Lemma~\ref{Lem:NonRepellingRenormalizable}) critical point is combinatorially recurrent. In this case, the sequence $k_i=n_{i+1}-n_i$ is eventually constant. There is a converse to this observation, as follows:

\begin{lemma}[First return times for the pullback nest]
\label{Lem:FirstReturnTimesNest}
For a given $n_0 \ge 0$, the sequence $k_i=n_{i+1}-n_i$ associated to a combinatorially recurrent critical point $x$ via the pullback construction above is monotonically increasing. It is bounded (and hence eventually constant) if and only if $x$ is renormalizable.
\end{lemma}
\begin{proof}
Monotonicity of $k_i$ follows immediately from the definition that $k_i$ is minimal so that $g^{k_i}(x)\in P_{n_i}(x)$: a larger value of $i$ means a smaller set $P_{n_i}(x)$, and hence it can take only longer to return into the smaller set.

If the sequence $(k_i)$ is bounded, and hence eventually constant, then $x$ is combinatorially periodic with respect to $\inter P_{n_i}(x)$ for $i$ large enough. Thus $x$ is renormalizable by Lemma~\ref{Lem:NonRepellingRenormalizable}. Conversely, if $x$ is renormalizable, then $x$ is periodic with some period $k$ (which is equal to the renormalization period). This implies that eventually $k_i \le k$; hence the sequence is bounded. 
\end{proof}

For combinatorially recurrent critical points that are not renormalizable we can assure that the boundaries of the puzzle pieces in the pullback nest constructed above are disjoint for all sufficiently large depths, due to the following lemma. 

\begin{lemma}[Non-renormalizable recurrent points are well inside]
\label{Lem:NonRenNest}
For every $n_0 \geqslant 0$, if $x$ is a combinatorially recurrent critical point that is non-renormalizable, and $\left(P_{n_i}(x)\right)_{i \ge 0}$ is the nest obtained by pulling back $P_{n_0}(x)$ along the orbit of $x$, then there exists $j$ large enough so that $P_{n_{i+1}}(x) \subset \inter P_{n_i}(x)$ for all $i \ge j$.
\end{lemma}

\begin{proof}	
The proof is similar to the proof of \cite[Proposition 3.10]{DMRS}. First of all, possibly by increasing $n_0$ we can assume that $\partial P_{n_0}(x)$ is disjoint from $\infty$: this is possible since $\infty$ has trivial fiber (Lemma~\ref{Lem:FibersInfinity}), and hence the puzzle pieces containing $\infty$ cannot intersect all elements in the nest $\left(P_{n_i}(x)\right)_{i \ge 0}$. Hence we can assume that the boundaries of the puzzle pieces in the nest are disjoint from periodic points of $g$.

Since $\partial P_{n_0}(x)$ contains finitely many (pre-)poles and no periodic points, there exists $k$ large enough so that for all (pre-)poles $w \in \partial P_{n_0}(x)$ and all $r \ge k$ we have $g^r(w) \not \in \partial P_{n_0}(x)$. By Lemma~\ref{Lem:FirstReturnTimesNest}, the sequence $\left(k_i\right)_{i \ge 0}$ of first return times tends to infinity. Therefore, there is a minimal $j$ so that $k_j \ge k$. Let us show that the lemma holds for this $j$. If not, then there exists $i \geqslant j$ with $\partial P_{n_{i+1}}(x) \cap \partial P_{n_i}(x) \not= \emptyset$. This intersection, by construction of puzzle pieces (see the beginning of Subsection~\ref{Subsec:PropNewtonPuz}), must contain a (pre-)pole, say $z$. Hence $\partial P_{n_i}(x)$ contains the (pre-)poles $z$ and $g^{k_i}(z)$, and they are distinct (there are no periodic points on the boundary). Mapping these two points forward, it follows that $\partial P_{n_0}$ must contain a pair of distinct (pre-)poles $w$ and $g^{k_i}(w)$. But $k_i \geqslant k_j \ge k$, and this contradicts our choice of $k$.
\end{proof}

\Newpage

\section{Proof of dynamical rigidity for Newton maps (Theorem ~\ref{Thm:RRP})}
\label{Sec:RRPProof}

\subsection{Proof of Theorem~\ref{Thm:RRP}}

We will prove Theorem~\ref{Thm:RRP} \and its corollaries for a given Newton map $N_p$, or rather its iterate $g=N_p^{M}$. Unless mentioned otherwise, every orbit will be understood as an orbit under iteration of $g$, and the same applies to puzzles, fibers and their triviality. It will not be a loss of generality to assume that $N_p$ is attracting-critically-finite because the relevant properties are preserved by the surgery construction.

The overall plan is to extract, in the dynamical plane of $g$, a box mapping for which the non-escaping set ``captures'' most of the critical orbits that do not belong to the Newton graph at any level; this is done in Lemma~\ref{Lem:NewtonBox} below. For the points in $\Cc$ whose orbits intersect the non-escaping set of such a box mapping we will be able to conclude ``rigidity or polynomial-like dynamics'' using Theorem~\ref{Thm:RigidityComplexBox}. The remaining points in $\Cc$ will be treated by means of Lemma~\ref{Lem:FibersInfinity} and Lemma~\ref{Lem:AvoidingCritPtsNewton}.

Our strategy depends on the following distinction of the critical points
\begin{equation}
\label{Eq:CompleteCrit}
\Crit(g) = \Crit_B\sqcup \Crit_I \sqcup \Crit_{nR} \sqcup \Crit_{R}
\;:
\end{equation}
\begin{enumerate}
\item[(B)]
we say that $c\in \Crit_B$ if the orbit of $c$ converges to one of the roots of the polynomial $p$ (critical points in the \textbf{B}asin of a root); 

\item[(I)]
we say that $c\in \Crit_I$ if the orbit of $c$ lands at $\infty$ (critical points landing at \textbf{I}nfinity); 

\item[(R)]
we say that $c \in \Crit_{R}$ if $c$ is combinatorially \textbf{R}ecurrent (it may or may not be renormalizable);

\item[(nR)]
we say that $c \in \Crit_{nR}$ if $c \not\in \Crit_I \sqcup \Crit_B\sqcup \Crit_R$ ($c$ is combinatorially \textbf{n}on-\textbf{R}ecurrent).
\end{enumerate}
Let $\CritNF(g)$ be the set of all critical points of $g$ whose orbits do not land at $\infty$ or at a root (the critical points of the original Newton map that are \textbf{n}ot eventually \textbf{f}ixed), hence $\CritNF(g) = \Crit_{nR} \cup \Crit_{R}$. 
Finally, define $\Crit^a_{nR} \subset \Crit_{nR}$ as the set of all critical points in $\Crit_{nR}$ that \textbf{a}ccumulate at the fiber of at least one critical point in $\CritNF$ then $\Crit^a(g) := \Crit^a_{nR} \cup \Crit_{R}$.

\begin{lemma}[Newton box mapping]
\label{Lem:NewtonBox}
If $\CritNF(g) \neq \emptyset$, then there exists a complex box mapping $F \colon \mathcal U \to \mathcal V$ with the following properties:
\begin{enumerate}
\item
the components of $\U$ and $\V$ are the interiors of puzzle pieces of $g$, and for every component $U$ of $\U$ there exists $k \ge 1$ so that $F|_U = g^k|_U$;
\item
every $g$-orbit that accumulates at some point in $\V$ intersects $K(F)$;
\item
$\Crit(F) \subset \CritNF(g) \subset \V$ and $\Crit^a(g) \subset \Crit(F) \cap K(F)$;
\item
\label{It:ExactlyOne}
every critical component of\/ $\U$ contains exactly one critical fiber.
\end{enumerate}

Moreover, $F$ is dynamically natural in the sense of Definition~\ref{Def:NaturalBoxMapping}.
\end{lemma}

\begin{proof}

\setcounter{stepctr}{0}
\begin{step}[construction of a box mapping]
By finiteness of critical points, there exists a depth $s > 0$ so that $P_s(c) \cap \Crit(g) \subset \fib(c)$ for every $c \in \CritNF(g)$. 
In this proof we assume that the critical puzzle pieces are chosen of depth at least $s$. In particular, this assumption guarantees that the interiors of these puzzle pieces do not intersect $\Crit_B \cup \Crit_I$. 

By passing to an iterate of $g$ (and keeping the notation for simplicity), and possibly increasing $s$, we can assume that for every renormalizable critical point $c$, if $g^k \colon \inter P_{s+k}(c) \to \inter P_s(c)$ is the corresponding polynomial-like restriction (with non-escaping set equal to $\fib(c)$), then $\inter P_{s+k}(c)$ is a component of the first return domain $\Dom(\inter P_s(c))$. Such an iterate exists by Lemma~\ref{Lem:NonRepellingRenormalizable}. By passing to an iterate we do not change the number of renormalizable fibers.  

\looseness-1
We construct open sets $\W$ and $\W_0$,\dots, $\W_m$ for $m\le|\Crit_R|$, using induction over the set $\Crit_R$ and starting with $\W_0:=\emptyset$ (if $\Crit_R=\emptyset$, then $m=0$ and $\W=\W_0 = \emptyset$). Set $W_0:=\emptyset$. 

For the inductive step, assume that for some $k \ge 1$ critical points $c_1,\dots,c_{k-1}\in\Crit_R$ are chosen, and we have constructed a set $\W_{k-1}$. 
If $\W_{k-1}$ intersects the orbits of all critical points in $\Crit_R$, then the induction is complete with $m=k-1$. Otherwise, choose another critical point $c_{k} \in \Crit_R$ so that its orbit does not intersect $\W_{k-1}$. Choose a puzzle piece $W_{k}$ with $c_{k} \in \inter W_{k}$ and of depth no less than $W_{k-1}$ (at least $s$). Let $X_{k}\subset\C$ be the union of those components in $\DomE(\inter W_{k})$ that contain a critical point.  Since $W_{k}$ has depth at least $s$, all critical points of $g$ in $\inter W_k$ are in the same fiber, so there is a single component of $\inter W_{k} \cap {X_{k}}$ and its closure is a single critical puzzle piece, say $W_{k}^*$. (In particular, this means that all critical points in $\fib(c_{k})$ are combinatorially recurrent.)

We claim that we can choose $W_{k}$ to be of sufficiently large depth so that $W_k^* \subset \inter W_{k}$.  For non-renormalizable critical points this follows from Lemma~\ref{Lem:NonRenNest}, whereas for renormalizable points it follows by our choice of $s$ and considering the iterate of $g$ described above. Setting $\W_{k} := \W_{k-1} \cup \inter W^*_{k}$ completes the inductive step.

Suppose that the induction for $\Crit_R$ ended with the set $\W_m$. Let $\W$  be the union of $\W_m$ and all components in $\DomE(\W_m) \sm \W_m$ that contain critical points. Again, since all puzzle pieces in question are of depth at least $s$, each component in $\W$ contains at most one critical fiber. The construction of $\W_m$ implies $\Crit_R \subset \W$.  

Finally, let $\Crit_{nR}' := \Crit_{nR} \sm \W$; these are all combinatorially non-recurrent critical points whose orbits are disjoint from  $\W_m$. Define the set 
\begin{equation}
\label{Eq:BMDom}
\mathcal V := \W \cup \bigcup_{c \in \Crit_{nR}'} \inter P_l(c),
\end{equation}
where $l \ge s$ is a depth at least as large as the depths of the puzzle pieces in $\ovl \W$ with the further properties that the orbit of each $c \in \Crit_{nR}'$ never re-enters the puzzle piece $P_l(c)$ and each of the puzzle pieces $P_l(c)$ for $c \in \Crit_{nR}$ is weakly protected. The existence of such a depth follows from Lemma~\ref{Lem:FibersCompContained} and because the critical points in question are {combinatorially} non-recurrent. Observe that $\CritNF(g) \subset \mathcal V$.

By construction, the set $\mathcal V$ is a nice open set with finitely many connected components, each of which is the interior of a puzzle piece, and each containing at least one critical point. Define $\U := \Dom(\V)$ and let $F \colon \mathcal U \to \mathcal V$ be the first return map to $\V$ as defined in Lemma~\ref{Lem:FirstReturnConstruction}. We claim that the map $F \colon \mathcal U \to \mathcal V$ is a box mapping with the desired  properties. 

Indeed, $F$ is a box mapping in the sense of Definition~\ref{Def:BoxMap}: Lemma~\ref{Lem:FirstReturnConstruction}~\eqref{It:FRM2} implies property \eqref{It:BM4} in Definition~\ref{Def:BoxMap}, while property \eqref{It:BM2} of the definition follows from the construction of~$\mathcal V$. 

For Property~\eqref{It:BM3}, observe first that since each of the puzzle pieces $W_k^*$ is weakly protected by $W_k$,  Lemma~\ref{Lem:FRD} implies  that the set of points in $\W_m$ that eventually return to $\W_m$ is compactly contained in $\W_m$. At the same time, the components of $\W \sm \W_m$ are simultaneously components of $\mathcal U$ and $\mathcal V$. Similarly, by Lemma~\ref{Lem:FRD2}, since each $P_l(c)$ for $c \in \Crit_{nR}'$ is weakly protected (by a puzzle piece of some depth) and the orbit of $c$ escapes from $P_l(c)$, the set of points in $\inter P_l(c)$ returning to $\inter P_l(c)$ is compactly contained in $\inter P_l(c)$. 

Finally, property~\eqref{It:BM1} follows from Lemma~\ref{Lem:FirstReturnConstruction}~\eqref{It:FRM3}: we have that $\CritNF(g) \subset \V$ and the orbits of points in $\mathcal V$ that return to $\V$ do not intersect $\Crit_B \sqcup \Crit_I$, as the latter points do not belong to the interiors of the puzzle pieces of depth $s$. Therefore, the map $F \colon \U \to \V$ is indeed a box mapping.
\end{step}

\begin{step}[properties of the box mapping]
\label{St:BoxProp}
Lemma~\ref{Lem:FirstReturnConstruction}~\eqref{It:FRM4} implies that every $g$-orbit that accumulates at some point in $\V$ also intersects $K(F)$. Since $\CritNF(g) \subset \mathcal V$, we conclude that $\Crit^a(g) \subset K(F)$. Moreover, since $\Crit(F) \subset \CritNF(g)$ (by Lemma~\ref{Lem:FirstReturnConstruction}~\eqref{It:FRM3}), it follows that $\Crit^a(g) \subset \Crit(F) \cap K(F)$. 

The claim that each critical component of $\U$ contains exactly one critical fiber follows from the construction of $F$. (Note that there is no ambiguity by referring to critical fibers without explicitly mentioning the map: since $\Crit(F) \subset \CritNF(g)$, each critical fiber of $F$ is a critical fiber of $g$.)

It remains to prove that our box mapping is dynamically natural in the sense of Definition~\ref{Def:NaturalBoxMapping}. 
The fact that the box mapping $F \colon \mathcal U \to \mathcal V$ has no \eqref{Item:NE} components follows again from the construction: any cycle of \eqref{Item:NE} components of $\mathcal U$ must pass either through the set $\W_m$, or through one of the sets $\inter P_l(c)$ for some $c \in \Crit_{nR}'$. The former is impossible since all components in $\W_m \cap \mathcal U$ are compactly contained in $\W_m$ by construction of the box mapping, while the latter is impossible since $c$ is an escaping point for $\inter P_l(c)$, again by construction.

The next condition for dynamical naturality that we verify is  $\area(\Koc(F))=0$. For $n \ge s$, define $U_n := \bigcup_{c \in \CritNF(g)} P_n(c)$. By Lemma~\ref{Lem:ZeroArea},  $\area(E(U_n))=0$. Since $\Koc(F)\subset \bigcup_n E(U_n)$, and the latter set is a countable union of sets of measure zero, the claim follows. 

The final condition is $K(F) = \Kwi(F)$. If $\U$ has only finitely many connected components, then this property is obvious. So assume that $\U$ has infinitely many components: all but finitely many of them are compactly contained in $\V$.

Define $\X$ to be the set of all components of $\U$ with the following properties:

\begin{itemize}
\item
each $U \in \X$ is compactly contained in the respective component of $\V$;
\item
for every $U \in \X$, if $g^r \colon U \to V$, $V \subset \V$ is the restriction of $F$ to $U$, then $r$ is larger than the maximal depth of the components in $\ovl \V$. 
\end{itemize}

The set $\X$ contains all but finitely many components of $\U$. Let us show that the elements of $\X$ are ``well-inside'' $\V$, i.e.\ there exists $\delta > 0$ such that for every $U \in \X$, if $W$ is the component of $\V$ containing $U$, then $\modulus(W \sm \ovl U) \ge \delta$. This would finish the proof of dynamical naturality of $F$ because by definition of $\Kwi(F)$ it then would follow that $K(F) = K_{\delta'}(F)$, where $\delta'$ is the minimum over $\delta$ and finitely many moduli of annuli that separate the components of $\U$ not accounted in $\X$ but that are still compactly contained in $\V$. 

Let $U \in \X$ and $W$ be the component of $\V$ such that $\ovl U \subset W$, and suppose $g^r \colon U \to V$ is the restriction of $F$ to $U$. By construction of $F$, the degree of the map $g^r \colon U \to V$ is bounded independently of $U$ (because $F$ is a first return map, see Lemma~\ref{Lem:FirstReturnConstruction}) and $V$ is a component in the decomposition~\eqref{Eq:BMDom}. The latter means that $\ovl V$ is weakly protected by some open puzzle piece $V'$. Let $U'$ be the component of $g^{-r}(V')$ containing $U$ (see Figure~\ref{Fig:PullingSpace}). Since $V$ is weakly protected by $V'$, i.e.\ $\ovl V \subset V'$, we conclude that $U' \sm \ovl U$ is a non-degenerate annulus. Furthermore, $U' \subset W$ because $\ovl U'$ is of depth at least $r$ that is larger than the depths of all the components of $\ovl \V$ (by definition of $\X$). We claim that the degree of the map $g^r \colon U' \to V'$ is uniformly bounded (independently of the choice of $U \in \X$).   

\begin{figure}[htbp]
\begin{center}
\includegraphics[scale=0.5, trim=0 2 2 2, clip]{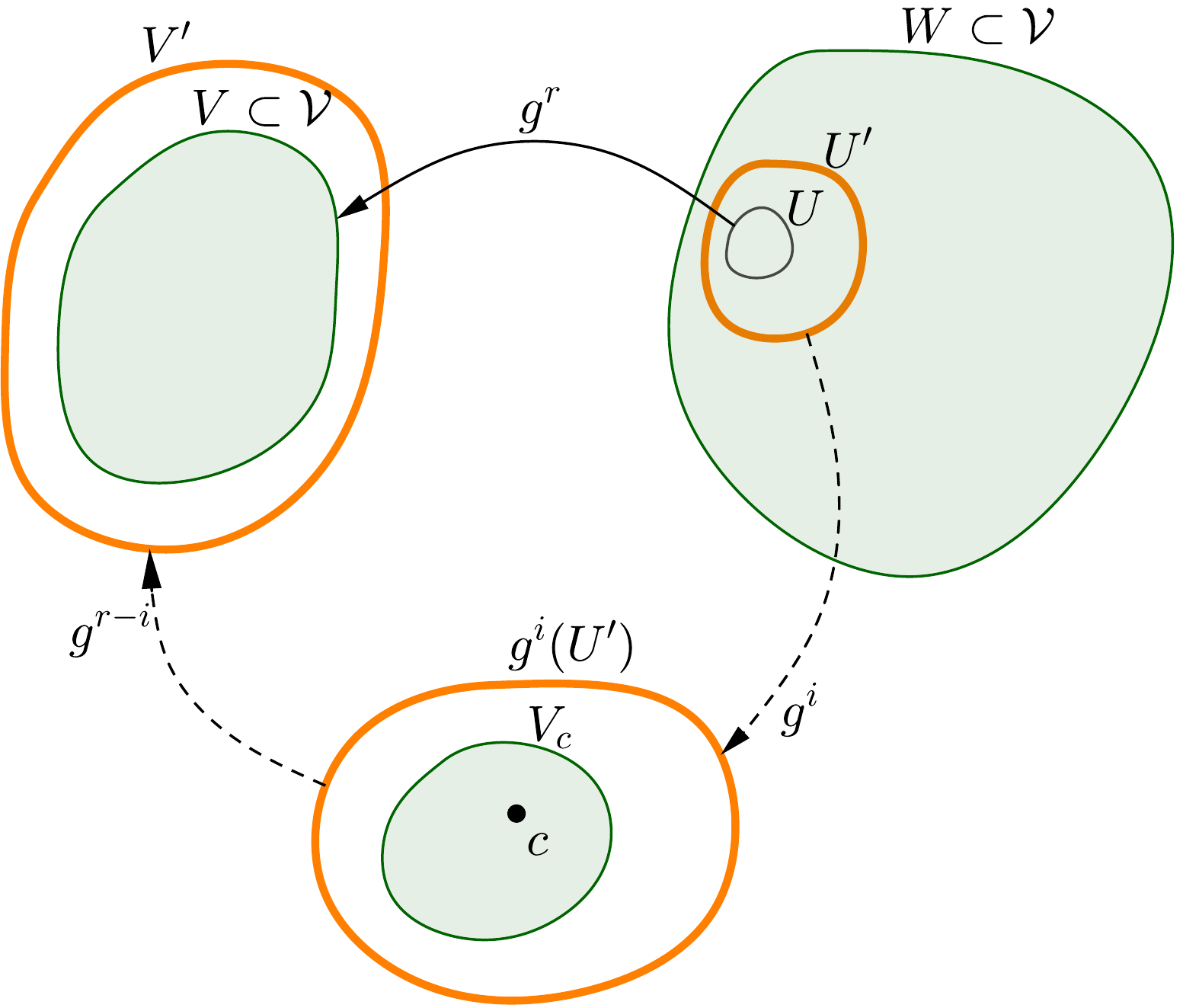}
\caption{An illustration for the proof of Lemma~\ref{Lem:NewtonBox}.}
\label{Fig:PullingSpace}
\end{center}
\end{figure}

Suppose $g^i(U')$ is an open puzzle piece containing a critical point $c \in \CritNF(g)$ for some $0 \le i < r$, and let $V_c$ be the component of $\V$ containing $c$. If the depth of $g^i(U')$ is larger than the depth of $V_c$, then $g^i(U') \subset V_c$, which is impossible: $i < r$ and hence $U$ would be contained in a component of $g^{-i}(V_c)$, a contradiction to the fact that $U$ is the domain of a branch of the first return map to $\V$ with the range equal to $V$. Therefore, the depth of $g^i(U')$ must be smaller than the depth of $V_c$. Now the claim follows because there are only finitely many critical puzzle pieces of depth smaller than the depths of the components of $\ovl \V$, and thus the number of such indices $i$ is bounded above independently of $r$ (and hence of $U$). 

The fact that the degree of the map $g^r \colon U' \to V'$ is uniformly bounded above, say by some $D > 0$, together with Lemma~\ref{Lem:AnnulusLemma} imply that the modulus of $U' \sm \ovl U$ is bounded below only in terms of $D$ and $\modulus(V' \sm \ovl V)$; since $W \sm \ovl U \supset U' \sm \ovl U$, the same is true for $\modulus(W \sm \ovl U)$. And since there are only finitely many components of $\V$ and our choice of $V'$ was independent of $U$ we conclude that the elements of $\X$ are all ``well-inside'' $\V$, as desired.
\end{step}
\end{proof}

\begin{proof}[Proof of Theorem \ref{Thm:RRP}]

The statement of Theorem \ref{Thm:RRP} is invariant under the quasiconformal surgery that makes the Newton map $N_p$ attracting-critically finite (see Section~\ref{Sec:Puzzles}), and also by passing to an iterate. Therefore, we assume that $N_p$ is attracting-critically-finite and consider the iterate $g = N_p^{M}$ that has a well-defined puzzle partition. 

Consider some $z \in \Cc$ and distinguish the following cases:

\begin{enumerate}
\item
\label{It:B1}
the point $z$ is in the basin of a root, or equivalently, $\omega(z) \cap \Crit_B \neq \emptyset$; 

\item
\label{It:I1}
$\orb(z) \ni \infty$, so $z$ is $\infty$ or a (pre-)pole; in particular,  $\orb(z)$ might intersect $\Crit_I$;

\item
\label{It:NoCrit}
$\omega(z) \cap \fib(c) = \emptyset$ for every $c \in \CritNF(g)$, $\omega(z) \cap \Crit_B = \emptyset$, and $\orb(z) \not \ni \infty$; in other words, $z$ is not $\infty$ or a (pre-)pole and the $\omega$-limit set of $z$ does not intersect the roots and the critical fibers of $g$, except, perhaps, at critical (pre-)poles.

\item
\label{It:R1}
$\omega(z) \cap \fib(c) \not= \emptyset$ for some $c \in \CritNF(g)$, that is the orbit of $z$ lands in, or accumulates at the fiber of a critical point that is not in $\Crit_I \sqcup \Crit_B$.
\end{enumerate}

Clearly, these four  cases cover all possibilities. Moreover, all the cases above are disjoint. Let us show in each of the cases which of the alternatives (\hyperref[It:B]{B}), (\hyperref[It:T]{T}), or (\hyperref[It:R]{R}) in Theorem~\ref{Thm:RRP} holds. 

(\ref{It:B1}) This is precisely case (\hyperref[It:B]{B}).

(\ref{It:I1}) If $\orb(z) \ni \infty$, then $\fib(z) = \{z\}$ by Lemma~\ref{Lem:FibersInfinity}, so we are in case (\hyperref[It:T]{T}).

(\ref{It:NoCrit}) Since $\omega(z)$ does not intersect the roots and the critical fibers of $g$, except possibly at poles or prepoles, we can apply Lemma~\ref{Lem:AvoidingCritPtsNewton}, so $z$ has trivial fiber and we are in case (\hyperref[It:T]{T}).  

(\ref{It:R1}) Since all the points in $\CritNF(g) \sm \Crit^a(g)$ fall into case \eqref{It:NoCrit} (by definition, these are all the points in $\CritNF(g)$, necessarily combinatorially non-recurrent, that do not accumulate on fibers of points in $\CritNF(g)$), case \eqref{It:R1} concerns only critical points in $\Crit^a(g)$. 

By Lemma~\ref{Lem:NewtonBox}, there exists a dynamically natural box mapping $F \colon \mathcal U \to \mathcal V$ such that the set of all non-escaping critical points of $F$ contains $\Crit^a(g)$. Therefore, if the orbit of $z$ accumulates on the fiber of a point in $\Crit^a(g)$, then it also intersects the filled-in Julia set $K(F)$ of $F$, again by Lemma~\ref{Lem:NewtonBox}. Since the box mapping $F$ contains no \eqref{Item:NE} components, and the orbit of $z$ is not of \eqref{Item:CB} type (those are taken care of in case (\ref{It:NoCrit})), the conclusion of Theorem~\ref{Thm:RRP} follows from Theorem~\ref{Thm:RigidityComplexBox}. 
\end{proof}

\subsection{Proof of Corollaries~\ref{Cor:BasinsLocConn}  and \ref{Cor:LC}}
\label{Sec:LC}

In this subsection, we will prove that every component of the basin of every root has locally connected boundary (Corollary~\ref{Cor:BasinsLocConn}), and  that the Julia set of every Newton map is locally connected provided that every polynomial-like restriction of $N_p$ with connected Julia set straightens to a polynomial in~${\Sspace}$ (Corollary~\ref{Cor:LC}). Theorem~\ref{Thm:RRP} will be our main ingredient for the proof. 
(Recall that ${\Sspace}$ stands for the set of all polynomials so that the two conditions are satisfied: 1) most of the Fatou components are small; 2) the boundaries of all Siegel disks are Jordan curves.)

\begin{proof}[Proof of Corollary~\ref{Cor:BasinsLocConn}]
As before, we assume that $N_p$ is attracting-critically-finite and replace  $N_p$ by its iterate $g:=N_p^M$, so we work in the setting of Section~\ref{Sec:Puzzles}.
By invariance, it suffices to prove that every immediate basin $U_\xi$ has locally connected boundary. We will prove the following somewhat more precise statement: \emph{every $z\in\partial U_\xi$ has trivial fiber, unless $z$ is eventually periodic and the fiber of $z$ intersects $\partial U_\xi$ in a single point, which is (pre)periodic. In both cases, $\partial U_\xi$ is locally connected at $z$.
}

By Theorem~\ref{Thm:RRP}, the point $z$ either has trivial fiber, or it maps after finitely many iterations to the little filled-in Julia set, say $K$, of a polynomial-like restriction of $g$ with connected Julia set, and so that $K$ is connected. In the first case, $U_\xi$ is locally connected at $z$ and we are done, so (after replacing $z$ by a point on its orbit) it suffices to assume that $z\in K$. The key step of the proof will be to show that $K\cap \partial U_\xi=\{z\}$, and as a byproduct that $z$ is a periodic point of $g$.

By Lemma~\ref{Lem:NonRepellingRenormalizable}, the set $K$ is the filled-in Julia set of a polynomial-like restriction $g^k \colon \inter P_{n+k}(z) \to \inter P_n(z)$ for some $n \ge 0$ and $k \ge 1$, and all critical points of $g^k$ on $\inter P_n(z)$ are already in $K$. Denote its degree by $\delta\ge 2$. Define $A:=\inter P_n(z)\sm K$; this is a non-degenerate annulus, again by Lemma~\ref{Lem:NonRepellingRenormalizable}.

For every $w_0\in A$, for each of the $\delta$ preimages $w_1\in (g^k)^{-1}(w_0)$,  and for every simple curve $\gamma_0\subset A$ connecting $w_0$ to $w_1$, one can construct a curve $\gamma$ that starts at $w_0$ and converges to $\partial K$ by connecting a preimage component $\gamma_1$ of $\gamma_0$ to $w_1$, then connecting another preimage component of $\gamma_1$ to the end of $\gamma_1$, and so on. A standard argument relating the hyperbolic and Euclidean geometries of $A$  shows that this curve must converge to a point in $\partial K$ that is fixed by $g^k$. After straightening, this curve is fixed by the polynomial and lies in the escaping set, so it lands at a fixed point and is homotopic (relative to the filled Julia set) to a fixed dynamic ray (by Lindel\"of's theorem).

There are only finitely many fixed points of $g^k$ on $\partial K$ and only finitely many homotopy classes of fixed rays, so many choices of $w_0$, $w_1$ and $\gamma_0$ will lead to homotopic curves. 

We are going to use the curves $\gamma_0$ and its extension $\gamma\subset A$ to construct three curves landing at $z$: one curve will be an internal ray of $U_\xi$ fixed by $g^k$, and two further curves together will separate $K\sm\{z\}$ from $U_\xi$. This will then establish local connectivity of $\partial U_\xi$ at $z$ (see Figure~\ref{Fig:RealBasilicaTouchBasin} for an illustration of the construction).  

\begin{figure}[htbp]
\begin{center}
\includegraphics[scale=0.5, trim={50 5 20 5},clip]{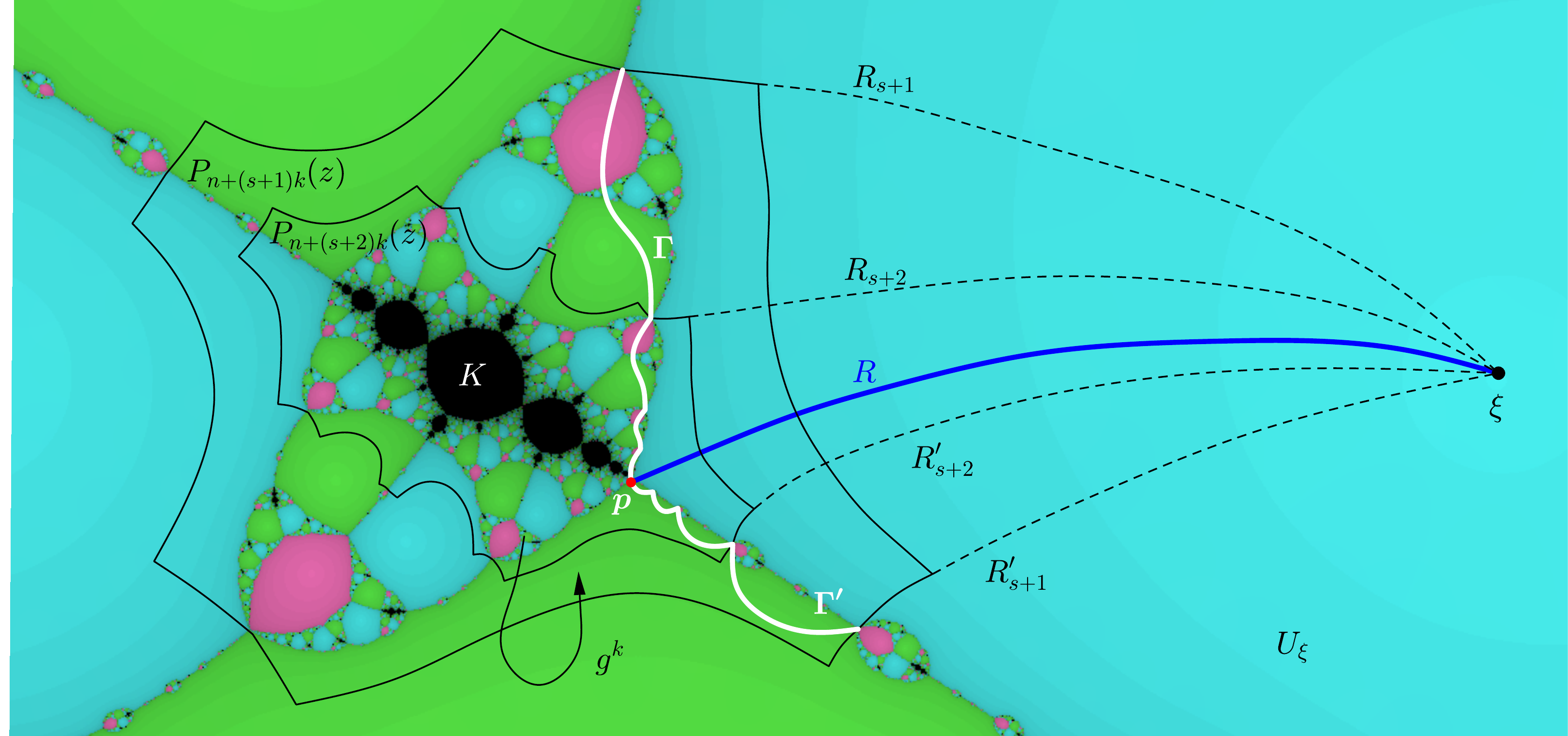}
\caption{Example of the separating curves $\Gamma$ and $\Gamma'$ in the proof of Corollary~\ref{Cor:BasinsLocConn}: their parts within $P_{n+(s+1)k}(z)$ are shown in white. The point $z$ is marked in red. The curves $\Gamma$, $\Gamma'$ and the internal ray $R$ land at the same point $p$, this point is fixed by $g^k$. The curve $\Gamma \cup \{p\} \cup \Gamma'$ separates $K \sm \{p\}$ and the immediate basin $U_\xi$ within the puzzle pieces.}
\label{Fig:RealBasilicaTouchBasin}
\end{center}
\end{figure}

The dynamics of $N_p$ on $U_\xi$ is conformally conjugate to $z\mapsto z^m$ on $\disk$ for some $m\ge 2$, so dynamic rays on $U_\xi$ are well defined together with their usual mapping properties, and periodic rays land at periodic points. By the construction of Newton puzzles (see Section~\ref{SSec:ConstructionOfNewtonPuzzles}), for every $s \ge 0$ the puzzle piece $P_{n+sk}(z)$ intersects $U_\xi$ in a domain, say $D_s$, so that $\partial D_s\cap U_\xi$ consists of pieces of two pre-fixed dynamic rays, say $R_s$ and $R'_s$, together with some equipotential in $U_\xi$. We clearly have $D_{s+1}\subset D_s$ for all $s$, with common boundary only on $\partial U_\xi$. Since $g^k(P_{n+(s+1)k}(z))=P_{n+sk}(z)$ and hence $g^k(D_{s+1})=D_s$, the rays $R_s$ and $R'_s$ must converge from both sides to a single ray, say $R$, that is fixed by $g^k$. This ray $R$ must land at a point $p\in\partial U_\xi$ that is also fixed by $g^k$. 

One particular curve $\gamma\subset A$ that lands at a fixed point in $\partial K$ can be constructed by choosing $w_0$, $w_1$ and $\gamma_0$ on $R$; then the entire curve $\gamma$ is a subset of $R$ and lands at $p$.

Construct another curve $\Gamma$ as follows. Start by connecting the landing points of the rays $R_{s+1}$ and $R_{s+2}$ within $A$ by a curve $\Gamma_0\subset A$ that lies outside of $\ovl U_\xi$. In particular, since $\Gamma_0$ avoids the ray $R$, this fixes the homotopy class of $\Gamma_0$ in $A$. Let us explain in more detail why such a curve exists. By the structure of the boundary of a Newton puzzle piece (see the discussion in the beginning of Section~\ref{Subsec:PropNewtonPuz}), there exists a component $V$ of the basin of some root such that $\ovl V \cap \ovl U_\xi$ is exactly the landing point of $R_{s+1}$ and $V \cap \partial P_{n+(s+1)k}(z)$ is a piecewise smooth curve consisting of two pieces of internal rays (for $V$) and an arc of an equipotential. 
Consider the set
$B := \inter P_{n+(s+1)k}(z) \sm P_{n+(s+2)k}(z) \subset A$, which is a non-degenerate annulus. Then  the landing point of $R_{s+1}$ is accessible from within the connected set $B \sm \ovl U_\xi$.  Likewise, the landing point of $R_{s+2}$ is accessible from within $B \sm \ovl U_{\xi}$. The curve $\Gamma_0$ can be now chosen as a curve connecting the landing points of $R_{s+1}$ and $R_{s+2}$ in $B \sm \ovl U_{\xi}$. 

Extend the curve $\Gamma_0$ as before; this will yield the curve $\Gamma$. The two curves $\Gamma_0$ and $\gamma_0$ (between $w_0$ and $w_1$) have finite hyperbolic distance between each other within $A$, and this distance is preserved by taking preimages with respect to the hyperbolic distance of preimage domains of $A$; hence the distance is contracted with respect to $A$. Therefore the extensions $\Gamma$ and $\gamma$ must land at the same point $p$. A third curve $\Gamma'$ that lands at $p$ can be constructed  analogously starting from the landing points of the rays $R'_{s+1}$ and $R'_{s+2}$. The hyperbolic distance argument shows that all three curves land through the same access to $p$ relative to $K$. Hence, the union $\Gamma\cup \{p\}\cup \Gamma'$ disconnects $P_{n+k}(z)$, but it does not disconnect $K$. Therefore, it separates $K\sm\{p\}$ from $U_\xi$ within $P_{n+k}(z)$. 

The conclusions now follow: the only point in $K\cap \partial U_\xi$ is $p=z$, so $z$ is fixed by $g^k$. Therefore the fiber of $z$, which is $K$, intersects $U_\xi$ only in $\{z\}$. This shows that $U_\xi$ is locally connected at $z$.. 
\end{proof}

\begin{proof}[Proof of Corollary~\ref{Cor:LC}]

Consider a Newton map $N_p$ that satisfies the hypothesis of the corollary. By \cite{Pr}, the Julia set $J(N_p)$ is connected. By a well known classical condition \cite{Why, MiIntro}, it is locally connected if and only if every complementary component (i.e.\ every Fatou component) has locally connected boundary, and moreover for every $\eps>0$, there are only finitely many components with spherical diameters exceeding $\eps$. 

By Sullivan's no-wandering-domain theorem, every Fatou component is eventually periodic, and it is either a component of the basin of some root, or a component of the attracting or parabolic basin of a non-repelling periodic point of period at least $2$, or a component of a Siegel disk, again of period at least $2$. By Corollary~\ref{Cor:BasinsLocConn}, components of basins of roots have locally connected boundaries. The other types of components are eventually mapped to bounded Fatou components of a polynomial-like restriction $f$ of $N_p$ (Theorem~\ref{Thm:RRP}). By \cite{RY08}, all the components in the basins of attracting or parabolic periodic points of $f$ have locally connected boundaries. On the other hand, all the Siegel disks of $f$ (if any), as well as their iterated preimages, have locally connected boundary by hypothesis ($f$ straightens to a polynomial in ${\Sspace}$). This establishes the first condition for local connectivity of $J(N_p)$, so it remains to show that there are only finitely many Fatou components of $N_p$ with spherical diameters exceeding any $\eps>0$.

Assuming the contrary, let $(X_i)_{i \ge 0}$ be a sequence of Fatou components with 
\begin{equation}
\label{Eq:DiamAssum}
\diam(X_i) \ge \eps > 0 \text{ for every } i \ge 0.
\end{equation}

Let $x \in J(N_p)$ be an accumulation point of $(X_i)$. Without loss of generality we can assume that $N_p$ is attracting-critically-finite; after passing to the iterate $g$ of $N_p$, we can use the puzzle construction from Section~\ref{Sec:Puzzles}. By Theorem~\ref{Thm:RRP}, we have either $\fib(x) = \{x\}$, or the orbit of $x$ belongs to or eventually lands in the filled Julia set of a suitable polynomial-like restriction of $g$. Let us consider these two cases separately.

In the first case there exists a nest $(P_n(x))_{n \ge 0}$ of puzzle pieces or unions of puzzle pieces (depending on whether $x$ is $\infty$, a (pre-)pole or not) with $\diam(P_n(x)) \to 0$ as $n \to \infty$, so there is an $N$ so that $\diam(P_n(x))<\eps$ for all $n\ge N$. 
The boundary of each individual $P_n(x)$ passes through finitely many touching components of the basins of roots. Therefore, infinitely many $X_i$ must be contained in $P_n(x)$, and this is a contradiction. 

In the second case, we can assume that $\fib(x)$ is the little filled Julia set of some polynomial-like map $f:= g^s \colon \inter P_{n+s}(x) \to \inter P_n(x)$ (Lemma~\ref{Lem:NonRepellingRenormalizable}). Up to a subsequence, we can further assume that all $X_i$ either lie in $\Cc \sm \fib(x)$, or in $\fib(x)$. 

Suppose first that all $X_i\subset\fib(x)$. By hypothesis, the polynomial-like map $f$ straightens to a polynomial in ${\Sspace}$, so most of its Fatou components are small, and most of the Fatou components of $f$ are small as well (homeomorphisms on compacts are uniformly continuous). This gives a contradiction to~\eqref{Eq:DiamAssum}.

Finally, suppose all $X_i\subset \Cc\sm\fib(x)$. The map $f \colon \inter P_{n+s}(x) \sm \fib(x) \to \inter P_n(x) \sm \fib(x)$ is a local isometry with respect to the hyperbolic metric in the domain and in the range. Since $\inter P_{n+s}(x) \sm \fib(x) \subset \inter P_n(x) \sm \fib(x)$, the restriction $f|_{\inter P_{n+s}(x) \sm \fib(x)}$ is weakly expanding with respect to the hyperbolic metric in $\inter P_n(x) \sm \fib(x)$. Let us use this weak expansivity to complete the proof.

Assume first, up to passing to a subsequence, that each $X_i$ is disjoint from the boundary of $P_{n+ks}(x)$ for every $k \ge 1$. In terms of the hyperbolic metric in $\inter P_n(x) \sm \fib(x)$ assumption~\eqref{Eq:DiamAssum} means that the hyperbolic diameters of $X_i$ are unbounded (because $\fib(x)$ is a part of the ideal boundary and $(X_i)$ accumulates on $\fib(x)$). Choose an $i_0$ such that the hyperbolic diameter of $X_{i_0}$ is larger than the hyperbolic diameter of the closed annulus $P_{n+s}(x) \sm \inter P_{n+2s}(x)$. If $X_{i_0}'$ is the iterated image of $X_{i_0}$ under $f$ that lies in the annulus, then by expansivity, the hyperbolic diameter of $X_{i_0}'$ is not smaller than the hyperbolic diameter of $X_{i_0}$, a contradiction to our choices.

If there are no subsequences of $(X_i)$ specified in the previous paragraph, then all but finitely many of $X_i$ intersect the boundaries of $P_{n+ks}(x)$, and hence all but finitely many $X_i$ lie in the basins of roots. In this case, since $\partial P_{n+s}(x)$ intersects only finitely many components of root basins, there exists another subsequence in $(X_i)$ consisting of $f$-iterated preimages of a component of a root basin that intersects $\partial P_{n+s}(x)$. Again by weak expansivity (or, equivalently, by weak contraction for the corresponding inverse branches of $f$), the hyperbolic diameters of the elements this subsequence are uniformly bounded above. But they accumulate on $\fib(x)$, the ideal boundary of $\inter P_n(x) \sm \fib(x)$ in its hyperbolic metric, which implies that their spherical diameters must go to zero. This again contradicts~\eqref{Eq:DiamAssum} and concludes the proof. 
\end{proof}

\Newpage

\section{Proof of parameter rigidity for Newton maps (Theorem ~\ref{Thm:QCRigidity})}
\label{Sec:ProofParamRigidity}

In this section we prove parameter rigidity for Newton maps (Theorem~\ref{Thm:QCRigidity}). This will be accomplished by combining the rigidity results of Kozlovski--van Strien, as described in~\cite{DKvS}, together with our results from Section~\ref{Sec:RRPProof}.

Our strategy is to decompose the Newton dynamics into a box mapping that captures the non-renormalizable part of the Newton dynamics (done in Subsection~\ref{SSec:NonRenBM}), as well as some renormalizable parts that will be treated separately, and so that what is left has measure zero. We then show in Subsection~\ref{SSec:ExtrCombEquiv} that for combinatorially equivalent Newton maps the corresponding non-renormalizable box mappings are combinatorially equivalent as well, and hence we can apply the result of Kozlovski--van Strien (spelled out in Subsection~\ref{SSec:QCBM}). In Subsections~\ref{SSec:QCRig} and~\ref{SSec:General}, we will piece together a global conjugation between combinatorially equivalent Newton maps that are renormalizable in the same way (defined in Subsection~\ref{Sub:CombinatoriallyEquivalent}), and show that the only way this can fail to be (parameter) rigid is coming from the renormalizable parts, i.e.\ from the dynamics of non-rigid polynomials. This will be done modulo a key proposition (Proposition~\ref{Prop:KeyClaim}), which we then prove is Subsection~\ref{SSec:ProofKeyClaim}.

\subsection{Constructing non-renormalizable box mappings for Newton maps}
\label{SSec:NonRenBM}

We start this section by showing how to upgrade the result of Lemma~\ref{Lem:NewtonBox} so that the resulting box mapping is non-renormalizable. 

\begin{lemma}[Extracting non-renormalizable box mappings]
\label{Lem:ExtrNonRenormBoxMappings}
Let $g$ be the iterate of an attracting-critically-finite Newton map as defined in Section~\ref{Sec:Puzzles}. Then there exists a dynamically natural complex box mapping $F \colon \U \to \V$ with the following properties: 
\begin{enumerate}
\item
\label{It:Extr1}
the components of\/ $\U$ and $\V$ are open puzzle pieces of $g$, and $F$ is the restriction of $g$ to the components of\/ $\U$;
\item
\label{It:Extr21}
every $g$-orbit that accumulates at a point in $\V$ also intersects $K(F)$;
\item
\label{It:Extr22}
$\Crit(F) \subset \CritNF(g) \sm \Crit_{er}(g) \subset \V$; 
\item
\label{It:Extr3}
$F$ is non-renormalizable with all periodic points repelling.
\item
\label{It:Extr23}
the set $\CritNF(g) \sm (\Crit(F) \cup \Crit_{er}(g))$ consists of critical points $c$ such that $c \not\in \Crit_{er}(g)$ and $\omega(c)$ is disjoint from $\CritNF(g) \sm \Crit_{er}(g)$;
\end{enumerate}
\end{lemma}

\begin{remark}
In Lemma~\ref{Lem:ExtrNonRenormBoxMappings}~\eqref{It:Extr23}, $\omega(c)$ is disjoint from $\CritNF(g) \sm \Crit_{er}(g)$ if and only if it does not intersect fibers of points in $\CritNF(g) \sm \Crit_{er}(g)$ because the latter fibers are trivial by Theorem~\ref{Thm:RRP}.
\end{remark}

\begin{proof}
Let $F' \colon \U' \to \V'$ be the dynamically natural complex box mapping constructed in Lemma~\ref{Lem:NewtonBox} for the map $g$; it satisfies $\Crit(F') \subset \CritNF(g) \subset \V'$ and the set of non-escaping critical points of $F'$ contains $\Crit^a(g)$. 
Let $\V$ be the union of all components of $\U'$ intersecting $\CritNF(g) \sm \Crit_{er}(g)$. Consider the first return map $F \colon \U \to \V$ for $F'$ to $\V$. Since each critical component of $\U'$ contains exactly one critical fiber (Lemma~\ref{Lem:NewtonBox}~\eqref{It:ExactlyOne}), the orbits of the eventually renormalizable critical points do not intersect $\V$. Therefore, $F$ is a complex box mapping with the property that $\Crit(F) \subset \CritNF(g) \sm \Crit_{er}(g)$  (see Lemma~\ref{Lem:FirstReturnConstruction} \eqref{It:FRM3}); by Lemma~\ref{Lem:FirstReturnConstruction}~\eqref{It:FRM4}, every $g$-orbit that accumulates at a point in $\V$ also intersects $K(F)$. 

By construction, $F \colon \U \to \V$ is a non-renormalizable box mapping with all periodic points repelling. Since $F'$ is dynamically natural, the same is true for $F$ (compare the end of the proof of Lemma~\ref{Lem:NewtonBox}). 

Finally, property~\eqref{It:Extr23} follows from~\eqref{It:Extr21} and~\eqref{It:Extr22}.
\end{proof}

The dynamics on the Julia set of the box mapping $F$ constructed in Lemma~\ref{Lem:ExtrNonRenormBoxMappings}, together with the dynamics on all renormalizable fibers of $g$, describes the behavior of almost all orbits in the Julia set of the Newton map $g$. This is made precise in the following lemma.

\begin{lemma}[Almost nothing is left]
\label{Lem:NothingLeft}
If $F \colon \U \to \V$ is the dynamically natural box mapping constructed in Lemma~\ref{Lem:ExtrNonRenormBoxMappings} for $g$, then the set of points in $J(g)$ whose orbits are disjoint from $J(F)$ as well as from all renormalizable fibers of $g$ has Lebesgue measure zero. 
\end{lemma}

\begin{proof}
The set of points $z$ so that $\omega(z)$ avoids the fibers of all $c\in \CritNF(g)$ has measure zero by Lemma~\ref{Lem:ZeroArea}~\eqref{It:ZeroArea}. We may thus exclude these and consider the set $A$ of points $z\in J(g)$ so that $\orb(z)$ is disjoint from $J(F)$ and from all renormalizable fibers of $g$, and so that $\omega(z)$ intersects the fibers of some $c\in \CritNF(g)$. 

If $\orb(z)$ accumulates on the fiber of some $c \in \CritNF(g) \sm \Crit_{er}(g)$, then by Lemma~\ref{Lem:ExtrNonRenormBoxMappings}~\eqref{It:Extr21} it accumulates on the fiber of a point in $\Crit(F)$, so $\orb(z)$ intersects $J(F)$, contrary to assumption. Therefore, for all $z\in A$ there is a $c\in\Crit_{er}(g)$ the fiber of which intersects $\omega(z)$ but not $\orb(z)$.

By Lemma~\ref{Lem:FibersCompContained}, all fibers are contained in the interior of each of its puzzle pieces, so Lemma~\ref{Lem:AccumAtPeriod}, and hence Corollary~\ref{Cor:AccumAtPeriodDeg} apply. 
Therefore, there is an increasing sequence of integers $(\nu_j)_{j \ge 0}$ and a pair of nested puzzle pieces $\inter P_s \subset P_0$ (possibly after an index shift) so that $g^{\nu_j}(z) \in \inter P_s$ for every $j \ge 0$ and  $g^{\nu_j} \colon \inter P_{\nu_j}(z) \to \inter P_0$ are branched coverings of uniformly bounded degrees (independent of $j$). Let $w \in P_s$ be an accumulation point of $(g^{\nu_j}(z))$, and let $B := B(w, \eps)$ be an open round disk that intersects $\partial P_s$ so that $\ovl B\subset \inter P_0$ (such a disk exists because $\inter P_0 \sm P_s$ is a non-degenerate annulus). 

Since, by the Newton puzzle construction, $\partial P_s \cap J(g)$ is a finite set of poles or prepoles, while the remaining part of $\partial P_s$ lies in the basins of roots, there is a round sub-disk $B' \subset B$ fully lying in one of the root basins and compactly contained in $B$. We can now transfer this ``hole'' $B'$ in the Julia set at fixed scale to ever-smaller scales with bounded distortion, so as to apply the Lebesgue Density Theorem. 

The construction of $w$ gives us a subsequence $(s_j)$ of $(\nu_j)$ so that $g^{s_j}(z) \in B$. Let $D_j \subset \inter P_{s_j}(z)$ be the pullback of $B$ under $g^{s_j}$ containing $z$, and $D_j' \subset D_j$ be a corresponding pullback of $B'$. From Lemma~\ref{Lem:AccumAtPeriod} it follows that $\diam D_j \to 0$ as $j \to \infty$. Since the degrees of the maps $g^{s_j} \colon \inter P_{s_j}(z) \to \inter P_0$ are uniformly bounded, $\ovl B \subset \inter P_0$, and $\ovl B' \subset B$, the disks $D_j$ and $D_j'$ have uniformly bounded shapes and uniformly comparable diameters (Lemmas~\ref{Lem:Fact} and \ref{Lem:Diam}). As in the proof of Lemma~\ref{Lem:ZeroArea}~\eqref{It:ZeroArea}, we conclude that $z$ is not a Lebesgue density point of $J(g)$. Therefore, $A$ has zero Lebesgue measure. 
\end{proof}

\subsection{Quasiconformal rigidity of complex box mappings}
\label{SSec:QCBM}

In this subsection we introduce a result from \cite{DKvS} on quasiconformal rigidity of combinatorially equivalent complex box mappings, including the required notation.

For a box mapping $F$, define $\PC(F) := \left\{F^n(c) \colon c \in \Crit(F), n \ge 0 \right\}$ to be the union of the forward orbits of all critical points of $F$ (the critical and postcritical set). 
Since for $F \colon \U \to \V$ the components of the sets $\U$ and $\V$ are Jordan disks, by the Carath\'eodory theorem there is a well-defined continuous extension $\ovl F \colon \ovl \U \to \ovl \V$ to the boundary. 

\begin{definition}[Itinerary of puzzle pieces relative to curve family]
\label{Def:CurveFamily}
Let $F \colon \U \to \V$ be a dynamically natural complex box mapping, and $X\subset \partial \V$ be a finite set with one point on each component of $\partial \V$. Let $\Gamma$ be a collection of simple curves in $\V \sm (\U \cup \PC(F))$, one for each $y \in \ovl F^{-1}(X)$, that connects $y$ to a point in $X$.  Then for every $n \ge 0$ and for each component $P$ of $F^{-n}(\U)$ there exists a simple curve connecting $\partial P$ to $X$ of the form $\gamma_0\ldots\gamma_n$ where $F^k(\gamma_k) \in \Gamma$. The word $(\gamma_0, F(\gamma_1), \ldots, F^n(\gamma_n))$ is called the \emph{$\Gamma$-itinerary} of $P$.
\end{definition}

Every orbit can intersect $\V\sm\U$ at most once, so $\PC(F)\cap (\V\sm \U)$ is finite. Therefore, the existence of the curves in $\Gamma$ is clear, but of course there is a choice involved. Hence the $\Gamma$-itinerary of $P$ is in general not uniquely defined (there is a unique finite word for every $y'\in \ovl F^{-n}(x)\cap \partial P$). However, different components of $F^{-n}(\U)$ have different $\Gamma$-itineraries.

\begin{definition}[Combinatorial equivalence of box mappings]
\label{Def:CombEquivBoxMappings}
Let $F \colon \U \to \V$ and $\tilde F \colon \tilde \U \to \tilde \V$ be two dynamically natural complex box mappings.
The maps $F$ and $\tilde F$ are called \emph{combinatorially equivalent with respect to a homeomorphism $H \colon \V \to \tilde \V$} if
\begin{enumerate}
\item
$H(\U)=\tilde \U$;
\item
$H$ is a bijection between the critical sets of $F$ and $\tilde F$;
\item
$H(\PC(F)\sm\U) = \PC(\tilde F) \sm \tilde \U$;
\item
there exists a curve family $\Gamma$ as in Definition~\ref{Def:CurveFamily} so that each critical point $c \in \Crit(F)$ is mapped by $H$ to a critical point $\tilde c \in \Crit(\tilde F)$ with the property that for every integer $k \ge 0$ and $n \ge 0$, the $\Gamma$-itineraries of the puzzle piece $P_n(F^k(c))$ coincide with the $\tilde \Gamma$-itineraries of $\tilde P_n(\tilde F^k(\tilde c))$ for $\tilde \Gamma = H(\Gamma)$.
\end{enumerate}
\end{definition}

Note that every critical point of $F$ may have several itineraries, and they should all coincide with the itineraries of the corresponding critical point of $\tilde F$. Roughly speaking, this equivalence says that there is a correspondence between components of the box mappings that is preserved by the dynamics (not in the  sense of conjugation), that  this  correspondence respects critical points and the postcritical set, and that within components containing critical points the  choice of preimages is respected. We do not require that  $H$ is a conjugation between the postcritical sets: the definition of combinatorial equivalence is not a dynamical condition, and upgrading it to a dynamical condition is one of the major goals that is included in the next theorem, which is the main parameter rigidity result for non-renormalizable box mappings (see \cite{DKvS}).

\begin{theorem}[Rigidity of complex box mappings]
\label{Thm:Linefield}
\label{Thm:QC}
Let $F\colon \U\to \V$ be a dynamically natural complex box mapping that is not renormalizable and for which all periodic points are repelling. Then 
\begin{enumerate}
\item 
\label{It:Linefield}
$F$ carries no measurable invariant line field on $J(F)$.
\item 
\label{It:Rig}
Suppose $\tilde F \colon \tilde \U \to \tilde \V$ is another dynamically natural complex box mapping for which there exists a quasiconformal homeomorphism $H \colon \V \to \tilde\V$ so that
\begin{enumerate}
\item
$H(\U) = \tilde \U$,
\item
$\tilde F \circ H = H \circ F$ on $\partial \U$,
\item
$\tilde F$ is combinatorially equivalent to $F$ w.r.t\ $H$.
\end{enumerate}
Moreover, assume that the boundary of each component of\/ $\U$, $\tilde \U$, $\V$, and $\tilde \V$ consists of piecewise smooth arcs. Then $F$ and $\tilde F$ are quasiconformally conjugate, and this conjugation can be chosen to agree with $H$ on $\V\setminus \U$.     \qed
\end{enumerate}
\end{theorem}

By definition, the Julia set $J(F)$ carries a measurable invariant line field if there is an $F$-invariant measurable Beltrami differential supported on $J(F)$ (see \cite[$\S 3.5$]{McM}).

\subsection{Combinatorially equivalent Newton maps}
\label{Sub:CombinatoriallyEquivalent}

Let $N_p \colon \Cc \to \Cc$ be an attracting-critically-finite Newton map and let $\Delta_n$ be the Newton graph of level $n \ge 0$ for $N_p$. Let $\LL>0$ be the smallest level such that $\Delta_{\LL}$ contains all the critical points that eventually land on the Newton graph (either at $\infty$ or at a root), as well as all poles of $N_p$; such a level exists by Theorem~\ref{Thm:MagicLevels}.
Similarly, let $N_{\tilde p} \colon \Cc \to \Cc$ be another attracting-critically-finite Newton map with Newton graph $\tilde \Delta_n$ at level $n\ge 0$, and level $\tilde \LL$ analogous to $\LL$.

\begin{definition}[Combinatorial equivalence of Newton maps]
\label{Def:CombEquiv}
We call two attracting-critically-finite Newton maps $N_p$ and $N_{\tilde p}$ \emph{combinatorially equivalent} if 
\begin{enumerate}
\item
\label{It:CEItem1}
$\LL = \tilde \LL$, and the Newton maps restricted to $\Delta_{\LL}$ and $\tilde \Delta_{\LL}$ are topologically conjugate, respecting vertices;
\item
\label{It:CEItem2}
there is a bijection between the critical points of $N_p$ on $\Cc\sm \Delta_{\LL}$ and of $N_{\tilde p}$ on $\Cc\sm \tilde \Delta_{\LL}$ that respects degrees and itineraries with respect to (complementary components of) the Newton graphs. 
\end{enumerate}

Two Newton maps (not necessarily attracting-critically-finite) are combinatorially equivalent if the quasiconformal surgery in the basins of roots (described in Section~\ref{Sec:Puzzles}) turns them into combinatorially equivalent (attracting-critically-finite) Newton maps. This is clearly independent of all the choices in the surgery.
\end{definition}

\begin{remark}
Let us make several comments on the conditions in Definition~\ref{Def:CombEquiv}, phrased for convenience for attracting-critically-finite Newton maps. 
The first condition in the definition above means that there exists a graph homeomorphism, say $\phi \colon \Delta_{\LL} \to \tilde \Delta_{\LL}$ that sends vertices to vertices and edges to edges (preserving their cyclic order locally around vertices), and such that the diagram
\[
{\includegraphics[trim=200 608 200 128, clip]{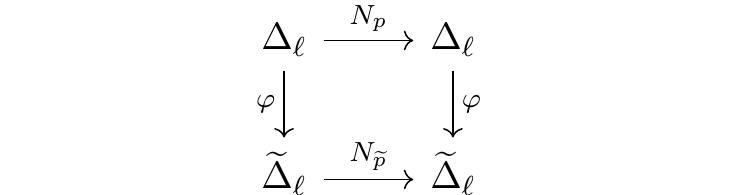}}
\]
is commutative. It follows that $\phi$ maps fixed points to fixed points  and hence the channel diagram $\Delta$ of $N_p$ edge-wise onto the channel diagram $\tilde \Delta$ of $N_{\tilde p}$. Since $\infty$ is the unique fixed point that is connected to all other fixed points by the channel diagram, it follows that $\phi (\infty) = \infty$, and hence roots are mapped to roots respecting their cyclic order around $\infty$, and pre-fixed points of order $l \ge 1$ are mapped to pre-fixed points of order $l$. 

In particular, the map $\phi$ provides a bijection between the eventually fixed critical points and preserves their local degrees. Combined with property \eqref{It:CEItem2} of Definition~\ref{Def:CombEquiv}, this means that for every pair of combinatorially equivalent (attracting-critically-finite) Newton maps there is a degree-preserving bijection between their critical sets. 
Therefore, combinatorially equivalent Newton maps have equal degrees. 
\end{remark}

The notion of combinatorial equivalence for Newton maps gives rise to the notion of \emph{corresponding puzzle pieces respecting boundary marking} as follows. 

Since Newton graphs are defined by pulling back the channel diagrams, the homeomorphism $\phi\colon\Delta_{\LL}\to\tilde\Delta_{\LL}$ extends to a homeomorphism $\phi_n \colon \Delta_n \to \tilde \Delta_n$ for all $n\ge \LL$ using $\phi_n = N_{\tilde p}^{-(n- \LL)} \circ \phi \circ N_p^{n-\LL}$ for appropriate choices of inverse branches (where $\phi_{\LL} = \phi$). This induces a bijection between the components of $\Cc \sm \Delta_n$ and $\Cc \sm \tilde \Delta_n$: we will call a pair of components  \emph{corresponding} if their boundaries consist of homeomorphic subsets of $\Delta_n$ and $\tilde\Delta_n$. The notion of corresponding components extends to the notion of {corresponding puzzle pieces} for the iterates $g$ and $\tilde g$ of $N_p$, resp.\ $N_{\tilde p}$ (see Section~\ref{Sec:Puzzles}). 

We want to upgrade this bijection to a homeomorphism between corresponding boundary pieces with good properties. To begin with, since there is a bijection between the roots of $p$ and $\tilde p$ respecting the degrees of the Newton maps in the immediate basins, there is a biholomorphic conjugation $\psi$ between the immediate basins of $N_p$ and $N_{\tilde p}$. This conjugation extends to the entire basins, and hence induces a correspondence of equipotentials there.

Consider a puzzle piece $P_s$ be of depth $s \ge 0$ defined for $g$, and let $\tilde P_s$ be the corresponding puzzle piece for $\tilde g$. By construction (see Section~\ref{Sec:Puzzles}), the boundary of $P_s$ consists of finitely many edges of $\Delta_k$ some $k = k(s)$, cropped and connected by arcs of equipotentials; those two types of boundary pieces alternate along $\partial P_s$. The same is true for $\tilde P_s$ with the same choice of $k$ and equipotentials. We say that a homeomorphism $h \colon \inter P_s \to \inter {\tilde P}_s$ \emph{respects the boundary marking} if $h$ extends to a continuous map $\ovl h \colon P_s \to \tilde P_s$, and this extension agrees with the maps $\phi_k$ and $\psi$ on the alternating boundary pieces of $P_s$ (along edges respectively equipotentials). 

We can now describe the good properties of the homeomorphisms between corresponding puzzle pieces. This is \cite[Lemma 5.3]{KSS} transferred almost verbatim to our setting, so we can omit the proof.

\begin{lemma}[Initial QC maps respecting boundary marking]
\label{Lem:AnyQCBdMarking}
For every pair of corresponding Newton puzzle pieces $P$ and $\tilde P$, there exists a quasiconformal homeomorphism $\phi \colon P \to \tilde P$ that respects the boundary marking. \qed
\end{lemma}

\begin{definition}[Renormalizable in the same way]
\label{Def:SameWay}
Two combinatorially equivalent Newton maps $g$ and $\tilde g$ are \emph{renormalizable in the same way} if for every pair of corresponding renormalizable critical points $c \in \Crit(g)$ and $\tilde c \in \Crit(\tilde g)$ there exists a pair of corresponding puzzle pieces $P_n(c)  \supset P_{n+k}(c)\ni c $ and $\tilde P_n(\tilde c)\supset  \tilde P_{n+k}(\tilde c)\ni\tilde c $ and a homeomorphism $h \colon \inter P_n(c) \to \inter{\tilde P}_n(\tilde c)$ mapping $\inter P_{n+k}(c)$ onto $\inter{\tilde P}_{n+k}(\tilde c)$ such that:
\begin{enumerate}
\item
the restrictions $g^k \colon \inter P_{n+k}(c) \to \inter P_n(c)$ and $\tilde g^k \colon \inter{\tilde P}_{n+k}(\tilde c) \to \inter{\tilde P}_n(\tilde c)$ are polynomial-like mappings with connected Julia sets, and $k$ is the period of renormalization; 
\item
$h$ is a hybrid equivalence between $g^k$ and $\tilde g^k$ on $\inter P_{n}(c)$; 
\item
\label{It:BdMarking}
$h$ respects the boundary marking;
\item
\label{It:PrePeriodic}
moreover, if $c_0 \in \Crit(g)$ is a non-renormalizable critical point so that $g^s(c_0) \in \fib(c)$ for the minimal such $s$, then the hybrid conjugacy $h$ lifts to a quasiconformal homeomorphism $h_0 \colon \inter P_{n+s}(c_0) \to \inter{\tilde P}_{n+s}(\tilde c_0)$ that respects the boundary marking and such that its dilatation vanishes on $\fib(c_0)$.  
\end{enumerate}
\end{definition}

\begin{remark}
Condition~\eqref{It:BdMarking} means that $h$ respects the combinatorial positions of little filled Julia sets corresponding to renormalizations within the puzzle partition of the dynamical plane of the Newton map.

Condition~\eqref{It:PrePeriodic} ensures that the critical fibers that are not renormalizable but are mapped to renormalizable ones also respect the hybrid conjugacy coming from the polynomial-like restrictions, and hence the combinatorial position within the Newton puzzle.
\end{remark}

A critical point $c \in \Crit(g)$ is called \emph{eventually renormalizable} if it is either renormalizable, or is mapped be some finite iterate to a renormalizable fiber. We write $\Crit_{er}(g)$ for the set of all such critical points. Thus $c \in \Crit_{er}(g)$ if $c$ falls into case (\hyperref[It:R]{R}) of Theorem~\ref{Thm:RRP}.

\subsection{Combinatorially equivalent box mappings from Newton maps}
\label{SSec:ExtrCombEquiv}

In this subsection, we consider two combinatorially equivalent Newton maps $g$ and $\tilde g$ (not necessarily  renormalizable the same way). 

The first lemma, which is one of the key ingredients towards the proof of rigidity, allows us to ``spread'' certain partially defined quasiconformal homeomorphisms over the whole Riemann sphere so as to produce a partial conjugation with nice properties. 

\begin{lemma}[The Spreading Principle]
\label{Lem:Spreading}
Let $g$ and $\tilde g$ be two combinatorially equivalent Newton maps. 
Let $U$ be a finite union of open puzzle pieces for $g$ such that $U$ is nice and $\CritNF(g) \subset U$, and let $\tilde U$ be the union of the corresponding pieces for $\tilde g$. Let $\phi \colon U \to \tilde U$ be a $K$-quasiconformal homeomorphism that respects the boundary marking. Then $\phi$ extends to a $K$-quasiconformal homeomorphism $\Phi \colon \Cc \to \Cc$ such that:
\begin{enumerate}
\item
\label{It:SP1}
$\Phi = \phi$ on $U$;
\item
\label{It:SP3}
for each $z \not\in U$
\[
\tilde g \circ \Phi(z) = \Phi \circ g(z)\;; 
\]
\item
\label{It:SP2}
$\Phi$ maps every open puzzle piece $V$ which is not contained in $\DomL(U)$ onto the corresponding puzzle piece $\tilde V$, and the homeomorphism $\Phi \colon V \to \tilde V$ respects the boundary marking;
\item
\label{It:SP4}
the dilatation of $\Phi$ vanishes on $\Cc \sm \DomL(U)$.
\end{enumerate}
\end{lemma}

\begin{remark} 
It follows that the homeomorphism $\Phi$ has the following additional property:
\begin{enumerate}
\addtocounter{enumi}{4}
\item
\label{It:SP6}
\emph{for every component\/ $V$ of $\DomL(U) \sm U$, we have
$\Phi|_{V} = \tilde g^{-k}\circ \phi \circ g^k|_{V}$, where $g^k|_V$ is the restriction to $V$ of the first landing map $L \colon \DomL(U) \to U$.}
\end{enumerate}
In fact, claim~\eqref{It:SP6} follows for each $z\in \DomL(U)$ by iterating~\eqref{It:SP3} until $z$ reaches $U$. 
\end{remark}

\begin{proof}
The proof goes verbatim as in \cite[Section 5.3]{KSS}, based on Lemmas~\ref{Lem:ZeroArea} and \ref{Lem:AnyQCBdMarking}.
\end{proof}

The second lemma shows that for combinatorially equivalent Newton maps the box mappings from Lemma~\ref{Lem:ExtrNonRenormBoxMappings} can be chosen to be combinatorially equivalent as well.

\begin{lemma}[Extracting combinatorially equivalent box mappings]
\label{Lem:ExtrCombEquivBoxMappings}
If $g$ and $\tilde g$ are two combinatorially equivalent Newton maps, then in Lemma~\ref{Lem:ExtrNonRenormBoxMappings} one can choose two dynamically natural complex box mappings $F \colon \U \to \V$ and $\tilde F \colon \tilde \U \to \tilde \V$, and so that they are combinatorially equivalent with respect to some quasiconformal homeomorphism $H \colon \V \to \tilde \V$ such that $H(\U) = \tilde \U$ and $\tilde F \circ H = H \circ F$ on $\partial \U$.
\end{lemma}

\begin{proof}
Let $F \colon \U \to \V$ be a box mapping given by Lemma~\ref{Lem:ExtrNonRenormBoxMappings}. If in that lemma we carry out the same construction for $\tilde g$ as we did for $g$ by picking the same iterates along the way, we obtain a non-renormalizable box mapping $\tilde F \colon \tilde \U \to \tilde \V$ for which all periodic points are repelling and that is dynamically natural. Since $g$ and $\tilde g$ are combinatorially equivalent, the sets $\U$ and $\tilde \U$, as well as $\V$ and $\tilde \V$, consist of corresponding puzzle pieces of $g$ and $\tilde g$, and the restrictions of $F$ and $\tilde F$ to the corresponding components of $\U$ and $\tilde \U$ are the same iterates of $g$ and $\tilde g$. Let us show that $F$ and $\tilde F$ obtained this way are combinatorially equivalent (in the sense of Definition~\ref{Def:CombEquivBoxMappings}) with respect to some quasiconformal homeomorphism $H \colon \V \to \tilde \V$. We apply the standard pull-back argument. 

For a pair $V \subset \ovl\V$, $\tilde V \subset \ovl{\tilde \V}$ of corresponding puzzle pieces, let $h \colon V \to \tilde V$ be some quasiconformal homeomorphism that respects the boundary marking (Lemma~\ref{Lem:AnyQCBdMarking}). Modify $h$ on $V$ so that: 1) $h$ is unchanged on $\partial V$; 2) $h$ maps the critical values of $F$ in $V$ to the corresponding critical values of $\tilde F$ in $\tilde V$; 3) $h$ maps every point in $\PC(F) \sm \U$ that happens to lie in $V$ to the corresponding point in $\tilde V$. The modification is possible because it requires to change the map at finitely many points while preserving the map on the boundary. Define a quasiconformal homeomorphism $H_0 \colon \V \to \tilde \V$ by setting $H_0|_V = h$ on each such component $V$. Since $H_0$ maps the critical values of $F$ to the corresponding critical values of $\tilde F$, we can lift $H_0$ to a quasiconformal map $H_1 \colon \U \to \tilde \U$ that also respects the boundary marking. Finally, construct a map $H \colon \V \to \tilde \V$ by setting 
\begin{equation*}
H(z) := \left\{
\begin{aligned}
&H_0(z), \text{ for }z \in \V \sm \U\;;\\
&H_1(z), \text{ for }z \in \U.
\end{aligned}
\right.
\end{equation*}
Defined in this way, $H$ is a quasiconformal homeomorphism between $\V$ and $\tilde \V$ such that $H(\U) = \tilde \U$, $H(\PC(F) \sm \U) = \PC(\tilde F) \sm \tilde \U$, $\tilde F \circ H = H \circ F$ on $\partial \U$, and $H$ respects the boundary marking. The last fact combined with the observation that $F$ resp.\ $\tilde F$ restricted to the components of $\U$ resp.\ $\tilde \U$ yield equal iterates of $g$ resp.\ $\tilde g$ implies that $H$ provides a combinatorial equivalence between $F$ and $\tilde F$ in the sense of curve itineraries in Definition~\ref{Def:CombEquivBoxMappings}. 
\end{proof}

\Newpage

\subsection{Proof of Theorem~\ref{Thm:QCRigidity}: the attracting-critically-finite case}
\label{SSec:QCRig}

We can restate the theorem as follows, using the notation developed so far: if $N_p$ and $N_{\tilde p}$ are two attracting-critically-finite Newton maps that are combinatorially equivalent and renormalizable the same way, then they are affinely conjugate. 

By Theorem~\ref{Thm:NewtonPuzzles}, there exists an iterate $M$ so that $g:=N_p^M$ has a well-defined Markov partition; choose $M$ to be the minimal with this property. By combinatorial equivalence, the same holds for $\tilde g:=N_{\tilde p}^M$ with the same iterate $M$; then every combinatorial property of the orbits for $g$, i.e.\ defined in terms of puzzle itineraries, immediately transfers to the corresponding property for $\tilde g$. 

For convenience, let us choose a depth $\LL_0$ so that
\begin{enumerate}
\item
$\LL_0 \ge \LL$, where $\LL$ is defined before Definition~\ref{Def:CombEquiv};
\item
any two critical points of $g$ that are not in the same fiber are in different puzzle pieces of depth $\LL_0$.
\end{enumerate}
The depth $\LL_0$ has the same properties for $\tilde g$. 

Since $g$ and $\tilde g$ are renormalizable in the same way, we can pull back the conjugating homeomorphisms from Definition~\ref{Def:SameWay} as follows. Define
\[
\mathcal O:= \bigcup_{c\in \Crit_{er}(c)} \bigcup_{i\ge 0}g^i(\fib(c))
\;.
\]
This is really a finite union of fibers because all renormalizable critical points have periodic fibers: every renormalizable $c$ has a $k=k(c)\ge 1$ so that $\fib(c)=g^{k}(\fib(c))$, and every non-renormalizable $c_0\in\Crit_{er}(g)$  has a minimal $s=s(c_0)$ so that $g^s(\fib(c_0))=g^{k_0}(c)$ for some renormalizable $c=c(c_0)$ and some $k_0=k_0(c_0)$. For each $n\ge \LL_0$ define the following open puzzle neighborhood of $\mathcal O$:
\begin{equation}
\label{Eq:Omega}
\Omega_n:=\bigcup_{c}\left( \bigcup_{i=0}^{k-1}\inter P_{n+k-i}\left(g^i(c)\right) \right)
\cup
\bigcup_{c_0}\left( \bigcup_{j=0}^{s-1}\inter P_{n+k-k_0+s-j}\left(g^j(c_0)\right) \right)
\;,
\end{equation}
where the first union is over all renormalizable critical points and the second one is over the others, and where we wrote $k$, $s$, and $k_0$ instead of $k(c)$, $s(c)$, and $k_0(c_0)$ for simplicity (in the last union, $k=k(c)$, where $c=c(c_0)$). The depths $n$ in this definition are large enough so that all $\Omega_n$  are nice sets. All puzzle pieces in $\Omega_n$ have depth $n+1$ or more. Of course we have analogous sets $\tilde{\mathcal O}$ and $\tilde\Omega_n$.

The homeomorphisms from Definition~\ref{Def:SameWay} extend by pull-backs to  $K$-quasiconformal homeomorphism $\psi_n \colon \Omega_n \to \tilde \Omega_n$ such that: 1) $K$ is independent of $n$; 2) $\psi_n(\mathcal O) = \tilde{\mathcal O}$; 3) $\psi_n$ respects the boundary marking; 4) $\psi_n$ conjugates $g$ and $\tilde g$ on $\mathcal O$; 5) the dilatation of $\psi_n$ vanishes on $\mathcal O$. By construction, for $m>n$ the map $\psi_m$ is just the restriction of $\psi_n$ to the smaller set $\Omega_m$. In particular, the restriction $\psi_n|_{\mathcal O}$ is independent of $n$. 

The proof of Theorem~\ref{Thm:QCRigidity} is based on the following key proposition, which is in analogy to the main claim in \cite[Section 6.4]{KSS}.

\begin{proposition}[Uniform control of dilatation]
\label{Prop:KeyClaim}
There exist $K > 0$ and $n_0 \ge \LL_0$ so that for every $c \in \CritNF(g) \sm \Crit_{er}(g)$ and $n \ge n_0$ there exists a $K$-quasiconformal homeomorphism
\[
\phi_{n,c} \colon \inter P_{n}(c) \to \inter{\tilde P}_{n}(\tilde c)
\]
that respects the boundary marking. 
\end{proposition}

\begin{remark}
In Proposition~\ref{Prop:KeyClaim}, by our choice of $\LL_0$ all critical points in each $P_n(c)$ are contained in $\fib(c)$. However, the map $\phi_{n,c}$ does not depend on a choice of a critical point in $\fib(c)$ because each $c \in \CritNF(g) \sm \Crit_{er}(g)$ has trivial fiber by Theorem~\ref{Thm:RRP}.
\end{remark}

Conceptually, Proposition~\ref{Prop:KeyClaim} says that critical puzzle pieces do not degenerate as their depths increase, akin to ``a priori bounds'' for polynomial renormalization. 
We postpone the proof to Subsection~\ref{SSec:ProofKeyClaim} and use it first, together with the Spreading Principle (Lemma~\ref{Lem:Spreading}), to complete the proof of Theorem~\ref{Thm:QCRigidity} in the attracting-critically-finite case.

\begin{proof}[Proof of Theorem~\ref{Thm:QCRigidity}]

\setcounter{stepctr}{0}

The proof is done in three steps. First we construct a quasiconformal homeomorphism conjugating $g$ and $\tilde g$. Second, we show that this conjugation is affine. Finally, we conclude that $N_p$ and $N_{\tilde p}$ are also affine conjugate.

\begin{step}[a quasiconformal conjugation] 
Take a depth $n\ge n_0$, where $n_0$ is given by Proposition~\ref{Prop:KeyClaim}. Set $U := \Omega_{n+1} \cup \bigcup_{c \in \CritNF(g) \sm \Crit_{er}(g)} \inter P_{n}(c)$. This construction is such that $U$ is a nice open set containing $\CritNF(g)$. Let $\tilde U$ be the corresponding set for $\tilde g$. 

Define a homeomorphism $\phi \colon U \to \tilde U$ by setting $\phi|_{\Omega_{n+1}} := \psi_{n+1}$ and $\phi|_{\inter P_n(c)} := \phi_{n,c}$, where the latter map is given by Proposition~\ref{Prop:KeyClaim}. The map $\phi$ defined this way is a $K$-quasiconformal homeomorphism that respects the boundary marking. Therefore, we can apply the Spreading Principle (Lemma~\ref{Lem:Spreading}): it guarantees that $\phi$ extends to a $K$-quasiconformal homeomorphism $\Phi_n \colon \Cc \to \Cc$ that conjugates $g$ and $\tilde g$ on $\Cc \sm U$. Since all $\Phi_n$ have dilatation uniformly bounded by $K$, the sequence $(\Phi_n)$ has a convergent subsequence. A limiting map is a quasiconformal homeomorphism $\Phi_\infty \colon \Cc \to \Cc$ that conjugates $g$ and $\tilde g$ on $\Cc \sm (\mathcal O \cup \bigcup_{c \in \CritNF(g) \sm \Crit_{er}(g)} \fib(c))$. But on $\mathcal O$ the homeomorphism $\Phi_\infty$ coincides with $\psi_n |_{\mathcal O}$, and hence is a conjugacy there as well. Each of the finitely many critical fibers $\fib(c)$ for $c \in \CritNF(g) \sm \Crit_{er}(g)$ are trivial by Theorem~\ref{Thm:RRP}. Therefore, $\Phi_\infty$ extends to a global quasiconformal conjugacy between $g$ and $\tilde g$. Finally, any two limiting maps coincide on an everywhere dense open set, and hence must be equal.
\end{step}

\begin{step}[the conjugation $\Phi_\infty$ is affine]
By construction and the Spreading Principle, the dilatation of $\Phi_\infty$ vanishes on the set $\Cc \sm \bigcup_{c \in \CritNF(g), \, s \ge 0} g^{-s}(\fib(c))$. However, on the full backward orbit of a renormalizable fiber $\fib(c)$ the map $\Phi_{\infty}$ coincides with $\psi_{n}|_{\mathcal O}$ or its conformal lifts. Therefore, the dilatation of $\Phi_\infty$ vanishes on $\bigcup_{c \in \Crit_{er}(g), s\ge 0}g^{-s}(\fib(c))$. 

In order to conclude that $\Phi_\infty \colon \Cc \to \Cc$ is conformal, and thus affine, it suffices to show that the set of points in $J(g)$ that do not land in one of the renormalizable fibers under the iteration does not support a measurable invariant line field.

Let $A \subset J(g)$ be the set of points whose orbits do not land in one of the renormalizable fibers. Let $F \colon \U \to \V$ be the box mapping constructed in Lemma~\ref{Lem:ExtrNonRenormBoxMappings} for $g$. By Theorem~\ref{Thm:QC}~\eqref{It:Linefield}, the set of points in $A$ whose orbits land in $J(F)$ does not support a measurable invariant line fields, and all other points in $A$ have zero Lebesgue measure by Lemma~\ref{Lem:NothingLeft}.
\end{step}

\begin{step}[the Newton maps are affine conjugate] 
We still need to upgrade the affine conjugation from the iterates $g$ and $\tilde g$ to the Newton maps $N_p$ and $N_{\tilde p}$ themselves. Up to a M\"obius conjugation, we may assume that $g$ and $\tilde g$ have the same Fatou set, and hence the same unbounded Fatou components, with the same centers that are super-attracting fixed points. But then $N_p$ and $N_{\tilde p}$ also have identical Fatou set and identical immediate basins with identical fixed points, which are the roots of $p$ and $\tilde p$. Since these roots must be simple, we have $p=\tilde p$ up to a constant multiple, hence $N_p=N_{\tilde p}$ (up to M\"obius conjugation). 
\end{step}
\end{proof}

\subsection{Proof of Theorem~\ref{Thm:QCRigidity}: general case}
\label{SSec:General}

If $N_p$ and $N_{\tilde p}$ are not attracting-critically-finite, then we can perform the surgery in the basins of roots described in Subsection~\ref{SSec:ConstructionOfNewtonPuzzles} to turn them into attracting-critically-finite maps. By the previous arguments, these new maps are affine conjugate on the Riemann sphere. Hence, before the surgery the maps were quasiconformally conjugate in some neighborhood of the Julia set, and this neighborhood includes all the components of the Fatou set that are not in the basins of roots. The dilatation of this conjugation vanishes on those components as well as on the entire Julia set. This concludes the proof of Theorem~\ref{Thm:QCRigidity}. 
\qed

\Newpage

\subsection{Proof of Proposition~\ref{Prop:KeyClaim}}
\label{SSec:ProofKeyClaim}

In many cases, Proposition~\ref{Prop:KeyClaim} follows from quasiconformal rigidity of the box mappings constructed in Lemma~\ref{Lem:ExtrCombEquivBoxMappings}. However, these mappings do not capture all the critical points of the original maps. The main part of the proof below will be dealing with the ``run-away'' points that are not eventually renormalizable. 

Let $F \colon \U \to \V$ and $\tilde F \colon \tilde \U \to \tilde \V$ be the box mappings constructed in Lemma~\ref{Lem:ExtrCombEquivBoxMappings} for $g$ and $\tilde g$. By Lemma~\ref{Lem:ExtrNonRenormBoxMappings}~\eqref{It:Extr23}, $\Crit(F)$ is disjoint from $\Crit_{er}(g)$. The set $S:= \CritNF(g) \sm (\Crit(F) \cup \Crit_{er}(g))$ consists of critical points not in $\Crit_{er}(g)$ that do not accumulate on $\CritNF(g) \sm \Crit_{er}(g)$. Define the corresponding set $\tilde S$ for $\tilde g$. 

The depths of fibers have been chosen large enough so that every critical puzzle piece contains a single critical fiber. In view of Theorem~\ref{Thm:RRP}, the only case when this fiber can fail to be trivial, and hence contain more than one critical point, is when the fiber is eventually renormalizable, so all its critical points are in $\Crit_{er}(g)$. We are not dealing with this case in Proposition~\ref{Prop:KeyClaim}.

\subsubsection{Strategy of the proof}
\label{Sub:ProofPropHardCase}

In Lemma~\ref{Lem:Lem68} and Proposition \ref{Prop:Cor63} we will construct certain nice and arbitrary deep puzzle neighborhoods for points in $S$; these neighborhoods will be constructed with good control over the geometry of first landing domains to them. After that we extend these neighborhoods to arbitrary deep nice puzzle neighborhoods of the entire critical set of $g$ by adding deep enough critical puzzle pieces of the box mapping $F$ and renormalization domains containing $\Crit_{er}(g)$. These extended neighborhoods will also come with quasiconformal maps to the corresponding sets for $\tilde g$. Using the fact that these maps respect the boundary marking, they will be globalized by the Spreading Principle (Lemma~\ref{Lem:Spreading}). The dilatations of these globalized maps, however, will be uniformly controlled everywhere except on the first landing domains to the neighborhoods of $S$ (for puzzle pieces landing in the neighborhood of $\Crit(F)$ we will have uniform control by Theorem~\ref{Thm:QC}, and for puzzle pieces first landing in the neighborhood of $\Crit_{er}(g)$ this control will be given by the hypothesis on being renormalizable in the same way). The result of Proposition~\ref{Prop:Cor63} will allow us to use the QC-Criterion (Theorem~\ref{Thm:QCCr} below) to improve, in a uniform way, the ``uncontrolled'' dilatation caused by the landing domains to the neighborhoods of $S$. The required uniformly quasiconformal maps in Proposition~\ref{Prop:KeyClaim} will be then constructed as restriction of the improved maps to the required depth.

\begin{remark}
If $S = \emptyset$, then the proof of Proposition~\ref{Prop:KeyClaim} follows immediately by Theorem~\ref{Thm:QC} and the Spreading Principle.
\end{remark}

\subsubsection{Constructing puzzle neighborhoods of $S$ with uniformly bounded geometry and moduli bounds}

The following lemma gives the control over shapes of the critical puzzle pieces for points in $S$. 

\begin{lemma}[Geometric control for puzzle neighborhood of $S$]
\label{Lem:Lem68}
For every $c \in S$ there is a constant $\eta>0$ and for every $\eps > 0$ there exists a puzzle piece $W$ with $c \in \inter W$ such that $\diam W < \varepsilon$ and $W$ has $\eta$-bounded geometry. Moreover, the statement remains true if we replace $W$ by $\tilde W$ and $c$ by $\tilde c$.
\end{lemma}

Recall that an open topological disk $U$ in $\C$ has \emph{$\eta$-bounded geometry} if it contains a round Euclidean disk of radius $\eta \cdot \diam U$.

\begin{proof}[Proof of Lemma~\ref{Lem:Lem68}]
By definition of $S$, every critical point $c \in S$ is non-recurrent. Therefore, by Theorem~\ref{Thm:RRP}, the fiber of $c$ is trivial, and hence the critical puzzle pieces around $c$ shrink in diameter. Let us show that we can pick a shrinking nest of such puzzle pieces with uniformly bounded geometry. 	

If $\omega(c)$ is disjoint from critical fibers of all points in $\CritNF(g)$, then the claim follows by the Koebe Distortion Theorem. Otherwise, there exists at least one critical point $c' \in \Crit_{er}(g)$ such that $\orb(c)$ accumulates at but does not land in $\fib(c')$. In this situation, by Corollary~\ref{Cor:AccumAtPeriodDeg} there exist an increasing sequence of integers $(\nu_j)_{j \ge 0}$ and a pair of nested puzzle pieces $\inter P_s \subset P_0$ so that $g^{\nu_j}(c) \in \inter P_s$ and $g^{\nu_j} \colon \inter P_{\nu_j}(c) \to \inter P_0$ is a branched covering of uniformly bounded degree. The claim now follows by Lemma~\ref{Lem:Fact} applied to $V' := \inter P_0$, $V := \inter P_s$, $f := g^{\nu_j}$.
\end{proof}

For a subset $A$ of $\CritNF(g)$, an open set $V$ containing $A$ is a \emph{nice open puzzle neighborhood of $A$} if $V$ is nice and each component of $V$ is an open puzzle piece intersecting $A$. 

In what follows, we will need to distinguish puzzle pieces of $g$ from those of the box mapping $F$ (and similarly for $\tilde g$ and $\tilde F$). For this, we will say \emph{an $F$-puzzle piece of $F$-depth $n$} for the piece of depth $n$ view as a puzzle piece of $F$; from the point of view of the map $g$ this is a puzzle piece of depth (\emph{$g$-depth}) at least $n$.

\begin{proposition}[Good landing domains to puzzle neighborhood of $S$]
\label{Prop:Cor63}
There exists $\eta >0$ and an integer $n_0 \ge \LL_0$ such that for every depth $n \ge n_0$ there exists a nice open puzzle neighborhood $W$ of $\CritNF(g) \sm \Crit_{er}(g)$ with the following properties:
\begin{enumerate}
\item
\label{It:S1}
The depths of the components of $W$ are larger than $n$.
\item
\label{It:S2}
The components of $W$ intersecting $\Crit(F)$ are $F$-puzzle pieces of the same $F$-depth.
\item
\label{It:S3}
Let $L \colon \DomL(W) \to W$ be the first landing map to $W$ under $g$. If $W_S$ is the union of all components of $W$ intersecting $S$, then for every component $U$ of $\DomL(W)$ with $L(U) \subset W_S$ the following hold:
\begin{itemize}
\item
$U$ has $\eta$-bounded geometry;
\item
if $P$ is a puzzle piece of depth $n$ containing $U$, then $\modulus\left(\inter P \sm \ovl U\right) \ge \eta$.
\end{itemize}
\end{enumerate}
Moreover, the same statements hold true if we replace all the objects for the corresponding objects with tilde.
\end{proposition}

\begin{proof}
Let us start by choosing $n_0 \ge \LL_0$ such that for every $c \in S$ the following holds:
\begin{enumerate}
\renewcommand{\labelenumi}{(\roman{enumi})}
\renewcommand{\theenumi}{\roman{enumi}}
\item
\label{It:Dis1}
for every $c' \in S \cup \Crit(F)$ the orbit of $F(c)$ is disjoint from $P_{n_0}(c')$;
\item
\label{It:Dis2}
for every $c'' \in \Crit_{er}(g)$ the orbit of $c''$ is disjoint from $P_{n_0}(c)$.
\end{enumerate}

The existence of $n_0$ satisfying \eqref{It:Dis1} follows by definition of $S$ and Lemma~\ref{Lem:ExtrNonRenormBoxMappings} \eqref{It:Extr23}: the orbits of points from $S$ do not accumulate on $\CritNF(g) \sm \Crit_{er}(g) = S \cup \Crit(F)$. Property~\eqref{It:Dis2} can be easily satisfied since $\fib(c'')$ has a (pre-)periodic itinerary. 

Let $n \ge n_0$. Since $n_0 \ge \LL_0$, every critical puzzle piece of depth $n$ of a point in $\CritNF(g) \sm \Crit_{er}(g)$ contains a single fiber and this fiber is equal to the critical point. By Lemma~\ref{Lem:Lem68}, for every $c \in S$ there exists a constant $\eta > 0$ and arbitrary small open puzzle piece $W_c$ containing $c$ and having $\eta$-bounded geometry. By choosing $W_c$ smaller if necessary, we can find an open puzzle piece $W_c' \ni c$ such that $W_c \subset W_c' \subset P_n(c)$ and $\modulus(W_c' \sm \ovl{W_c}) \ge \eta$. Since $S$ is finite, we can assume that $\eta$ is the same for all $c \in S$. Define
\[
W_S := \bigcup_{c \in S} W_c.
\]    

Choose $\kappa = \kappa(n)$ to be an integer such that all critical $F$-puzzle pieces of $F$-depth equal to $\kappa$ are of $g$-depth at least $n$. We write $P^F_\kappa(c)$ for such pieces. Define 
\[
W := W_S \cup \bigcup_{c \in \Crit(F)} \inter P^F_\kappa(c).
\]
So defined, $W$ is a nice open set. Indeed, the only thing we need to check is that $g^m(\partial W_c) \cap \inter P^F_\kappa(c') = \emptyset$ for all $m \ge 1$, $c \in S$ and $c' \in \Crit(F)$. This follows by assumption~\eqref{It:Dis1} on $n_0$. Hence $W$ is a nice open puzzle neighborhood of $\CritNF(g) \sm \Crit_{er}(g)$. By construction, it satisfies properties \eqref{It:S1} and \eqref{It:S2}. Let us check property \eqref{It:S3}.

Let $L \colon \DomL(W) \to W$ be the first landing map to $W$ under $g$ , and let $g^s \colon U \to W_c$, $c \in S$ be a branch of this map with the range in $W_S$. Define $U'$ to be the component of $g^{-s}(W_c')$ containing $U$. We claim that the degree of the map $g^s \colon U' \to W_c'$ is uniformly bounded independently of the chosen branch.

Indeed, suppose there exists $0 \le i < s$ such that $g^i(U')$ contains a critical point. This critical point, say $c'$, necessarily lies in $\Crit(F)$ because it cannot lie in $\Crit_{er}(g)$ by assumption \eqref{It:Dis2} on $n_0$, and it cannot lie in $S$ by assumption \eqref{It:Dis1} and the first landing property. Since $U$ is a component of the domain of the first landing map to $W$ under $g$, it follows that the depth of $g^i(U')$ must be not larger than the depth of $P^F_\kappa(c')$. And since there are only finitely many puzzle piece of depth smaller than the depths of the puzzle pieces in $\ovl{W \sm W_S}$, the degree of the map $g^s \colon U' \to W_c'$ is uniformly bounded above independently of the choice of the branch (compare the proof in Step~\ref{St:BoxProp} of Lemma~\ref{Lem:NewtonBox}). The claim now follows by Lemmas~~\ref{Lem:FirstTime} and \ref{Lem:Fact}: there exists $\eta' >0$ (depending only $\eta$ and the degrees of the critical points of $g$) such that $U$ have $\eta'$-bounded geometry and $\modulus(U' \sm \ovl U) \ge \eta'$. As the depth of $\ovl W'_c$ is at least $n$, the open puzzle piece $U'$ is contained in some puzzle piece $P$ of depth $n$ and $\modulus(\inter P \sm \ovl U) \ge \modulus (U' \sm \ovl U) \ge \eta'$.

The claim for the objects with tilde follows by just repeating the arguments above. 
\end{proof}

Finally, we will need the following QC-Criterion from \cite[Appendix 1]{KSS}.

\begin{theorem}[QC-Criterion]
\label{Thm:QCCr}
For any constants $0 \le K < 1$ and $\eta > 0$ there exists a constant $K'$ with the following property. Let $\psi \colon \Omega \to \tilde \Omega$ be a quasiconformal homeomorphism between two Jordan disks. Let $X$ be a subset of $\Omega$ consisting of pairwise disjoint topological disks. Assume that the following holds:

\begin{enumerate}
\item
If $V$ is a component of $X$, then both $V$ and $\psi(V)$ have $\eta$-bounded geometry, and moreover
\[
\modulus(\Omega \sm V) \ge \eta, \quad \modulus(\tilde \Omega \sm \psi(V)) \ge \eta;
\]
\item
$|\ovl \partial \psi| \le K |\partial \psi|$ holds almost everywhere on $\Omega \sm X$.
\end{enumerate}
Then there exits a $K'$-quasiconformal map $\Psi \colon \Omega \to \tilde \Omega$ such that $\Psi = \psi$ on $\partial \Omega$. \qed
\end{theorem}

\subsubsection{Proof of Proposition~\ref{Prop:KeyClaim}}

Let $n_0$ be given by Proposition~\ref{Prop:Cor63}. For every $n \ge n_0$, construct an open puzzle neighborhood $W$ of $\CritNF(g) \sm \Crit_{er}(g)$ using that proposition. The set $W$ can be decomposed as
\[
W = W_S \cup \bigcup_{c \in \Crit(F)} \inter P^F_\kappa(c).
\]
Let $l$ be the maximal depth of the puzzle pieces in $W$; define
\[
U := W_S \cup \bigcup_{c \in \Crit(F)} \inter P^F_\kappa(c) \cup \bigcup_{c \in \Crit_{er}(g)} \inter P_l(c).
\]

By construction, $U$ is a nice open puzzle neighborhood of $\CritNF(g)$ consisting of open puzzle pieces of depths at least $n+1$.

Let us now define a quasiconformal map $\phi \colon U \to \tilde U$ that respects the boundary marking, specifying the map on each component in the decomposition for $U$ as follows. 

On each $\inter P_l(c)$ we define $\phi|_{\inter P_l(c)} := \psi_{m}|_{\inter P_l(c)}$, where $\psi_m$ is the map defined before Proposition~\ref{Prop:KeyClaim} and $m = m(l)$ is chosen so that $\inter P_l(c) \subset \Omega_m$.

On each $\inter P_{\kappa}^F(c)$ we will define the map using Theorem~\ref{Thm:QC}. By Lemma~\ref{Lem:ExtrCombEquivBoxMappings}, the box mappings $F$ and $\tilde F$ are combinatorially equivalent with respect to some quasiconformal homeomorphism $H$. Hence, by Theorem~\ref{Thm:QC}~\eqref{It:Rig}, there exists a quasiconformal map $\Psi \colon \V \to \tilde \V$ that respects the boundary marking, maps $\U$ onto $\tilde \U$, and is a conjugation between $F$ and $\tilde F$ on $\U$. We define $\phi|_{\inter P_{\kappa}^F(c)} := \Psi|_{\inter P_{\kappa}^F(c)}$.

The maps defined so far are $K$-quasiconformal and respect the boundary marking, with $K$ independent of $n$. 

On the remaining components in the decomposition, that is on the components of $W_S$, define $\phi$ by means of Lemma~\ref{Lem:AnyQCBdMarking}. We do not control the dilatation of $\phi$ on these components.

Since $U$ is a nice open set containing $\CritNF(g)$ and $\phi$ respects the boundary marking, we can spread this map around using Lemma~\ref{Lem:Spreading}. In this way we obtain a quasiconformal homeomorphism $\Phi \colon \Cc \to \Cc$ such that for every $c \in \CritNF(g) \sm \Crit_{er}(g)$ the restriction $\Phi \colon \inter P_n(c) \to \inter{\tilde P}_n(\tilde c)$ has the following properties:
\begin{itemize}
\item
$\Phi$ respects the boundary marking (by Lemma~\ref{Lem:Spreading} \eqref{It:SP2} using the fact that $\inter P_n(c)$ does not lie in $\DomL(U)$ as the depths of the components of $\DomL(U)$ are at least $n+1$);
\item
the dilatation of $\Phi$ vanishes on $\inter P_n(c) \sm \DomL(U)$ (by Lemma~\ref{Lem:Spreading} \eqref{It:SP4});
\item
$\Phi(V) = \tilde V$ for each component $V$ of $\DomL(U) \cap \inter P_n(c)$, and $\Phi|_{V}$ is either equal to $\phi|_{V}$, or is the corresponding lift of $\phi$ (by claims \eqref{It:SP3} and \eqref{It:SP6} of Lemma~\ref{Lem:Spreading}).
\end{itemize}

Let us now apply Theorem~\ref{Thm:QCCr}: set $\Omega := \inter P_n(c)$, $\tilde \Omega := \inter{\tilde P}_n(\tilde c)$, $\psi := \Phi$, and furthermore, let $X := \inter P_n(c) \cap \DomL^*(U)$, where $\DomL^*(U)$ is the union of the components $V$ of the domain $\DomL(U)$ of the first landing map $L \colon \DomL(U) \to U$ such that $L(V) \subset W_S$. Defined this way, the dilation of $\psi$ is bounded by $K$ outside of $X$, and hence the second assumption in Theorem~\ref{Thm:QCCr} is satisfied. To see that the first assumption is satisfied, note that every component $V$ of $X$ is a component of $\DomL(W)$, and thus, by Proposition~\ref{Prop:Cor63}, both $V$ and $\tilde V$ have $\eta$-bounded geometry and 
\[
\modulus\left(\inter P_n(c) \sm \ovl V\right) \ge \eta, \quad \modulus\left(\inter{\tilde P}_n(\tilde c) \sm \ovl{\tilde V}\right) \ge \eta
\]
for some constant $\eta > 0$ independent of $n$. 

Since both assumptions in Theorem~\ref{Thm:QCCr} are satisfied, we conclude that $\Phi \colon \partial P_n(c) \to \partial \tilde P_n(\tilde c)$ extends to a $K'$-quasiconformal map between $\inter P_n(c)$ and $\inter{\tilde P}_n(\tilde c)$, with $K'$ independent of $n$. This is the desired map $\phi_{n,c}$ in Proposition~\ref{Prop:KeyClaim}. \qed



\end{document}